\documentclass[12pt]{amsart}
\usepackage[top=1in, bottom=1in, left=1in, right=1in]{geometry}
\usepackage{amssymb,mathrsfs,amsmath}
\usepackage{booktabs}

\usepackage{graphicx,color}

\usepackage{longtable}

\theoremstyle{plain}
\newtheorem{theorem}{Theorem}[section]
\newtheorem{lemma}[theorem]{Lemma}
\newtheorem{prop}[theorem]{Proposition}
\newtheorem{cor}[theorem]{Corollary}

\theoremstyle{remark}
\newtheorem{rk}{Remark}[section]

\newcommand{\abs}[1]{\left\lvert#1 \right\rvert}

\newcommand {\N} {{\mathbb N}}
\newcommand {\R} {{\mathbb R}}

\newcommand {\Z} {{\mathbb Z}}
\newcommand {\D} {{\mathbb D}}
\newcommand {\C} {{\mathbb C}}
\newcommand {\E} {{\mathbb E}}
\newcommand {\U} {{\mathbb U}}

\newcommand {\CA} {{\mathcal A}}

\newcommand {\CC} {{\mathcal C}}

\newcommand {\CG} {{\mathcal G}}

\newcommand{\CP} {{\mathcal P}}

\newcommand {\CX} {{\mathcal X}}
\newcommand {\CZ} {{\mathcal Z}}

\newcommand{\ds}{\displaystyle}

\newcommand{\eq}[2]{ \begin{equation} \label{#1}\begin{split} #2 \end{split} \end{equation} }

\newcommand{\als}[1]{\begin{align*} #1 \end{align*} }

\newcommand{\bs}\boldsymbol{}
\DeclareMathOperator{\prob}{{\bf Prob}}

\renewcommand{\pmod}[1]{\mod{#1}}
\renewcommand{\Re}{\textrm{Re}}

\newcommand{\fl}[1]{\left\lfloor#1\right\rfloor}
\newcommand{\leg}[2]{\left( \frac{#1}{#2} \right) }

\newcommand{\nin}{\noindent}
\numberwithin{equation}{section}

\renewcommand{\bar}[1]{\overline{#1}}

\renewcommand{\mod}[1]{{\ifmmode\text{\rm\,(mod\,$#1$)}\else\discretionary{}{}{\hbox{ }}\rm(mod~$#1$)\fi}}

\newcommand\ignore[1]{}

\definecolor{red}{rgb}{1,0,0}

\definecolor{orange}{rgb}{0.7,0.3,0}

\definecolor{blue}{rgb}{.2,.6,.75}

\definecolor{green}{rgb}{.4,.7,.4}

\begin{document}

\title{The frequency and the structure of large character sums}

\author[J. Bober]{Jonathan Bober}
\address{JB: Heilbronn Institute for Mathematical Research \\ 
School of Mathematics \\
University of Bristol \\
Howard House \\
Queens Avenue \\
Bristol BS8 1SN \\
United Kingdom}
\email{{\tt j.bober@bristol.ac.uk}}

\author[L. Goldmakher]{Leo Goldmakher}
\address{LG: Department of Mathematics and Statistics \\
Bronfman Science Center \\
Williams College \\
18 Hoxsey St \\
Williamstown, MA 01267 \\
USA}
\email{{\tt leo.goldmakher@williams.edu}}

\author[A. Granville]{Andrew Granville}
\address{AG: D\'epartement de math\'ematiques et de statistique\\
Universit\'e de Montr\'eal\\
CP 6128 succ. Centre-Ville\\
Montr\'eal, QC H3C 3J7\\
Canada}
\email{{\tt andrew@dms.umontreal.ca}}

\author[D. Koukoulopoulos]{Dimitris Koukoulopoulos}
\address{DK: D\'epartement de math\'ematiques et de statistique\\
Universit\'e de Montr\'eal\\
CP 6128 succ. Centre-Ville\\
Montr\'eal, QC H3C 3J7\\
Canada}
\email{{\tt koukoulo@dms.umontreal.ca}}

\subjclass[2010]{Primary: 11N60. Secondary: 11K41, 11L40}
\keywords{Distribution of character sums, distribution of Dirichlet $L$-functions, pretentious multiplicative functions, random multiplicative functions}

\date{\today}

\begin{abstract}
Let $M(\chi)$ denote the maximum of $|\sum_{n\le N}\chi(n)|$ for a given non-principal Dirichlet character $\chi \pmod q$, and let $N_\chi$ denote a point at which  the maximum is attained. In this article we study the distribution of  $M(\chi)/\sqrt{q}$ as one varies over characters $\pmod q$, where $q$ is prime, and investigate the location of $N_\chi$. We show that the distribution of  $M(\chi)/\sqrt{q}$ converges weakly to a universal distribution $\Phi$, uniformly throughout most of the possible range, and get (doubly exponential decay) estimates for $\Phi$'s tail. Almost all $\chi$ for which $M(\chi)$ is large are odd characters that are $1$-pretentious. Now, $M(\chi)\ge |\sum_{n\le q/2}\chi(n)| = \frac{|2-\chi(2)|}\pi \sqrt{q} |L(1,\chi)|$, and one knows how often the latter expression is large, which has been how earlier lower bounds on $\Phi$ were mostly proved. We show, though, that for most $\chi$ with $M(\chi)$ large, $N_\chi$ is bounded away from $q/2$, and the value of $M(\chi)$ is little bit larger than $\frac{\sqrt{q}}{\pi} |L(1,\chi)|$.
 

\end{abstract}

\maketitle

\setcounter{tocdepth}{1}
\tableofcontents

\section{Introduction}

For a given non-principal Dirichlet character $\chi \mod{q}$, where $q$ is an odd prime, let
\[
M(\chi) := 
\max_{1 \le x \le q} \left| \sum_{n \le x} \chi(n) \right| .
\]
This quantity plays a fundamental role in many areas of number theory, from modular arithmetic to $L$-functions. Our goal in this paper is to understand how often $M(\chi)$ is large, and to gain insight into the structure of those characters $\chi\mod q$ for which $M(\chi)$ is large. 

It makes sense to renormalize $M(\chi)$ by defining 
\[
m(\chi) = \frac{M(\chi)}{e^\gamma\sqrt{q}/\pi},
\]
and we believe that 
\eq{eq:ConjBound}{
   m(\chi)  \le (1+ o_{q\to\infty}(1))\log\log q,
}
where $\gamma$ is the Euler--Mascheroni constant. R\'enyi \cite{Ren} observed that $m(\chi)\ge c+ o_{q\to\infty}(1)$ with $c=e^{-\gamma}\pi /\sqrt{12} = 0.509\dots$

Upper bounds on $M(\chi)$ (and hence on $m(\chi)$) have a rich history. The 1919 P\'{o}lya--Vinogradov Theorem states that
\[
m(\chi) \ll   \log q 
\]
for all non-principal characters $\chi\mod{q}$. Apart from some improvements on the implicit constant \cite{hildebrand88,GS07}, this remains the state-of-the-art for the general non-principal character, and any improvement of this bound would have immediate consequences for other number theoretic questions (see e.g. \cite{BG14}). Montgomery and Vaughan \cite{MV77} (as improved in \cite{GS07}) have shown that the Generalized Riemann Hypothesis implies
\[
m(\chi) \le	 (2 + o_{q \to \infty}(1) )  \log\log  q ;
\]
whereas, for every prime $q$, there are characters $\chi\mod{q}$ for which 
\[
m(\chi) \ge	 (1 + o_{q \to \infty}(1) )  \log\log  q ,
\]
(see \cite{BC, GS07}, which improve on Paley \cite{Pal}),
so the conjectured upper bound \eqref{eq:ConjBound} is the best one could hope for. 
However, for the vast majority of the characters $\chi\mod q$, $M(\chi)$ is somewhat smaller, and so we study the distribution function
\[
\Phi_q(\tau) := \frac{1}{\varphi(q)} \# \left\{ \chi\mod q: m(\chi)>   \tau \right\}.
\]

Montgomery and Vaughan \cite{MV79} showed that $\Phi_q(\tau)\ll_C \tau^{-C}$ for all $\tau\ge1$ for any fixed $C\ge1$ (they equivalently phrase this in terms of the moments of $M(\chi)$). This  was recently improved by Bober and Goldmakher \cite{BG}, who proved that, for $\tau$ fixed and as $q \to \infty$ through the primes,
\eq{BGthm}{
 \exp\left\{- \frac{e^{\tau+A}}{\tau} \left(1+O(1/\sqrt{\tau})\right)  \right\} \le \Phi_q(\tau)
 	\le \exp\left\{ - Be^{\sqrt{\tau}/(\log\tau)^{1/4}} \right\}   ,
}
where $B$ is some positive constant and
\eq{A}{
A = \log2 - 1 - \int_0^2\frac{\log I_0(t)}{t^2} dt - \int_2^\infty (\log I_0(t) -  t) \frac{dt}{t^2} = 0.088546\dots,
}
and $I_0(t)$ is the modified Bessel function of the first kind given by
\eq{bessel}{
I_0(t) = \sum_{n \ge 0} \frac{(t^2/4)^n}{n!^2} .
}
We improve upon these results and demonstrate that $\Phi_q(\tau)$ decays in a double exponential fashion:


\begin{theorem}\label{main thm} Let $\eta=e^{-\gamma}\log2$. If $q$ is a prime number, and $1\le \tau\le \log\log  q-M$ for some $M\ge4$, then
\[
\exp\left\{ - \frac{e^{\tau +A - \eta }}{\tau}  (1+O(\epsilon_1))  \right\} 
	\le \Phi_q(\tau)  \le \exp\left\{ - \frac{e^{\tau -2 - \eta }}{\tau}  (1+O(\epsilon_2))\right\}  ,
\]
where
\[
\epsilon_1= (\log\tau)^2 / \sqrt{\tau} +e^{-M/2} 
\quad\text{and}\quad
\epsilon_2=(\log\tau) /\tau
\]
\end{theorem}


\nin The proof implies that $1+o_{\tau\to\infty}(1)$ of the characters for which $m(\chi)>\tau$ are odd, so it makes sense to consider odd and even characters separately. Hence we define
\eq{phi odd even}{
\Phi_q^{\pm}(\tau) = \frac 1{\phi(q)/2} \  \#\left\{ \chi\mod q: \chi(-1) = \pm1 \textrm{ and } m(\chi) >  \tau  \right\} ,
}
We outline the proof of Theorem \ref{main thm} in Section \ref{outline}, and fill in the details in subsequent sections. The range of uniformity barely misses implying the upper bound in \eqref{eq:ConjBound}; even so, it does show that $m(\chi)$ is rarely large. Calculations reveal that $\Phi_q, \Phi_q^+, \Phi_q^-$ each tend to a universal distribution function $\Phi, \Phi^+, \Phi^-$, and we will show this for $\Phi_q$ later on. In Figure \ref{figure-phiq-plot} we graph $\Phi_q(\tau)$ for a typical $q$:


\begin{figure}[!htbp]
    \includegraphics[scale=.65]{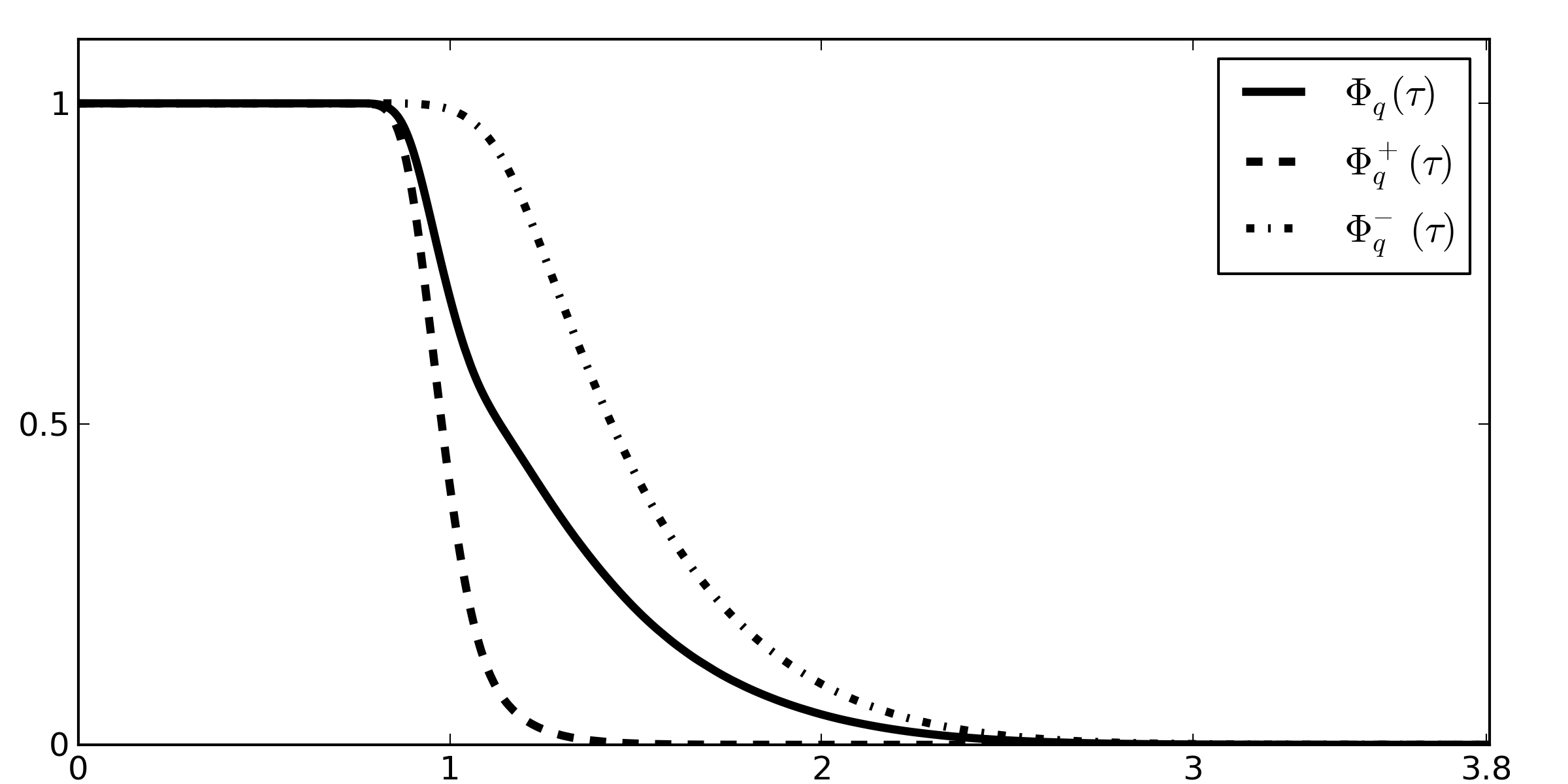}
    \caption{$\Phi_q$,
        $\Phi_q^+$,
        and
        $\Phi_q^-$
        for $q = 12000017$ 
    }\label{figure-phiq-plot}
\end{figure}


Notice that $\Phi_q(\tau) = 1$ for all $\tau\le .7227$, and then it decays
quickly: $\Phi_q(1) \approx .697, \Phi_q(2) \approx 0.0474$ and
$\Phi_q(3) \approx .000538$. For more computational data, the reader is invited to consult Table \ref{minmax} in Section \ref{moredata}.

The distribution of $|L(1,\chi)|$, for $q$ prime, decays similarly \cite{GS06}:
\eq{GS L(1,chi)}{
\Phi_q^{L}(\tau):= \frac{1}{\phi(q)}
  \#\left\{ \chi\mod q: |L(1, \chi)|> e^\gamma \tau \right\}  
   = \exp\left\{- \frac{e^{\tau+A}}{\tau} \left(1+ O(\epsilon_3)\right) \right\} 
}
with $\epsilon_3=\tau^{-1/2}+e^{-M/2}$, uniformly for $1\le\tau\le\log\log q-M$, $M\ge1$. This similarity is no accident, since for an odd, non-primitive character $\chi\mod q$, the average of the character sum is
\[
\mathbb E_{1\le N\le q} \  
   \left(  \sum_{n\le N} \chi(n) \right) = \frac{\CG(\chi)}{i\pi} L(1,\bar\chi) ,
\]
where $\CG(\chi)$ is the Gauss sum, so that $|\CG(\chi)|=\sqrt{q}$. In particular, if $L(1,\bar\chi)$ is large, then so is $m(\chi)$. Moreover, for $\chi$ as above, we have the pointwise formula
\[
 \sum_{n\le q/2}\chi(n)  = (2 - \bar{\chi}(2))\frac{\CG(\chi)}{i\pi}  L(1,\bar\chi)  ,
 \]
which implies that
\[
m(\chi) \ge 2e^{-\gamma}  \abs{ \prod_{p>2} \left(1-\frac{\chi(p)}{p}\right)^{-1} }, 
\]
a little larger than $|L(1,\chi)|$ if $|1-\chi(2)|\gg1$. The distribution of $2|\prod_{p>2}(1-\chi(p)/p)^{-1}|$ can be analyzed in the same way as the distribution of $|L(1,\chi)|$.
However, even if $\chi(2)=1$, then we can show that the average of our character sum up to $N$ is slightly larger than $|L(1,\chi)|$ when $N$ is close to $q/2$. This builds on ideas in \cite{bober-averages}.



\begin{theorem}\label{comparison thm} Let $q$ be an odd prime, $1\le \tau\le\log\log q- C$ for a sufficiently large absolute constant $C$. There exists a subset 
$\CC_q^L(\tau)\subset \left\{\chi\mod{q} : \chi(-1)=-1,\ |L(1,\chi)|>e^\gamma\tau \right\}$ of cardinality 
\eq{ccL}{
\#\CC^L_q(\tau)  =  \left(1+ O\left( e^{-e^\tau/\tau}\right)\right)
	\cdot \# \left\{\chi\mod{q} : \chi(-1)=-1,\ |L(1,\chi)|>e^\gamma\tau \right\} 
}
such that if $\chi\in\CC^L_q(\tau)$ and $c>0$ is sufficiently large, then
\eq{comp thm main}{
\E_{ q/2-q/\tau^c  \le N\le q/2+q/\tau^c} \left[ \abs{ \sum_{n\le N} \chi(n)  - 
	\frac{\CG(\chi)}{i \pi} (L(1,\bar{\chi})+\log2)} \right] \ll 	
			\sqrt{q} \cdot \frac{(\log \tau)^2}{\sqrt{\tau}} .
}
We can deduce that  
\[
m(\chi)  \ge \tau+ e^{-\gamma}  \log 2 + O((\log\tau)^2/\sqrt{\tau}) .
\]
\end{theorem}


This result, together with \eqref{GS L(1,chi)}, suggests that the lower bound in Theorem \ref{main thm} should be sharp, and that $m(\chi)\approx e^{-\gamma} (|L(1, \chi)| + \log 2)$ when these quantities are ``large".
However the numerical evidence indicates that $m(\chi)$ is perhaps typically a little bigger:

\begin{figure}[!htbp]
    \includegraphics[scale=.65]{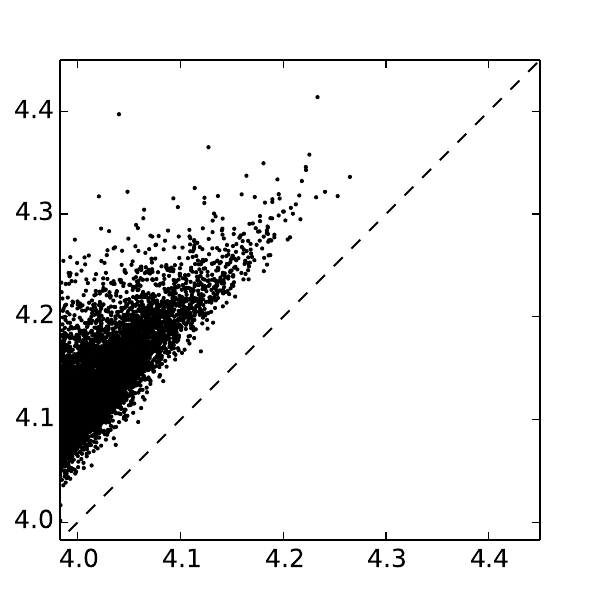}
    \caption{Scatter plot of $m(\chi)$ (vertical axis) and
        $e^{-\gamma} (|L(1, \chi)| + \log 2)$ for the 13617 odd characters with
        $|L(1, \chi)| > 6.4$, and $q$ prime, $10^9 \le q \le 10^9 + 75543$.}
    \label{figure-l-one-m-chi}
\end{figure}


\newpage

The numerical comparison between $\Phi_q$ and $\Phi_q^{L-}$ (the $|L(1,\chi)|$-distribution for odd characters $\chi$) is given in Figure 3:


\begin{figure}[!htbp]
    \includegraphics[scale=.7]{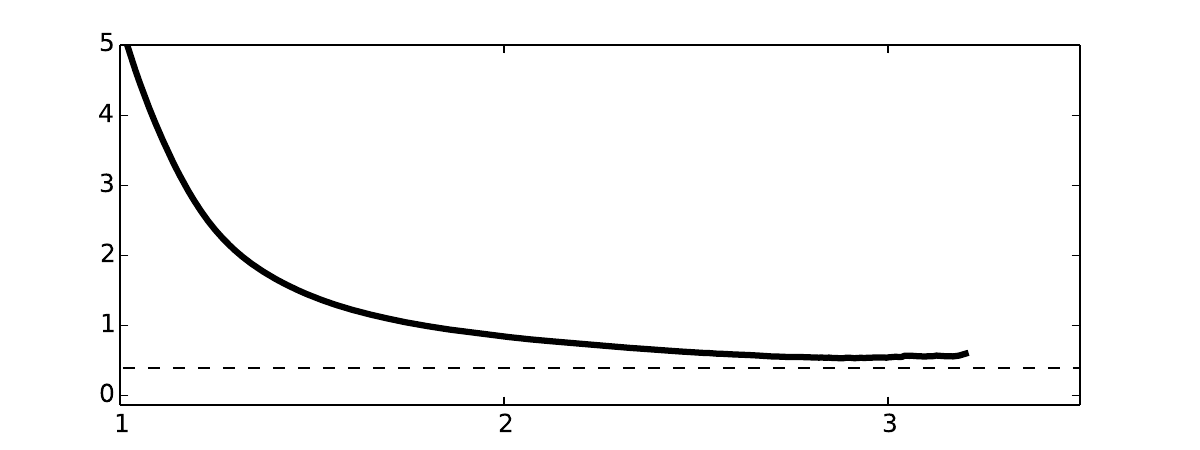}
    \caption{$\log(-\log \Phi_q^{L^-}(\tau)) - \log(-\log \Phi_q^-(\tau))$,
       $q = 12000017$, with a dashed line indicating $e^{-\gamma} \log 2$.
    }\label{figure-l-odd-plot}
\end{figure}




\subsection*{Even characters}  The focus above has been on odd characters since almost all characters with $m(\chi)>\tau$ are odd.   However, we can obtain  analogous results for even characters. 


\begin{theorem}\label{main thm even} There exists an absolute constant $c\ge1$ such that
if $q$ is an odd prime and $1\le \tau\le (\log\log q- M)/\sqrt{3}$ for some $M\ge1$, then
\[
 \exp\left\{- \frac{e^{\sqrt{3}\tau+A}}{\sqrt{3}\tau}  \left(1+O\left( \tau^{-1/2}+e^{-M/2} \right)\right) \right\}
 	\le  \Phi_q^+(\tau)
 		\ll  \exp\left\{ - \frac{e^{\sqrt{3}\tau}}{\tau^c}      \right\}  .
\]
\end{theorem}


The lower bound in the above theorem has the same shape as \eqref{GS L(1,chi)} and it might be close to the truth, unlike the corresponding result for $\Phi_q$ (as explained by Theorem \ref{comparison thm}). This is supported by computations, as it can be seen in Figure \ref{figure-l-even-plot}.


\begin{figure}[!htbp]
    \includegraphics[scale=.7]{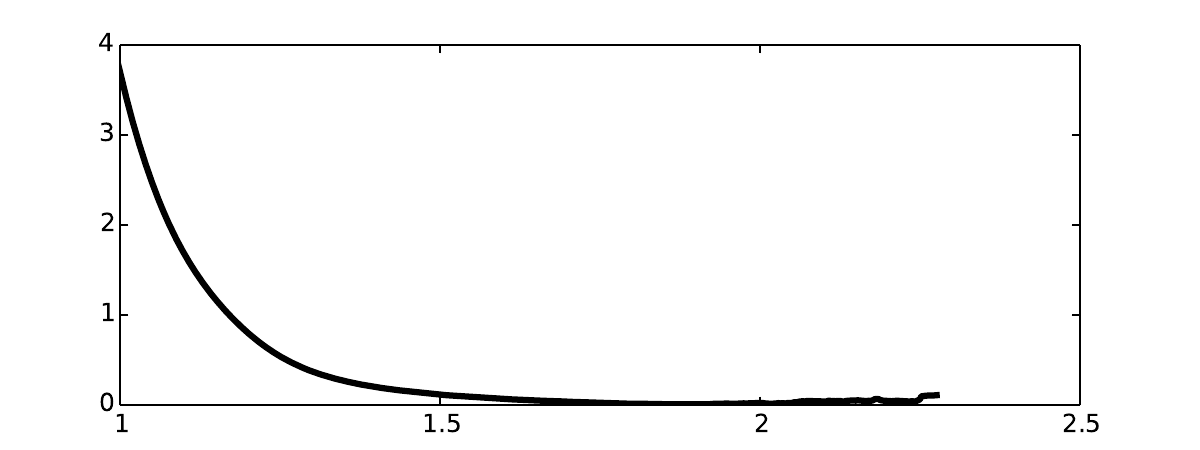}
    \caption{$\log(-\log \Phi_{3q}^{L^-}(\tau)) - \log(-\log \Phi_q^+(\tau))$,
       $q = 12000017$.
    }\label{figure-l-even-plot}
\end{figure}


\subsection*{The distribution function} Classically, the statistical behavior of $\chi(n)$ as $\chi$ varies over all non-principal characters mod $q$ is modelled by $X_n$, where the $\{X_p:p\ \text{prime}\}$ are independent random variables, each uniformly distributed on the unit circle $\U := \{z\in\C:|z|=1\}$, and  $X_n = X_{p_1}^{e_1} \cdots X_{p_k}^{e_k}$ for $n=p_1^{e_1}\ldots p_k^{e_k}$. This suggests modelling $L(1,\chi)$, $\chi\mod q$, by 
$L(1,X)=\sum_{n\ge 1} X_n/n$. Indeed, in \cite{GS03} this was shown to be a very successful model. Similarly, one might guess that the distribution of $\{ \sum_{n\le N} \chi(n)\}_{\chi \pmod q,\ \chi\ne\chi_0}$ could be accurately modelled by 
$\sum_{n\le N} X(n)$ (and thus $M(\chi)$ by $M(X)$), but this seems unlikely since, for fixed $\alpha\in(0,1)$, we have that
\[
\E\left[ \abs{\sum_{n\le \alpha q} X_n }^2\right]   \sim \alpha q  \ \ \text{  whereas  } \ \ 
\frac{1}{\phi(q)}\sum_{\substack{\chi\mod q \\ \chi\neq \chi_0}} \abs{\sum_{n\le \alpha q} \chi(n) }^2 
	 \sim \alpha(1-\alpha) q
\]
as $q\to \infty$.
Moreover Harper \cite{Har13} recently showed that $\sum_{n\le N} X(n)$ is not normally distributed, and subsequently we have no idea what its distribution should be (though \cite{CS} shows that $\sum_{N<n\le N+y} X(n)$ is normally distributed provided $y$ is not too large).

So, what is the right way to model $M(\chi)$? One reason the above model failed is that it does not take account of characters' periodicity.   The periodicity of $\chi$, via the formula
\[
\sum_{n=1}^q \chi(n)e(an/q) = \bar{\chi}(a)  \CG(\chi),
\]
where $e(x)=e^{2\pi ix}$, leads to P\'olya's   formula \cite[eqn. (9.19), p. 311]{MV07}:
\eq{polya}{
\sum_{n\le \alpha q} \chi(n)
	= \frac{\CG(\chi)}{2\pi i}  \sum_{1\le |n|\le z}  \frac{\bar{\chi}(n)(1-e\big(-n\alpha)\big)}{n}
		+ O\left(  1+ \frac{q\log q}{z} \right)  
}
(which is periodic for $\alpha \mod 1$ ). Since $|\CG(\chi)| = \sqrt{q}$, this suggests modelling 
$M(\chi)/(\sqrt{q}/2\pi) =  2e^\gamma m(\chi)$ by 
the random variable
\[
S = \max_{0\le \alpha\le 1}  \left| \sum_{\substack{n\in \mathbb Z \\ n\ne 0}}
	\frac{X_n(1-e(n\alpha))}{n} \right| ,
\]
where $X_{-1}$ is a random variable independent of the $X_p$'s, with probability 1/2 of being $-1$ or $1$, and we have set $X_{-n}=X_{-1}X_n$. The infinite sum here converges with probability 1,\footnote{This follows by the methods of Section \ref{moments} applied to $X_n$ in place of $\chi(n)$.} so that  $S$ is well defined, and we can consider its distribution function,
\[
\Phi(\tau) := \prob( S> 2 e^\gamma  \tau).
\]
The following result confirms our intuition that $S$ serves as a good model for $2e^\gamma m(\chi)$.


\begin{theorem}\label{DistributionTheorem} If $b\ge a>0$ and $\Phi$ is continuous at every point of $[a,b]$, then the sequence of functions $\{\Phi_q\}_{q\ \text{prime}}$ converges to $\Phi$ uniformly on $[a,b]$. In particular, the sequence of distributions $\{\Phi_q\}_{q\ \text{prime}}$ converges weakly to $\Phi$.
\end{theorem}


Theorem \ref{DistributionTheorem} will be shown in Section \ref{prob}. Since $\Phi$ is a decreasing function, it 
has at most countably many discontinuities, and so Theorem \ref{DistributionTheorem} implies that $\lim_{q\to\infty,\ q\ \text{prime}}\Phi_q(\tau)=\Phi(\tau)$ for almost all $\tau>0$. We conjecture that $\Phi$ is a continuous function.

We can prove that, for any $\tau\ge1$,
\[
\exp\left\{- \frac{e^{\tau-\eta+A}}{\tau}  \left(1+O\left(\frac{(\log \tau)^2}{\sqrt{\tau}} \right) \right)   \right\} \le \Phi(\tau)
 	\le \exp\left\{ - \frac{e^{\tau-\eta-2}}{\tau}   \left(1+O\left(\frac{\log \tau}{\tau} \right)\right)   \right\} 
\]
by considering two points $\tau_1$ and $\tau_2$ where $\Phi$ is continuous, with $\tau_1<\tau<\tau_2$ so that 
$\Phi(\tau_2)\le \Phi(\tau) \le \Phi(\tau_1)$. We can bound $\Phi(\tau_1)$ and $\Phi(\tau_2)$ by Theorems \ref{DistributionTheorem} and \ref{main thm}, and then the result follows by letting $\tau_1\to\tau^-$ and $\tau_2\to\tau^+$.



\subsection*{Acknowledgements} The authors would like to thank Sandro Bettin, Adam Harper, K. Soundararajan, Akshay Venkatesh, John Voight, and Trevor Wooley for some helpful discussions. AG and DK are partially supported by Discovery Grants from the Natural Sciences and Engineering Research Council of Canada.


\subsection*{Notation} Given a positive integer $n$, we let $P^+(n)$ and $P^-(n)$ be the largest and smallest prime divisors of $n$, with the notational convention that $P^+(1)=1$ and $P^-(1)=\infty$. Furthermore, we write $\CP(y,z)$ for the set of integers all of whose prime factors lie in $(y,z]$. The symbol $d_k(n)$ denotes the number of ways $n$ can be written as a product of $k$ positive integers. We denote with $\omega(n)$ the number of distinct primes that divide $n$, and with $\Omega(n)$ the total number of prime factors of $n$, counting multiplicity. Finally, we use the symbol ${\bf 1}_A$ to indicate the characteristic function of the set $A$. For example, ${\bf 1}_{(n,a)=1}$ equals 1 if $(n,a)=1$ and 0 otherwise.


\section{The structure of characters with large $M(\chi)$}

We now determine more precise information about the structure of characters with large $M(\chi)$.

\begin{theorem}\label{structure thm} Let $q$ be an odd prime and $1\le\tau\le\log\log q- C$ for a sufficiently large absolute constant $C$. There exists a set $\CC_q(\tau)\subset  \{\chi\mod{q} : m(\chi)>  \tau\}$ of cardinality 
\eq{cc}{
\#\CC_q(\tau) = \left(1+ O\left( e^{-e^\tau/\tau}\right)\right) 
	\cdot  \#  \left\{\chi\mod{q} : m(\chi)>  \tau\right\}
}
for which the following holds: If $\chi\in \CC_q(\tau)$, then
\begin{enumerate}
\item $\chi$ is an odd character.
\item Let $\alpha=N_\chi/q\in[0,1]$, and $a/b$ be the reduced fraction for which $|\alpha-a/b|\le 1/(b\tau^{10})$ with $b\le\tau^{10}$. Let $b_0 =  b$ if $b$ is prime, and $b_0=1$ otherwise. Then 
\eq{1-pretentious}{
\sum_{\substack{p\le e^\tau \\ p\neq b_0}} \frac{|1-\chi(p)|}{p} \ll \frac{(\log\tau)^{3/4}}{\tau^{1/4}} ; \ \ \text{and}
}
\eq{M(chi) odd}{
m(\chi)  =
		e^{-\gamma}	\frac{b_0}{\phi(b_0)} \abs{ \prod_{p\neq b_0} \left(1-\frac{\chi(p)}{p}\right)^{-1} }
				+  O\left(\sqrt{\tau}\log\tau\right) .
}
\end{enumerate}
\end{theorem}


Granville and Soundararajan showed in \cite{GS07} that if $M(\chi)$ is large then $\chi$ must pretend to be a character $\psi$ of small conductor and of opposite parity (that is, $\chi$ is ``$\psi$-pretentious''). Building on the techniques of  \cite{GS07},
Theorem \ref{1-pretentious} shows that, for the vast majority of characters with large $M(\chi)$, the character $\psi$ is the trivial character, and $\chi$ must be  odd (and so $1$-pretentious as in \eqref{1-pretentious}).

The discrepancy between Theorem \ref{main thm} and \eqref{GS L(1,chi)} means that the error term in \eqref{M(chi) odd} cannot be made $o(1)$.


Inequality \eqref{1-pretentious}  allows us to establish  accurate estimates for  $\sum_{n\le N}\chi(n)$ for all $N$.   Here and for the rest of the paper, for $u\ge0$ we set
\eq{P}{
P(u) = e^{-\gamma}\cdot \int_0^u\rho(t)dt ,
}
where $\rho$ is the Dickman--de Bruijn function, defined by $\rho(u)=1$ for $u\in[0,1]$ and via the differential-delay equation $u\rho'(u)=-\rho(u-1)$ for $u>1$. Then we have the following estimate:


\renewcommand{\labelenumi}{(\alph{enumi})}

\begin{theorem}\label{partial sums odd} Let $q, \tau, \alpha, a, b, b_0$ and $\chi\in \CC_q(\tau)$ be as in Theorem \ref{structure thm}. Given $\beta\in[0,1]$,  let $k/\ell$ be the reduced fraction for which $|\beta-k/\ell|\le 1/(\ell \tau^{10})$ with $\ell\le\tau^{10}$. Define $u>0$ by $|\beta-k/\ell|=1/(\ell e^{\tau u})$.
\begin{enumerate}
\item If $b_0=1$, then
\[
\frac{e^{-\gamma}\pi i}{\CG(\chi)} \sum_{n\le \beta q} \chi(n)
	= \begin{cases}
		\ds \tau \big(1-P(u)\big)
		+ O( (\tau\log\tau)^{3/4}) 
			&\mbox{if $\ell=1$},\\
		\ds \tau
                +  O( (\tau\log\tau)^{3/4} )
			&\mbox{if $\ell>1$} .
		\end{cases}
\]
\item If $b_0=b$, then
\[
\frac{e^{-\gamma}\pi i}{\CG(\chi)} \sum_{n\le \beta q} \chi(n)
	= \tau\cdot  \frac{1-1/b}{1-\bar{\chi}(b)/b}\cdot\lambda+ O( (\tau\log\tau)^{3/4}) ,
\]
where
\[
\lambda =   \begin{cases}
				1-P(u)  &\mbox{if $\ell=1$}, \\ 
		 \ds 1+ P(u) \left(\frac{\bar{\chi}(b)}b\right)^{v-1} \frac{1-\bar{\chi}(b)}{b-1} 
			&\mbox{if $\ell=b^v$, $v\ge1$}, \\ 
		1
				&\mbox{otherwise}.
	\end{cases}
\]
Moreover, if we take $\alpha=\beta$ and write $|\alpha-a/b|=1/(be^{\tau u_0})$, then we have that 
\[
(1-P(u_0)) \cdot \frac{|1-\chi(b)|^2}{b^2} \ll \frac{(\log\tau)^{3/4}}{\tau^{1/4}} .
\]
\end{enumerate}
\end{theorem}


\nin
Notice that if $\beta=k/\ell$ then $P(u)=1$, so $1-P(u)$ is a measure of the distance between $\beta$ and $k/\ell$.

Theorems \ref{structure thm} and \ref{partial sums odd} will be proved in Section \ref{odd}. Analogous results can also be proven for even characters:


\renewcommand{\labelenumi}{(\arabic{enumi})}

\begin{theorem}\label{structure thm even}  Let $q$ be an odd prime and $1\le\tau\le \frac{1}{\sqrt{3}} \log\log q- C$ for a sufficiently large constant $C$. There exists a set $\CC_q^+(\tau)\subset  \{\chi\mod{q} : \chi(-1)=1,\ m(\chi)>  \tau\}$ of cardinality 
\eq{cc+}{
\#\CC_q^+(\tau) = \left(1+ O\left( e^{-e^\tau/\tau}\right)\right) 
	\cdot  \#  \left\{\chi\mod{q} : \chi(-1)=1,\  m(\chi)>  \tau\right\}
}
for which the following statements hold. If $\chi\in \CC_q^+(\tau)$ then
\begin{enumerate}
\item If $\alpha=N_\chi/q\in[0,1]$, and $a/b$ is the reduced fraction for which $|\alpha-a/b|\le 1/(b\tau^{10})$ with $b\le\tau^{10}$, then $b=3$.
\item We have that
\eq{mod3-pretentious}{
\sum_{\substack{p\le e^\tau \\ p\neq 3}} \frac{\abs{\leg{p}{3}-\chi(p)}}{p} \ll \sqrt{ \frac{\log\tau}{\tau} }; \ \ \text{and}
}
\eq{M(chi) even}{
m(\chi) 
	= \frac{e^{-\gamma}\sqrt{3}}{2} \abs{L\left(1,\chi\leg{\cdot}{3}\right) } + O(\log \tau) .
}
\item Given $\beta\in[0,1]$,  let $k/\ell$ be the reduced fraction for which $|\beta-k/\ell|\le 1/(\ell \tau^{10})$ with $\ell\le\tau^{10}$. Define $u>0$ by $|\beta-k/\ell|=1/(\ell e^{\sqrt{3}\tau u})$. Then
\[
\frac{e^{-\gamma}\pi}{\CG(\chi)} \sum_{n\le \alpha q} \chi(n)
	= \begin{cases} \ds  \tau P(u)\cdot \leg{k}{3} \cdot 
			 \frac{\bar{\chi}(3^{v-1}) }{3^{v-1}}  
		+ O(\sqrt{\tau}\log \tau)    &\mbox{if $\ell=3^v$ for some $v\ge1$}, \\ 
		\\
		  O(\sqrt{\tau}\log \tau)    &\mbox{otherwise}.
	\end{cases}
\]
\end{enumerate}
\end{theorem}

The proof of Theorem \ref{structure thm even} will be given in Section \ref{even}.


\section{Auxiliary results about smooth numbers}

Before we get started with the proofs of our main results, we state and prove some facts about smooth numbers. As usual, we set
\[
\Psi(x,y) = \#\{n\le x:P^+(n)\le y\} .
\]
We begin with the following simple estimates.


\begin{lemma}\label{rho}
For $u\ge1$, we have that
\[
\rho(u),|\rho'(u)| \le \left(\frac{e^{O(1)}}{u\log u}\right)^u .
\]
\end{lemma}

\begin{proof}
The claimed estimates follow by Corollary 2.3 in \cite{HT93} and Corollary 8.3 in Section III.5 of \cite{Ten}.
\end{proof}


\begin{lemma}\label{smooth}
For $y\ge2$ and $u\ge1$, we have that
\[
\sum_{\substack{n>y^u \\ P^+(n)\le y}} \frac{1}{n} \ll \frac{\log y}{u^u} +  e^{-\sqrt{\log y}}  .
\]
\end{lemma}


\begin{proof} Without loss of generality, we may assume that $y\ge100$. 
When $v\in[1,\sqrt{\log y}]$, de Bruijn \cite{Bru} showed that 
\eq{deBruijn}{
\Psi(y^v,y) = y^v\rho(v) \left(1+O\left(\frac{\log(v+1)}{\log y}\right)\right) .
}
Therefore, Lemma \ref{rho} yields that 
\eq{deBruijn2}{
\Psi(y^v,y)\ll (y/v)^v   \qquad  \left(1\le v\le \sqrt{\log y} \,\right).
}
Together with partial summation, this implies that
\[
\sum_{\substack{y^u<n\le e^{(\log y)^{3/2}} \\ P^+(n)\le y}} \frac{1}{n} \ll \frac{\log y}{u^u} .
\]
Finally, if $\epsilon=2/\log y$, then we note that
\[
\sum_{\substack{n>e^{(\log y)^{3/2}} \\ P^+(n)\le y}} \frac{1}{n} 
	\le e^{-2\sqrt{\log y}}   \sum_{P^+(n)\le y} \frac{1}{n^{1-\epsilon}} 
	\le e^{-2\sqrt{\log y}}  \prod_{p\le y} \left(1-\frac{1}{p^{1-\epsilon}} \right)^{-1} .
\]
Since $p^\epsilon=1+O(\log p/\log y)$ for $p\le y$, the product over the primes is $\ll \log y$, which completes the proof of the lemma.
\end{proof}


We also need the following result, where $P$ is defined by \eqref{P}. 


\begin{lemma}\label{smooth2} If $y\ge2$ and $u>0$, then
\[
\sum_{\substack{n\le y^u \\ P^+(n)\le y}} \frac{1}{n} = P(u)\cdot e^\gamma\log y + O(1) .
\]
In particular, $\lim_{u\to\infty}P(u)=1$.
\end{lemma}


\begin{proof}
If $u\le 1$, then the lemma follows from the estimate $\sum_{n\le x}1/n = \log x+O(1)$. Assume now that $u\ge1$, and set $v=\min\{u,\sqrt{\log y}\}$. Lemma \ref{smooth} implies that
\[
\sum_{\substack{n\le y^u \\ P^+(n)\le y}} \frac{1}{n}
	= \sum_{\substack{n\le y^v \\ P^+(n)\le y}} \frac{1}{n} + O(e^{-\sqrt{\log y}}) .
\]
We use de Bruijn's estimate \eqref{deBruijn} and partial summation to deduce that
\als{
\sum_{\substack{n\le y^v \\ P^+(n)\le y}} \frac{1}{n}
	= \log y + O(1)+\sum_{\substack{y<n\le y^u \\ P^+(n)\le y}} \frac{1}{n}
	&= \log y + O(1)+\int_y^{y^v} \frac{d\Psi(y^t,y)}{y^t} \\
	&=\log y+ O(1) + (\log y)\int_1^v \rho(t) dt \\
	&= (\log y)\int_0^v \rho(t)dt + O(1) \\
	&=(\log y)\int_0^u \rho(t)dt + O(1) ,
}
thus completing the proof of the lemma.
\end{proof}

\begin{rk} Using the slightly more accurate approximation 
\[
\Psi(x,y)=x\rho(u)+\frac{(\gamma-1)x\rho'(u)}{\log y} 
				+ O_\epsilon\left(\frac{x e^{-u}}{(\log y)^2}\right)
\]
for $u\in[\epsilon,1]\cup[1+\epsilon,\sqrt{\log y}]$ (see \cite{Sai89}), one can similarly deduce the stronger estimate
\[
\sum_{\substack{n\le y^u \\ P^+(n)\le y}} \frac{1}{n} = 
	P(u)\cdot e^\gamma\log y + \gamma \rho(u) + o_{y\to\infty}(1) 
\] 
for all $u>0$. 
\end{rk}


Finally, we have the following key estimate. Its second part is a strengthening of Theorem 11 in \cite{FT}.


\renewcommand{\labelenumi}{(\alph{enumi})}

\begin{lemma}\label{around a/b} Let $y\ge10$ and $\alpha\in[-1/2,1/2)$.\begin{enumerate}
\item We have that
\[
\sum_{P^+(n)\le y} \frac{e(n\alpha)}{n} = \sum_{\substack{n\le1/|\alpha| \\ P^+(n)\le y}} \frac{1}{n}+O(1) .
\]

\item There is an absolute constant $c_0\ge2$ such that if $c_1\ge 2$, $a/b$ is a reduced fraction with $b\le (\log y)^{c_1}$ and $|\alpha-a/b|\le \min\{1/(b(\log y)^{c_1}),|\alpha|/(\log y)^{c_0}\}$, then
\[
\sum_{P^+(n)\le y} \frac{e(n\alpha)}{n}
	= \log\frac{1}{1-e(a/b)} + \frac{\Lambda(b)}{\phi(b)} (1-\rho(u))
		+ O_{c_1}\left( \frac{\log\log y}{\log y} \right) ,
\]
where $u$ is defined by $y^{-u}=|\alpha-a/b|$ and $\Lambda$ is von Mangoldt's function.
\end{enumerate}
\end{lemma}


\begin{proof}  (a)  If $n\le 1/|\alpha|$, then $e(n\alpha)=1+O(n|\alpha|)$, and so
\[
\sum_{\substack{n\le 1/|\alpha| \\ P^+(n)\le y}} \frac{e(n\alpha)}{n}
	 -\sum_{\substack{n\le 1/|\alpha| \\ P^+(n)\le y}} \frac{1}{n} \ll
	\sum_{n\le 1/|\alpha| } |\alpha|   \ll 1  .
\]
Hence it suffices to show that
\eq{smooth0}{
\sum_{\substack{n>1/|\alpha| \\ P^+(n)\le y}} \frac{e(n\alpha)}{n} \ll 1.
}
Note that if $y>1/|\alpha|$, then
\[
\sum_{1/|\alpha|<n\le y} \frac{e(n\alpha)}{n}  \ll 1 
\]
using partial summation on the estimate $\sum_{n\le x} e(n\alpha) \ll 1/|\alpha|$. Therefore, in order to complete the proof of \eqref{smooth0}, it suffices to show that
\[
\sum_{\substack{n>w \\ P^+(n)\le y}} \frac{e(n\alpha)}{n} \ll1, 
\]
where $w:=\max\{y,1/|\alpha|\}$. Moreover, Lemma \ref{smooth} reduces the above relation to showing that
\eq{smooth-e1}{
\sum_{\substack{w<n\le w'\\ P^+(n)\le y}} \frac{e(n\alpha)}{n} \ll 1,
}
where $w'=y^{2\log\log y/\log\log\log y}$. In particular, we may assume that $|\alpha|\ge 1/w'$. If $x=y^u\in[w,w']$, then Theorem 10 in \cite{FT} and relation \eqref{deBruijn2} imply that
\[
\sum_{\substack{n\le x \\ P^+(n)\le y}} e(n\alpha)  \ll \Psi(x,y) \cdot \frac{\log(u+1)}{\log y} \ll \frac{x}{e^u\log y},
\]
provided that $|\alpha|>(\log y)^C/x$, for $C$ sufficiently large. (This lower bound on $|\alpha|$ guarantees that the parameter $q$ in Theorem 10 of \cite{FT} is $\ge2$, as required.) Consequently, if we set $w''=\max\{y,(\log y)^C/|\alpha|\}$, then partial summation implies that
\[
\sum_{\substack{w''<n\le w' \\ P^+(n)\le y}} \frac{e(n\alpha)}{n}  \ll 1 .
\]
Then \eqref{smooth-e1} follows from
\eq{smooth-e3}{
\sum_{\substack{w<n\le w'' \\ P^+(n)\le y}} \frac{e(n\alpha)}{n} \ll 1,
}
which is trivial unless $|\alpha|\le (\log y)^C/y$, so that $w''=(\log y)^C/|\alpha|$. 

Let $y_m= w\cdot (1+1/(\log y)^{C+1})^m$ for $m\ge-1$. Let $M$ be the largest integer for which $y_M\le w''$. Then $M\ll (\log y)^{C+1}\log\log y$. If $x\in(y_{m-1},y_m]$ for some $m\in\{1,\dots,M\}$ and $u=\frac{\log x}{\log y}$, then Theorem 3 in \cite{Hil86} implies that
\als{
\Psi(x,y) &= \Psi(y_{m-2},y)+(x-y_{m-2}) \rho(u) \left( 1 + O\left(\frac{\log(u+1)}{\log y}\right) \right)  \\
	&=  \Psi(y_{m-2},y)+(x-y_{m-2}) \rho(u)  + O\left(\frac{x}{(\log y)^{C+2}}\right) .
}
Therefore, partial summation and Lemma \ref{rho} imply that
\als{
\sum_{\substack{y_{m-1}<n\le y_m\\ P^+(n)\le y}} \frac{e(n\alpha)}{n} 
	&= \int_{y_{m-1}}^{y_m} \frac{e(x\alpha)}{x} \left( \rho(u) + \frac{x-y_{m-1}}{x\log y} \rho'(u) \right) dx  
		+ O\left( \frac{1+|\alpha|(y_m-y_{m-1})}{(\log y)^{C+2}}\right)  \\
	&= \int_{y_{m-1}}^{y_m} \frac{e(x\alpha)}{x} \rho(u) dx  
		+ O\left( \frac{1+|\alpha|(y_m-y_{m-1})}{(\log y)^{C+2}}\right)  .
}
Summing the above relation over $m\in\{1,\dots,M\}$ and bounding trivially the contribution of $n\in(y_M,w'']$ implies that
\[
\sum_{\substack{w<n\le w''\\ P^+(n)\le y}} \frac{e(n\alpha)}{n} 
	=  \int_{w}^{w''} \frac{e(x\alpha)}{x} \rho(u) dx  
		+ O\left( \frac{\log\log y}{\log y} \right) .
\]
Finally, integration by parts gives us that
\als{
\int_{w}^{w''} \frac{e(x\alpha)}{x} \rho(u) dx  &= \int_w^{w''} \frac{d}{dx}\left( \frac{e(x\alpha)}{2\pi i \alpha}\right)
	\frac{\rho(u)}{x} dx \\
	&= \left.\frac{e(x\alpha)}{2\pi i\alpha x} \rho(u)\right|_{x=w}^{w''}
		+ \int_w^{w''} \frac{e(x\alpha)}{2\pi i\alpha } 
			\left( \frac{\rho(u)}{x^2} - \frac{\rho'(u)}{x^2\log y} \right) dx  \\
	&\ll \frac{1}{|\alpha| w} \le 1,
}
by the definition of $w$ and Lemma \ref{rho}. This completes the proof of \eqref{smooth-e3} and thus the proof of part (a) of the lemma.


\medskip


(b) Write $\alpha=a/b+\delta$ and set $N=1/(|\delta| \log y)$. We note that, since $|\delta|\le|\alpha|/(\log y)^{c_0}<|\alpha|$, by assumption, we must have that $b\ge2$. 

We start by estimating the part of the sum with $n>N$. We claim that
\eq{around a/b claim}{
\sum_{\substack{P^+(n)\le y \\ n>N}} \frac{e(n\alpha)}{n} \ll  \frac{\log\log y}{\log y} .
}
Note that if $y>N$, then
\[
\sum_{N<n\le y} \frac{e(n\alpha)}{n} \ll  \frac{1}{|\alpha|N} = \frac{|\delta|\log y}{|\alpha|} \le  \frac{1}{\log y}
\]
by partial summation and the estimate $\sum_{n\le x}e(n\alpha)\ll 1/|\alpha|$. Therefore, it suffices to prove that
\[
\sum_{\substack{P^+(n)\le y \\ n>N'}} \frac{e(n\alpha)}{n} \ll  \frac{\log\log y}{\log y} ,
\]
where $N':=\max\{N,y\}$. Furthermore, Lemma \ref{smooth} reduces the above estimate to showing that
\eq{around a/b claim1}{
\sum_{\substack{P^+(n)\le y \\ N'<n\le N''}} \frac{e(n\alpha)}{n} \ll  \frac{\log\log y}{\log y} ,
}
where $N''=y^{2\log\log y/\log\log\log y}$. For this sum to have any summands we need that $N\le N''$. 

Next, we fix $X\in[N',N''/2]$ and estimate 
\[
\sum_{\substack{P^+(n)\le y \\ n\le x}} e(n\alpha) 
\]
for $x\in[X,2X]$. Let $B$ be the constant in Theorem 10 of \cite{FT} when the parameters $A$ and $\delta$ there equal 100 and $1/10$, respectively, and set $c_0=B+1$ and $Q=2X/(\log y)^B$. There is some reduced fraction $r/s=r(X)/s(X)$ with $|\alpha-r/s|\le 1/(sQ)$ and $s\le Q$. Note that 
\[
|\alpha|\ge (\log y)^{c_0} |\delta| = \frac{(\log y)^B}{N} \ge \frac{(\log y)^B}{X} \ge \frac{2}{Q} .
\]
In particular, we must have that $s\ge2$. If, in addition, $s\ge(\log y)^2/2$, then
Theorem 11 in \cite{FT} and relation \eqref{deBruijn2} imply that
\als{
\sum_{\substack{P^+(n)\le y \\ n\le x}} e(n\alpha) 
	\ll \Psi(x,y) \left( \frac{1}{(\log y)^{1.9}} \cdot \frac{\log(u+1)}{\log y} +  \frac{1}{(\log y)^{100}} \right) 
	\ll \frac{x}{(\log y)^{2.9}} ,
}
where $u=\frac{\log x}{\log y}$. Therefore
\eq{around a/b e1}{
\sum_{\substack{P^+(n)\le y \\ X<n\le \min\{2X,N''\}}} \frac{e(n\alpha)}{n} \ll \frac{1}{(\log y)^{2.9}}
}
if $s\ge (\log y)^2/2$. 

Consider now $X$ such that $2\le s<(\log y)^2/2$ and set $\eta=\alpha-r/s$. We claim that $X\le N(\log y)^{B+1}$. We argue by contradiction: assume, instead, that 
$X>N(\log y)^{B+1}$. Then $Q>2N\log y=2/|\delta|$ and thus $|\delta|>1/(sQ)$ and $Q>2b$. In particular, $r/s\neq a/b$ and thus $1/(bs)\le |a/b-r/s|\le |\delta|+1/(sQ)\le |\delta|+1/(2bs)$, which implies that $s\ge1/(2b|\delta|)\ge(\log y)^{c_1}/2\ge(\log y)^2/2$, a contradiction.

Using the above information, we are going to show that
\eq{around a/b claim2}{
\sum_{\substack{P^+(n)\le y \\ X<n\le \min\{2X,N''\}}} \frac{e(n\alpha)}{n} \ll 
		\frac{1}{\log y}  
}
when $2\le s<(\log y)^2/2$. If this relation holds, then \eqref{around a/b claim1} follows by a straightforward dyadic decomposition argument. We use the stronger relation \eqref{around a/b e1} for the $O(\log N'')=O((\log y)^{1.1})$ dyadic intervals corresponding to $X\in[N(\log y)^{B+1/2},N'']$, and we use \eqref{around a/b claim2} for the $O(\log\log y)$ dyadic intervals with $X\in[N',\min\{N(\log y)^{B+1/2},N''\}]$.

It remains to prove \eqref{around a/b claim2} when $2\le s<(\log y)^2/2$, in which case $X$ lies in the interval $[N',\min\{N(\log y)^{B+1},N''\}]$. In addition, note that $|\eta|\le 1/(sQ)\le (\log y)^B/X$. Let $X_m= X\cdot (1+1/(\log y)^{B+1})^m$ for $m\ge-1$. Let $M$ be the largest integer for which $X_M\le \min\{2X,N''\}$. Then $M\ll (\log y)^{B+1}$. Consider $x\in(X_{m-1},X_m]$ for some $m\in\{1,\dots,M\}$ and set $u=\frac{\log x}{\log y}$. Since $s\le(\log y)^2/2$, Theorems 2 and 5 in \cite{FT} imply that
\als{
\sum_{\substack{P^+(n)\le y \\ n\le x }} e(nr/s) 
	&=\sum_{s=cd} \sum_{\substack{P^+(n)\le y \\ m\le x/d \\ (m,c)=1 }} e(mr/c) \\
	&=\sum_{s=cd} \sum_{\substack{1\le j\le c\\ (j,c)=1}}\frac{e(rj/c)}{\phi(c)}
			\sum_{\substack{P^+(n)\le y \\ m\le x/d \\ (m,c)=1 }} 1
		+ O\left(\frac{x}{(\log y)^{B+100}}\right) \\
	&= \sum_{s=cd}  \frac{\mu(c)}{\phi(c)} \sum_{g|c} \mu(g) \Psi(x/(gd),y) 
		+ O\left(\frac{x}{(\log y)^{B+100}}\right) ,
}
since $\sum_{1\le j\le c,\,(j,c)=1}e(rj/c)=\mu(c)$ (see, for example, Lemma \ref{gauss sum} below). Note that
\[
\sum_{s=cd}  \frac{\mu(c)}{\phi(c)} \sum_{g|c} \mu(g) \cdot \frac{1}{gd} = 0 
\]
for $s>1$. Therefore 
\eq{smooth exp sum}{
\sum_{\substack{P^+(n)\le y \\ n\le x }} e(nr/s) 
		= \sum_{s=cd}  \frac{\mu(c)}{\phi(c)} \sum_{g|c} \mu(g) \left(\Psi\left(\frac{x}{gd},y\right) 
			- \frac{\Psi(x,y)}{gd}\right)
		+ O\left(\frac{x}{(\log y)^{B+100}}\right) .
}
We apply Theorem 3 in \cite{Hil86} to find that
\als{
\Psi\left(\frac{x}{k},y\right) - \frac{\Psi(x,y)}{k} 
	&=\Psi\left(\frac{X_{m-2}}{k},y\right) - \frac{\Psi(X_{m-2},y)}{k}   + \frac{x-X_{m-2}}{k} 
		\left(\rho\left(u-\frac{\log k}{\log y}\right) - \rho(u) \right) \\ 
	&\quad+ O\left(\frac{X}{k(\log y)^{B+2}}\right)  .
}
for $x\in[X_{m-1},X_m]$ and $k|s$. Consequently,
\als{
\sum_{\substack{P^+(n)\le y \\ n\le x }} e(nr/s)  
	&=  \sum_{s=cd} \sum_{g|c} \frac{\mu(c)\mu(g)(x-X_{m-2})}{dg\phi(c)} 
				 \left(\rho\left(u-\frac{\log k}{\log y}\right) - \rho(u) \right) \\
	&\quad  + c   + O\left(\frac{X}{(\log y)^{B+2}}\right) ,
}
for some constant $c=c(X,y,m,r,s)\in\R$. Since $e(n\alpha)=e(n\eta)e(nr/s)$, applying partial summation as in the proof of part (a), we deduce that
\als{
\sum_{\substack{X_{m-1}<n\le X_m  \\ P^+(n)\le y}} \frac{e(n\alpha)}{n} 
	&=  \sum_{s=cd}  \frac{\mu(c)}{d\phi(c)} \sum_{g|c} \frac{\mu(g)}{g}
		\int_{X_{m-1}}^{X_m} \frac{e(\eta x)}{x} \left( \rho\left(u-\frac{\log(dg)}{\log y}\right) -\rho(u) \right) dx  \\
	&\quad	+ O\left( \frac{1+|\eta|(X_m-X_{m-1})}{(\log y)^{B+2}}\right)  .
}
Setting $X'=\min\{2X,N''\}$, summing the above relation over $m\in\{1,\dots,M\}$, and bounding trivially the contribution of $n\in(X_M,X']$ implies that
\[
\sum_{\substack{X<n\le X' \\ P^+(n)\le y}} \frac{e(n\alpha)}{n} 
	 =  \sum_{s=cd}  \frac{\mu(c)}{d\phi(c)} \sum_{g|c} \frac{\mu(g)}{g}
		\int_{X}^{X'}
			\frac{e(\eta x)}{x} \left( \rho\left(u-\frac{\log(dg)}{\log y}\right) -\rho(u) \right) dx   
	+ O\left( \frac{1}{\log y} \right) .
\]
Finally, we have that
\als{
\int_{X}^{X'} \frac{e(\eta x)}{x}  \left( \rho\left(u-\frac{\log(dg)}{\log y}\right) -\rho(u) \right) dx   
	\ll \frac{\log(dg)}{\log y} \int_{X}^{2X} \frac{dx}{x} 
	= \frac{(\log2)\log(dg)}{\log y} , 
}
whence
\[
\sum_{\substack{X<n\le X' \\ P^+(n)\le y}} \frac{e(n\alpha)}{n} 
	\ll\frac{1}{\log y}\left(1+ \sum_{s=cd} \sum_{g|c} \frac{c\log(dg)}{d\phi(c)^2}\right) \ll \frac{1}{\log y}, 
\]
which completes the proof of \eqref{around a/b claim2} and thus of \eqref{around a/b claim}. 

To conclude, we have shown that
\[
\sum_{P^+(n)\le y} \frac{e(n\alpha)}{n}
 	= \sum_{\substack{P^+(n)\le y \\ n\le N}} \frac{e(n\alpha)}{n} + O\left(\frac{\log\log y}{\log y}\right)  .
\]
Since $e(n\alpha)=e(n\alpha)+O(n|\delta|)$, we further deduce that
\[
\sum_{P^+(n)\le y} \frac{e(n\alpha)}{n}
 	= \sum_{\substack{P^+(n)\le y \\ n\le N}} \frac{e(an/b)}{n} + O\left(\frac{\log\log y}{\log y}\right)  .
\]
Observe that
\eq{around a/b e10}{
\sum_{n\le \min\{N,y\}} \frac{e(na/b)}{n} 
	&= \log\frac{1}{1-e(a/b)} - \sum_{n>\min\{N,y\}} \frac{e(an/b)}{n}  \\
	&=  \log\frac{1}{1-e(a/b)} + O\left(\frac{1}{\min\{N,y\}\|a/b\|}\right), 
}
where $\|x\|$ denotes the distance of $x$ from its nearest integer. Note that $\|a/b\|\ge|\alpha|-|\delta|\ge |\alpha|/2\ge|\delta|(\log y)^{c_0}/2$ by our assumption that $|\delta|\le|\alpha|/(\log y)^{c_0}$. Since we also have that $|\alpha|\ge1/(2(\log y)^{c_1})$, a consequence of the fact that $b\ge2$, the error term in \eqref{around a/b e10} is
\[
\ll \frac{1}{N|\alpha|}+\frac{1}{y|\alpha|} \le \frac{1}{(\log y)^{c_0-1}} 
	 + \frac{2(\log y)^{c_1}}{y} \ll_{c_1} \frac{1}{\log y}
\]
by our assumption that $|\delta|\le|\alpha|/(\log y)^{c_0}$ . Hence, the lemma is reduced to showing that
\[
\sum_{\substack{P^+(n)\le y \\ y<n\le N}} \frac{e(an/b)}{n}
 	= \frac{\Lambda(b)}{\phi(b)}(1-\rho(u)) + O_{c_1}\left(\frac{\log\log y}{\log y} \right) .
\]
By Lemma \ref{smooth}, it suffices to show that
\eq{around a/b claim3}{
\sum_{\substack{P^+(n)\le y \\ y<n\le y^v}} \frac{e(an/b)}{n}
 	= \frac{\Lambda(b)}{\phi(b)}(1-\rho(u)) + O_{c_1}\left(\frac{\log\log y}{\log y} \right) ,
}
where $y^v:=\min\{N,y^{2\log\log y/\log\log\log y}\}$. Since $2\le b\le(\log y)^{c_1}$, the argument leading to \eqref{smooth exp sum} implies that
\als{
\sum_{\substack{P^+(n)\le y \\ n\le x }} e(an/b) 
		= \sum_{b=cd} \sum_{g|c} \frac{\mu(c)\mu(g)}{\phi(c)} \left(\Psi\left(\frac{x}{gd},y\right) -\frac{\Psi(x,y)}{gd}\right)
		+ O_{c_1}\left(\frac{x}{(\log y)^{100}}\right) 
}
for $x\in[y,y^v]$. Therefore partial summation yields
\als{
\sum_{\substack{P^+(n)\le y \\ y<n\le y^v }} \frac{e(an/b)}{n}
	&=  \sum_{b=cd} \sum_{g|c} \frac{\mu(c)\mu(g)}{dg\phi(c)} 
			\left(\sum_{\substack{y/(dg)<n\le y^v/(dg) \\ P^+(n)\le y}} \frac{1}{n} 
				 - \sum_{\substack{y<n\le y^v \\ P^+(n)\le y}} \frac{1}{n}  \right)
		+ O_{c_1}\left(\frac{1}{(\log y)^{98}}\right) \\
	&=   \sum_{b=cd} \sum_{g|c} \frac{\mu(c)\mu(g)}{dg\phi(c)} 
			\left(\log(dg) - \sum_{\substack{y^v/(dg)<n\le y^v \\ P^+(n)\le y}} \frac{1}{n} \right) 
		+ O_{c_1}\left(\frac{1}{(\log y)^{98}}\right) .
}
Arguing as in Lemma \ref{smooth2}, we apply partial summation on relation \eqref{deBruijn} to deduce that
\[
\sum_{\substack{y^v/(dg)<n\le y^v \\ P^+(n)\le y}} \frac{1}{n}
	= \rho(v) \log(dg) + O\left(\frac{(1+\log(dg))^2}{\log y}\right) . 
\]
Consequently,
\als{
\sum_{\substack{P^+(n)\le y \\ y<n\le y^v }} \frac{e(an/b)}{n}
	&= (1-\rho(v))  \sum_{b=cd} \sum_{g|c} \frac{\mu(c)\mu(g)\log(dg)}{dg\phi(c)}  + O_{c_1}\left(\frac{1}{\log y}\right) \\
	&= \frac{\Lambda(b)}{\phi(b)}(1-\rho(v)) + O\left(\frac{1}{\log y}\right) \\
	&= \frac{\Lambda(b)}{\phi(b)}(1-\rho(u)) + O\left(\frac{\log\log y}{\log y}\right)  ,
}
since $v=\min\{u-\log\log y/\log y,2\log\log y/\log\log\log y\}$ and $\rho(u) \ll  u^{-u}$ by Lemma \ref{rho}. This completes the proof of \eqref{around a/b claim3} and thus the proof of the lemma.
\end{proof}


\begin{cor}\label{fund ineq lem} Let $\chi$ be a Dirichlet character modulo $q$ and $\alpha\in\R$. For all $z,y\ge1$, we have  
\[
\abs{ \sum_{\substack{1\le |n|\le z \\ P^+(n)\le y}} \frac{\chi(n)(1-e(n\alpha))}{n} }
	\le 2e^\gamma\log y + 2\log 2+ O\left( \frac{\log\log y}{\log y}\right) .
\]
\end{cor}


\begin{proof} If $\chi$ is an even character, then $\chi(n)/n+\chi(-n)/(-n)=0$, so that 
\[
\abs{ \sum_{\substack{1\le |n|\le z \\ P^+(n)\le y}} \frac{\chi(n)(1-e(n\alpha))}{n} }
	=2 \abs{ \sum_{\substack{n\le z \\ P^+(n)\le y}} \frac{\chi(n)e(n\alpha)}{n} }
	\le 2\sum_{P^+(n)\le y} \frac{1}{n} = 2e^\gamma\log y + O\left(\frac{1}{\log y}\right),
\]
by the Prime Number Theorem, as desired. We may therefore assume  that $\chi$ is an odd character, so that
\[
\abs{ \sum_{\substack{1\le |n|\le z \\ P^+(n)\le y}} \frac{\chi(n)(1-e(n\alpha))}{n} }
	=2 \abs{ \sum_{\substack{n\le z \\ P^+(n)\le y}} \frac{\chi(n)(1-\cos(2\pi n\alpha))}{n} }  
	\le 2 \sum_{\substack{n\ge 1 \\ P^+(n)\le y}} \frac{1-\cos(2\pi n\alpha)}{n}  .
\]
If $|\alpha|\le 1/(\log y)^{c_0}$, where $c_0$ is the constant from Lemma \ref{around a/b}(b), then Lemma \ref{around a/b}(a) implies that
\als{
 \sum_{\substack{n\ge 1 \\ P^+(n)\le y}} \frac{1-\cos(2\pi n\alpha)}{n} 
 	 = \sum_{\substack{n>1/|\alpha| \\ P^+(n)\le y}} \frac{1}{n} +O(1) 
	&\le  \sum_{P^+(n)\le y} \frac{1}{n} - \sum_{n\le (\log y)^{c_0}} \frac{1}{n} +O(1) \\
	&\le e^\gamma\log y - \log\log y +O(1),
}
which is admissible. Finally, assume that $|\alpha|>1/(\log y)^{c_0}$, and consider a reduced fraction $a/b$ with $b\le(\log y)^{2c_0}$ and $|\alpha-a/b|\le1/(b(\log y_0)^{2c_0})\le |\alpha|/(\log y)^{c_0}$. Then Lemma \ref{around a/b}(b) and the Prime Number Theorem imply that
\[
 \sum_{\substack{n\ge 1 \\ P^+(n)\le y}} \frac{1-\cos(2\pi n\alpha)}{n} 
 	= e^\gamma\log y + \log|1-e(a/b)| - \frac{\Lambda(b)}{\phi(b)}(1-\rho(u)) 
		+ O\left(\frac{\log\log y}{\log y}\right),
\]
where $|\alpha-a/b|=y^{-u}$. Since $\rho(u)\le1$ and $|1-e(a/b)|\le 2$, we deduce that
\[
 \sum_{\substack{n\ge 1 \\ P^+(n)\le y}} \frac{1-\cos(2\pi n\alpha)}{n} \le e^\gamma\log y +\log 2 
 	+ O\left(\frac{\log\log y}{\log y}\right),
\]
which concludes the proof.
\end{proof}


\section{Outline of the proofs of Theorems \ref{main thm} and \ref{main thm even} and proof of Theorem \ref{comparison thm}}\label{outline}

We first deal with Theorem \ref{main thm even}. For the lower bound, note that if $\chi$ is even and $q>3$, then \cite[eqn. (9.18), p. 310]{MV07} yields the inequality
\[
M(\chi)
	\ge \abs{\sum_{n\le q/3}\chi(n)} 
	= \frac{\sqrt{q}}{2\pi} \abs{\sum_{n=1}^\infty \frac{\chi(n)(e(n/3)-e(-n/3))}{n} }  .
\]
Since we also have that
\eq{leg 3}{
e(n/3) - e(-n/3) =  i\sqrt{3} \leg{n}{3} ,
}
we deduce that
\[
M(\chi)
	\ge \frac{\sqrt{3q}}{2\pi} \abs{\sum_{n=1}^\infty \frac{\chi(n)\leg{n}{3}}{n}} 
		= \frac{\sqrt{3q}}{2\pi} \abs{L\left(1,\chi\leg{\cdot}{3}\right)} 
\]
for all even characters $\chi$. The lower bound in Theorem \ref{main thm even} is a direct consequence of the above inequality and of the following result, whose proof is a straightforward application of the methods in \cite{GS07}:


\begin{theorem}[Granville - Soundararajan]\label{distr of euler products}
If $\psi$ is a character modulo some $b\in\{1,2,3\}$, $\CX$ is the set of odd or of even characters mod $q$, and $1\le \tau\le\log\log  q-M$ for some $M\ge0$, then
\[
\frac{1}{|\CX|} \#\left\{\chi\in\CX:   \abs{ L(1,\chi\psi) } > \frac{\phi(b)}{b} e^\gamma \tau \right\}	
	=  \exp\left\{- \frac{e^{\tau+A}}{\tau}  \left(1+O\left( \tau^{-1/2}+e^{-M/2} \right)\right)  \right\}  .
\]
\end{theorem}


\nin
Finally, the upper bound in Theorem \ref{main thm even} follows from Theorems \ref{distr of euler products} and \ref{structure thm even}. (The proof of the latter theorem is independent of the proof of the upper bound of Theorem \ref{main thm even}, as we will see.)

\medskip

Next, we turn to Theorem \ref{main thm}. Its lower bound is a direct consequence of Theorems \ref{comparison thm} and \ref{distr of euler products}. So it remains to outline the proof of the upper bound in Theorem \ref{main thm}, as well as to prove Theorem \ref{comparison thm}. In the heart of these two proofs lies a moment estimate which implies that the bulk of the contribution in P\'{o}lya's Fourier expansion \eqref{polya} comes from smooth inputs:
\[
\sum_{1\le |n|\le z}  \frac{\chi(n)(1-e(n\alpha))}{n}
\approx
\sum_{\substack{1\le |n|\le z \\ P^+(n)\le y}}
\frac{\chi(n)(1-e(n\alpha))}{n}
\]
for most $\chi$ and any $\alpha$. To state this more precisely, we need some notation. Given a set of positive integers $\CA$, set
\[
S_{\CA}(\chi)  =  \max_{\alpha\in[0,1]} \left| \sum_{ n\in \CA } \frac{\chi(n)e(n\alpha)}{n} \right| ;
\]
in the special case when $\CA=\{n\in\N: n\le z,\, P^+(n)>y\}$, write $S_{y,z}(\chi)$ in place of $S_{\CA}(\chi)$. Observe that 
\eq{polya 2}{
&\max_{\alpha\in[0,1]} \left| \sum_{1\le |n|\le z}  \frac{\chi(n)(1-e(n\alpha))}{n}
	-  \sum_{\substack{1\le |n|\le z \\ P^+(n)\le y}}  \frac{\chi(n)(1-e(n\alpha))}{n}  \right|  \\
	&\quad =\max_{\alpha\in[0,1]} \left| \sum_{\substack{1\le |n|\le z \\ P^+(n) > y}}  \frac{\chi(n)(1-e(n\alpha))}{n}  \right|
		 \le   \left| \sum_{\substack{1\le |n|\le z \\ P^+(n) > y}}  \frac{\chi(n)}{n}  \right|
		+  \max_{\alpha\in[0,1]} \left| \sum_{\substack{1\le |n|\le z \\ P^+(n) > y}}  \frac{\chi(n)e(n\alpha)}{n}  \right| \\
	&\quad \le  2 \left| \sum_{\substack{n\le z \\ P^+(n) > y}}  \frac{\chi(n)}{n}  \right|
		+ 2 \max_{\alpha\in[0,1]} \left| \sum_{\substack{n\le z \\ P^+(n) > y}}  \frac{\chi(n)e(n\alpha)}{n}  \right|
				\le 4 S_{y,z}(\chi) .
}
Our next goal is to show that $S_{y,z}(\chi)$ is small for most $\chi$. To do
this, we will prove in Section \ref{moments} that high moments of
$S_{y,z}(\chi)$ are small. As a straightforward application of our moment
bounds we will get the following theorem.


\begin{theorem}\label{tail bound cor}
If $q\in\N$, $3\le y \le q^{11/21}$ and $\delta\in[1/\log y,1]$, then 
\[
\frac{\#\left\{ \chi \mod q: S_{y,q^{11/21}}(\chi)> e^\gamma \delta \right\} }{\phi(q)}
		\ll \exp\left\{-\frac{\delta^2 y}{\log y}\left(1+O\left(\frac{\log\log y}{\log y}\right)\right)\right\} 
			+ q^{-1/(500\log\log q)} .
\]
\end{theorem}


We now show how to complete the proof of Theorems \ref{main thm} and \ref{comparison thm}.


\begin{proof}[Proof of the upper bound in Theorem \ref{main thm}] Let $\alpha\in[0,1]$ be such that $M(\chi)=\abs{\sum_{n\le\alpha q}\chi(n)}$. By P\'olya's expansion \eqref{polya}, \eqref{polya 2} and Lemma \ref{fund ineq lem}, we find that
\als{
m(\chi) & = \frac{e^{-\gamma}}{2} \abs{\sum_{1\le|n|\le  q^{11/21}} \frac{\chi(n)(1-e(n\alpha))}{n}} + O(q^{-1/43})\\
		&\le \log y + \eta + 2e^{-\gamma} S_{y,q^{11/21}}(\chi) + O\left(\frac{\log\log y}{\log y}\right) 
}
for all $y\ge10$, where $\eta=e^{-\gamma}\log 2$. We set $y=e^{\tau-\eta-2\delta}$ for some parameter $\delta\in[1/\log y,1]$ to be chosen shortly. Theorem \ref{tail bound cor} then implies that
\[
\Phi_q(\tau) \le 
	\exp\left\{-\frac{\delta^2 e^{\tau-\eta-2\delta}}{\tau}\left(1+O\left(\frac{\log\tau}{\tau}\right)\right)\right\} 
		+ q^{-1/(500\log\log q)} .
\]
Taking $\delta=1$ completes the proof of the upper bound in Theorem \ref{main thm}.
\end{proof}


We conclude this section with the proof of Theorem \ref{comparison thm}.


\begin{proof}
The set $\CC_q^L(\tau)$ in which we work is defined by
\[
\CC^L_q(\tau) = \{\chi\mod q: \chi(-1)=-1,\ |L(1,\chi)|>e^\gamma\tau,\  |S_{y,q^{11/21}}(\chi)|\le 1\} ,
\]
where we have set $y=e^{\tau+c}$ for some constant $c>0$. If the constant $C$ in the statement of Theorem \ref{comparison thm} is large enough, then Theorems \ref{tail bound cor} and \ref{distr of euler products} imply that \eqref{ccL} does hold. 

Assume now that $\chi\in\CC^L_q(\tau)$. Using partial summation on the P\'olya--Vinogradov inequality, then that 
$ |S_{y,q^{11/21}}(\chi)|\le 1$, and finally Lemma \ref{smooth}, we obtain
\eq{L1val}{
 L(1,\chi)  = \sum_{n\le q^{11/21}} \frac{\chi(n)}{n}  + O(q^{-1/43})
		=  \sum_{\substack{P^+(n)\le y \\ n\le q^{11/21}}} \frac{\chi(n)}{n}   + O(1)
     =  \sum_{P^+(n)\le y } \frac{\chi(n)}{n}   + O(1) .
}
Given that $|L(1,\chi)|>e^\gamma\tau$, we deduce that 
\eq{lb e2}{
e^\gamma\tau<|L(1,\chi)| = \abs{ \sum_{P^+(n)\le y } \frac{\chi(n)}{n} } + O(1) 
 \le \sum_{P^+(n)\le y}\frac{1}{n} + O(1)\le e^\gamma\tau+O(1), 
}
by Mertens's estimate. Therefore
\[
\abs{ \prod_{p\le y} \left(1-\frac{1}{p}\right)^{-1}\left(1-\frac{\chi(p)}{p}\right) }= 1 +O\left(\frac{1}{\tau} \right) .
\]
Taking logarithms, we find that
\[
\sum_{p\le y} \sum_{j=1}^\infty\frac{1-\Re(\chi^j(p))}{jp^j} \ll \frac{1}{\tau} ;
\]
that is, $\chi$ is ``$1$-pretentious''.
Since $|1-u|^2\le2\Re(1-u)$ for $u\in\U$, the above inequality and the Cauchy--Schwarz inequality imply that
\eq{chi close to 1}{
\sum_{p\le z} \sum_{j=1}^\infty\frac{|1-\chi^j(p)|}{p^j} \ll \sqrt{ \frac{\log\log z}{\tau } }
}
for all $z\in[10,y]$. Moreover, since
\[
|1-\chi(n)|\le \sum_{p^j\| n} |1-\chi(p^j)| ,
\]
we find that
\eq{d(chi,1) main thm}{
\sum_{P^+(n)\le z} \frac{|1-\chi(n)|}{n} 
	&\le \sum_{P^+(n)\le z } \frac{1}{n}
			\sum_{p^j\|n} |1-\chi(p^j)|  \\
	&\le \sum_{p\le z}\sum_{j\ge1} \frac{|1-\chi(p^j)|}{p^j} 
		\sum_{P^+(m)\le z} \frac{1}{m}  
		\ll \frac{(\log z)\sqrt{\log\log z}}{\sqrt{\tau}} .
}

Now, since $\chi$ is odd for $\chi\in\CC_q^L(\tau)$, P\'olya's Fourier expansion \eqref{polya} implies that
\eq{polya odd}{
\sum_{n\le \alpha q} \chi(n)
	= \frac{\CG(\chi)}{\pi i} \left( L(1,\bar{\chi}) - \sum_{n=1}^q \frac{\bar{\chi}(n)\cos(2\pi n\alpha)}{n}\right) +O(\log q) .
}
Set
\[
R_\chi(\alpha) = \sum_{n\le \alpha q}\chi(n)  - \frac{\CG(\chi)}{\pi i} ( L(1,\bar{\chi}) + \log 2) .
\]
If $1/y\le |\alpha-1/2|\le 1/\tau^{2c}$ and $c$ is sufficiently large, then Lemma \ref{around a/b} and relation \eqref{polya odd} imply that
\als{
R_\chi(\alpha) 
	&= \sum_{n\le \alpha q}\chi(n)  - \frac{\CG(\chi)}{\pi i} \left( L(1,\bar{\chi})
	- \sum_{P^+(n)\le y} \frac{\cos(2\pi n\alpha)}{n}  + O\left(\frac{\log\tau}{\tau}\right) \right) \\
	&= \frac{\CG(\chi)}{\pi i} \left( \sum_{n=1}^q \frac{\bar{\chi}(n)\cos(2\pi n\alpha)}{n} 
		-  \sum_{P^+(n)\le y} \frac{\cos(2\pi n\alpha)}{n} + O\left(\frac{\log\tau}{\tau}\right) \right)  .
}
Now, Lemma \ref{smooth} and relation \eqref{d(chi,1) main thm} yield that
\als{
&\sum_{n=1}^q \frac{\bar{\chi}(n)\cos(2\pi n\alpha)}{n} -  
		\sum_{P^+(n)\le y} \frac{\cos(2\pi n\alpha)}{n}  \\
		&\quad= \sum_{n=1}^q \frac{\bar{\chi}(n)\cos(2\pi n\alpha)}{n} -  
		\sum_{\substack{P^+(n)\le y \\ n\le q}} \frac{\cos(2\pi n\alpha)}{n} + O\left(\frac{1}{\tau}\right) \\
		&\quad= \sum_{\substack{n\le q \\ P^+(n)> y}} \frac{\bar{\chi}(n)\cos(2\pi n\alpha)}{n}
		+ \sum_{\substack{\tau^{2c}<n\le q \\ P^+(n)\le y}} \frac{(\bar{\chi}(n)-1)\cos(2\pi n\alpha)}{n} 
		+ O\left(\frac{(\log\tau)^2}{\sqrt{\tau}}\right) \\
		&\quad=: \sum_{\tau^{2c}<n\le q } \frac{c_n \cos(2\pi n\alpha)}{n}
		+ O\left(\frac{(\log\tau)^2}{\sqrt{\tau}}\right) ,
}
for some complex numbers $c_n$ of modulus $\le2$. Finally, if $|\alpha-1/2|\le 1/y$, then we simply note the trivial bound 
$R_\chi(\alpha)\ll \sqrt{q}\tau$, which follows by our assumption that $S_{y,q^{11/21}}(\chi)\le 1$ for the $\chi$ we are working with. So \eqref{comp thm main} will certainly follow if we show that
\[
\frac{\tau^c}{2} \int_{1/2-1/\tau^c}^{1/2+1/\tau^c} \abs{\sum_{\tau^{2c}<n\le q } \frac{c_n \cos(2\pi n\alpha)}{n}} d\alpha 
	\ll \frac{\log\tau}{\tau} .
\]
By Cauchy--Schwarz, it suffices to show that
\eq{comp thm claim}{
\mu_2:= \frac{\tau^c}{2} \int_{1/2-1/\tau^c}^{1/2+1/\tau^c} 
	\abs{\sum_{\tau^{2c}<n\le q } \frac{c_n \cos(2\pi n\alpha)}{n}}^2d\alpha \ll \frac{(\log\tau)^2}{\tau^2} .
}
For convenience, set $B=\tau^c$, and note that 
\[
\frac{B}{2}  \int_{1/2-1/B}^{1/2+1/B} \cos(2\pi m\alpha)\cos(2\pi n \alpha) d\alpha  
	= \frac{(-1)^{m+n}}{2} \left(f\left(\frac{m+n}{B}\right) + f\left(\frac{m-n}{B}\right)\right)  ,
\]
where
\[
f(u):= \begin{cases}
		\frac{\sin(2\pi u)}{2\pi u} &\text{if}\ u\neq0,\\
		1 &\text{otherwise}.
	\end{cases}
\]
Therefore
\[
\mu_2 \ll \sum_{k\ge0} |f(k/B)| \left(\sum_{\substack{B^2<m,n\le q \\ m+n=k}} \frac{1}{mn} 
	+\sum_{\substack{B^2<m,n\le q \\ m-n=k}} \frac{1}{mn}  \right)  .
\]
Note that 
\[
 \sum_{\substack{m+n=k \\ B^2< m,n\le q}} \frac{1}{mn}
	\le \frac{{\bf 1}_{k>2B^2}}{k} \sum_{\substack{m+n=k \\ 1\le m,n\le k-1}} \frac{m+n}{mn} 
		\ll \frac{{\bf 1}_{k>2B^2} \log k}{k}
\]
and
\als{
 \sum_{\substack{m-n=k \\ B^2< m,n\le q}} \frac{1}{mn}
 	\le \sum_{n>B^2} \frac{1}{n(n+k)} 
	&\le \frac{1}{k}\sum_{B^2<n\le k}\frac{1}{n}+\sum_{n>\max\{k,B^2\}}\frac{1}{n^2} \\
	&\ll \frac{{\bf1}_{k>2B^2} \log k}{k} +\frac{1}{\max\{k,B^2\}}  .
}
Using the bound $f(k)\ll\min\{1,B/k\}$, we conclude that
\[
\mu_2\ll \frac{\log B}{B} .
\]
Since $B=\tau^c\ge \tau^2$ for $c\ge2$, \eqref{comp thm claim} follows. This completes the proof of \eqref{comp thm main}.

Finally, note that relation \eqref{comp thm main} clearly implies that
\[
m(\chi) \ge |L(1,\chi)+\log2| + O((\log\tau)^2/\sqrt{\tau}) .
\]
Relations \eqref{L1val} and \eqref{d(chi,1) main thm} with $z=y$ imply that 
\[
L(1,\chi)=e^\gamma\tau+O(\sqrt{\tau \log \tau})
\]
If $L(1,\chi)=a+ib$ then $a/|a+ib|=(1+b^2/a^2)^{-1/2}=1+O(b^2/a^2)=1+O((\log \tau)/\tau)$. Therefore
$|L(1,\chi)+\log 2|/|L(1,\chi)|=(1+\frac{2a\log 2+(\log 2)^2}{a^2+b^2})^{1/2}
=1+\frac a{|a+ib|^2}\log 2+O(\tau^{-2})$, and so 
$|L(1,\chi)+\log 2|=|L(1,\chi)|+\log 2 +O(\log \tau)/\tau)$, which completes the proof of the theorem.
\end{proof}

\section{Truncating P\'olya's Fourier expansion}\label{moments}

In this section we show that for most $\chi$ we can limit the Fourier
expansion to a sum over very smooth numbers without much loss, which is
the content of Theorem \ref{tail bound cor}. We prove this theorem by showing
that high moments of $S_{y,z}$ are small:

\begin{theorem}\label{tail bound}
Let $q$ and $k$ be integers with $3\le k\le (\log q)/(400\log\log q)$. For $k\log k \le y\le z \le q^{11/21}$, we have that
\[
\frac1{\phi(q)} \sum_{\chi\mod q} S_{y,z}(\chi)^{2k}
	 \le e^{O(k\log\log y/\log y)} \left(\frac{e^{2\gamma-1}k\log y}{y} \right)^k  
	 	+  \frac{e^{O(k)} }{(\log y)^{19k} }   .
\]
\end{theorem}

One consequence of Theorem \ref{tail bound} is the desired conclusion that $S_{y,z}(\chi)$ is usually small:

\begin{proof}[Deduction of Theorem \ref{tail bound cor} from Theorem \ref{tail bound}] We may assume that $y$ and $q$ are large. Let $\rho$ be the proportion of characters $\chi\mod q$ such that $S_{y,q^{11/21}}(\chi)>e^\gamma\delta$. Moreover, set
\[
k=  \fl{ \min\left\{ \frac{\delta^2y}{\log y}  ,  \frac{\log q}{400\log\log q}  \right\}  },
\]
where $c$ is a constant to be determined. Then Theorem \ref{tail bound} implies that
\als{
\rho	\le \frac{(e^\gamma\delta)^{-2k}}{\phi(q)} \sum_{\chi\mod q} S_{y,q^{11/21}}(\chi)^{2k}
	&\le e^{O(k\log\log y/\log y)}\left(\frac{\delta^{-2}k\log y}{ey} \right)^k 
			+ \frac{\delta^{-2k}e^{O(k)}}{(\log y)^{19k}}   \\
	&\ll e^{-k+O(k\log\log y/\log y)} ,
}
which completes the proof.
\end{proof}


We prove Theorem \ref{tail bound} as an application of the following
technical estimates.
\begin{prop}\label{tail bound prelim} Let $q$ and $k$ be two integers with $3\le k\le (\log q)/(400\log\log  q)$, and $\CA\subset\{n\in\N: y<n\le z,\, P^-(n)>y\}$, where $y$ and $z$ are two positive real numbers such that $k^3\le y\le z\le q^{11/21}$. Then
\[
\frac1{\phi(q)} \sum_{\chi\mod q} S_{\CA}(\chi)^{2k}
	 \ll y^{-k/21} +  q^{-1/10}     \ll \frac{1}{(\log y)^{40 k}}.
\]
\end{prop}


\begin{prop}\label{tail bound prelim2} Let $q$ and $k$ be two integers with $3\le k\le (\log q)/(400\log\log  q)$, and $\CA\subset\{n\in\N: y<n\le z,\, P^-(n)>y\}$, where $y$ and $z$ are two positive real numbers such that $k\log k\le y\le z\le k^{\log\log  k}$. Then
\[
\frac1{\phi(q)} \sum_{\chi\mod q} S_{\CA}(\chi)^{2k}
	  \le \frac{e^{O(k\log\log y/\log y)}k^k}{(ey\log y)^k}  + \frac{e^{O(k)}}{(\log y)^{50k}}  .
\]
\end{prop}

\nin 
Before we proceed to the proof of these propositions, let us see how we can apply them to deduce Theorem \ref{tail bound}. 


\begin{proof}[Deduction of Theorem \ref{tail bound} from Propositions \ref{tail bound prelim} and \ref{tail bound prelim2}] 
Set $Y=\max\{y,k^3\}$. Then
\als{
\sum_{\substack{n\le z \\ P^+(n)>y}}  \frac{\chi(n)e(n\alpha)}{n} 
	&= \sum_{\substack{n\le z \\ y<P^+(n)\le Y}}  \frac{\chi(n)e(n\alpha)}{n} 	
		+ \sum_{\substack{n\le z \\ P^+(n)>Y}}  \frac{\chi(n)e(n\alpha)}{n} \\
	&= \sum_{\substack{a\le z \\ P^+(a)\le y}} \frac{\chi(a)}{a}
			\sum_{\substack{ 1<b\le z/a \\ b\in \CP(y,Y) }} \frac{\chi(b) e( ab\alpha) }{b}
				+ \sum_{\substack{ a\le z \\ P^+(a)\le Y  }} \frac{\chi(a)}{a}
					\sum_{\substack{ 1< b\le z/ a\\ P^-(b)>Y }} \frac{\chi(b) e( ab\alpha) }{b}  .
}
where $\CP(y,Y)$ is the set of integers all of whose prime factors lie in $(y,Y]$.
We let
\[
S^{(1)}_w(\chi) = \max_{\alpha\in[0,1]} \left| \sum_{\substack{y < n \le w \\ n\in \CP(y,Y) }}
	\frac{\chi(n) e(n\alpha)}{n} \right|
\quad\text{and}\quad
S^{(2)}_w(\chi) = \max_{\alpha\in[0,1]} \left| \sum_{\substack{Y < n \le w \\ P^-(n)>Y }}
	 \frac{\chi(n) e(n\alpha)}{n} \right| ,
\]
so that
\eq{tail pf e1}{
S_{y,z}(\chi) \le \sum_{P^+(a)\le y} \frac{S^{(1)}_{z/a}(\chi)}{a}
			+ \sum_{P^+(a)\le Y} \frac{S^{(2)}_{z/a}(\chi)}{a} .
}
We shall bound the moments of each summand appearing above, individually.

We start with the summand involving $S^{(1)}_{w}(\chi)$. Here we may assume that $y\le k^3$ (and thus $Y=k^3$), else
$\CP(y,Y)=\{1\}$ and so  $S^{(1)}_w(\chi)=0$ for all $w$. Set $w'=\min\{w,k^{\log\log  k}\}$ and note that
\[
S^{(1)}_{w}(\chi)
	= S^{(1)}_{w'}(\chi) + O\left(\sum_{\substack{ P^+(n)\le k^3 \\ n\ge k^{\log\log k}}}\frac{1}{n}\right)
	= S^{(1)}_{w'}(\chi) + O\left( \frac{1}{(\log y)^{100} } \right)
\]
by Lemma \ref{smooth}. So Minkowski's inequality and Proposition \ref{tail bound prelim2} imply that
\als{
\left(\frac{1}{\phi(q)} \sum_{\chi\mod q} S^{(1)}_w(\chi)^{2k}\right)^{\frac{1}{2k}}
	\le  \left(\frac{1}{\phi(q)} \sum_{\chi\mod q} S^{(1)}_{w'}(\chi)^{2k} \right)^{\frac{1}{2k}}
		+ O(1/(\log y)^{100}) \\
	\le e^{O(\log\log y/\log y)} \sqrt{\frac{k}{ey\log y}} + O(1/(\log y)^{25}) 
}
Hence, applying H\"older's inequality and Mertens' estimate $\sum_{P^+(a)\le y}1/a=e^\gamma\log y+O(1)$, we arrive at the estimate
\eq{tail pf e2}{
\frac{1}{\phi(q)} \sum_{\chi \mod q} \left( \sum_{P^+(a)\le y} \frac{S^{(1)}_{z/a}(\chi)}{a} \right)^{2k}
	&\le  \frac{(e^\gamma\log y+O(1))^{2k-1} }{\phi(q)}
		\sum_{\chi \mod q} \sum_{P^+(a)\le y}  \frac{S^{(1)}_{z/a}(\chi)^{2k}}{a}  \\
	&\le e^{O(k\log\log y/\log y)} \left( \sqrt{\frac{e^{2\gamma}k\log y}{ey}} + O(1/(\log y)^{24}) \right)^{2k} .
}

Next, in order to bound the summand in \eqref{tail pf e1} involving $S^{(2)}_{w}(\chi)$, we observe that
\[
\frac{1}{\phi(q)} \sum_{\chi \mod q} S^{(2)}_{z/a}(\chi)^{2k}   \ll \frac{1}{(\log Y)^{40k}}
\]
for all $m\ge1$ by Proposition \ref{tail bound prelim}. Therefore H\"older's inequality implies that
\eq{tail pf e3}{
\frac{1}{\phi(q)} \sum_{\chi \mod q} \left( \sum_{P^+(a)\le Y} \frac{S^{(2)}_{z/a}(\chi)}{a}  \right)^{2k}
		&\le \frac{e^{O(k)}(\log Y)^{2k-1}}{\phi(q)} \sum_{P^+(a)\le Y} \sum_{\chi \mod q} \frac{S^{(2)}_{z/a}(\chi)^{2k}}a \\
		&\le \frac{e^{O(k)} }{(\log Y)^{38k} } \le \frac{e^{O(k)} }{(\log y)^{38k} }  .
}
Finally, relations \eqref{tail pf e1}, \eqref{tail pf e2} and \eqref{tail pf e3}, together with an application of Minkowski's inequality, imply that
\eq{tail bound - almost there}{
	\frac{1}{\phi(q)} \sum_{\chi \mod q} S_{y,z}(\chi)^{2k}
			\le  e^{O(k\log\log y/\log y)} \left( \sqrt{\frac{e^{2\gamma}k\log y}{ey}} + O(1/(\log y)^{19}) \right)^{2k} .
}
We note that, for all $\epsilon,\delta>0$, we have that
\eq{eps-delta}{
(\epsilon+\delta)^{2k} \le (\epsilon^{2k}+\delta^k)(1+\sqrt{\delta})^{2k} \le  (\epsilon^{2k}+\delta^k) e^{\sqrt{\delta}/(2k)} .
}
Indeed, if $\epsilon\le \sqrt{\delta}$, then $\epsilon+\delta\le \sqrt{\delta}(1+\sqrt{\delta})$, whereas if $\epsilon>\sqrt{\delta}$, then $\epsilon+\delta\le \epsilon(1+\sqrt{\delta})$. Combining \eqref{tail bound - almost there} and \eqref{eps-delta} completes the proof of Theorem \ref{tail bound}.
\end{proof}


Our next task is to show Proposition \ref{tail bound prelim}. First, we demonstrate the following auxiliary lemma.


\begin{lemma}\label{d_k^2} Let $\epsilon\in(0,1]$ and $k\ge2$ be an integer. Uniformly for $\sigma\ge(2+\epsilon)/(2+2\epsilon)$ and for $y\ge k^{1+\epsilon}$, we have that
\[
\sum_{P^-(n)>y} \frac{d_k(n)^2}{n^{2\sigma}}  \le e^{O(k/\log k)} .
\]
\end{lemma}


\begin{proof} Lemma 3.1 in  \cite{BG}, which is a generalization of Lemma 4 in \cite{GS06}, implies that
\[
\log\left(\sum_{r=0}^\infty \frac{d_k(p^r)^2}{p^{2r\sigma}}\right)
	=\log I_0(2k/p^{\sigma} ) + O(k/p^{2\sigma}) ,
\]
where $I_0$ is defined by \eqref{bessel}. Note that if $p\ge y\ge k^{1+\epsilon}$, then $k/p^{\sigma} < 1$, which implies that
\[
1\le I_0(2k/p^{\sigma} ) \le 1 + \frac{k^2}{p^{2\sigma}} \sum_{m\ge 1} \frac{1}{m!^2} \le 1 + O \left( \frac{k^2}{p^{2\sigma}} \right) .
\]
So we arrive at the estimate
\[
\log\left(\sum_{r=0}^\infty \frac{d_k(p^r)^2}{p^{2r\sigma}}\right) 
	  \ll \frac{k^2}{p^{2\sigma}}  ,
\]
which in turn yields that
\als{
\log\left( \sum_{P^-(n)>y}  \frac{d_k(n)^2}{n^{2\sigma}} \right)
	\ll \sum_{p>y} \frac{k^2}{p^{2\sigma} }
	\ll \frac{k^2}{y^{2\sigma-1} \log y} \ll \frac{k}{\log k},
}
by our assumptions that $y\ge k^{1+\epsilon}$ and $\sigma\ge(2+\epsilon)/(2+2\epsilon)\ge 3/4$.
\end{proof}


\begin{proof}[Proof of Proposition \ref{tail bound prelim}] Without loss of generality, we may assume that $q$ is large, else the result is trivially true. Set $\CA(N)=\CA\cap(N/e,N]$, so that
\[
S_{\CA}(\chi)
	\le \sum_{ \log y< j \le \log z+1 } S_{\CA(e^j)}  (\chi) .
\]
H\"older's inequality with $p=2k/(2k-1)$ and $q=2k$ implies that
\eq{starting point}{
S_{\CA}(\chi)^{2k}
	&\le \left( \sum_{\log y< j \le \log z+1 } \frac{1}{j^{\frac{4k}{2k-1} } } \right)^{2k-1}
		  \left( \sum_{ \log y< j \le \log z+1 } j^{4k} S_{\CA(e^j)}(\chi)^{2k}  \right) \\
	&\le  \frac{1}{(\log y-1)^{2k+1}} \left( \sum_{ \log y< j \le \log z+1 } j^{4k} S_{\CA(e^j)}(\chi)^{2k}  \right) ,
}
which reduces the problem to bounding
\[
\frac{1}{\phi(q)} \sum_{\chi\mod q} S_{\CA(N)}(\chi)^{2k}
\]
for $N\in[y,ez]$. In order to do this, we first decouple the $\alpha_\chi$, the point where the maximum $S_{\CA(N)}(\chi)$ occurs, from the character $\chi$. We accomplish this by noticing that for every $R\in\N$ and for every $\alpha\in(0,1]$, there is some $r\in\{1,2,\dots,R\}$ such that $|\alpha-r/R|\le 1/R$. Then
\[
S_{\CA(N)}(\chi) = \sum_{n\in\CA(N)} \frac{\chi(n)e(n\alpha)}{n}
	= \sum_{n\in\CA(N)} \frac{\chi(n)e(nr/R)}{n} + O\left( \frac{N}{R}\right) .
\]
We choose $R=\fl{N^{21/20} }$. Then Minkowski's inequality implies that
\eq{SNr}{
S_{\CA(N)}(\chi)^{2k}
	&\le 2^{2k-1}\max_{1\le r\le R} \left| \sum_{n\in\CA(N)} \frac{\chi(n)e(nr/R)}{n}\right|^{2k}
		+ O\left( \frac{e^{O(k)}}{N^{k/10} }\right) \\
	&\le 2^{2k-1}\sum_{r=1}^R \left| \sum_{n\in\CA(N)} \frac{\chi(n)e(nr/R)}{n}\right|^{2k}
		+ O\left( \frac{e^{O(k)} }{N^{k/10} }\right) ,
}
which reduces Proposition \ref{tail bound prelim} to bounding
\[
S_{N,r} := \frac{1}{\phi(q)} \sum_{\chi\mod q} \left| \sum_{n\in\CA(N)} \frac{\chi(n)e(nr/R)}{n}\right|^{2k} .
\]
Notice that
\eq{before parseval}{
S_{N,r} = \frac{1}{\phi(q)} \sum_{\chi\mod q} \left| \sum_{\substack{(N/e)^k<n\le N^k\\P^-(n)>y}}
	\frac{\widetilde{d}_k(n;N) \chi(n)}{n}\right|^{2} ,
}
where
\[
\widetilde{d}_k(n;N) :=  \sum_{\substack{ n_1\cdots n_k=n \\ n_1,\dots,n_k\in \CA(N) }} \prod_{j=1}^k e(n_j r/R) .
\]
Clearly,
\[
|\widetilde{d}_k(n;N)|\le d_k(n;N):= \sum_{\substack{n_1\cdots n_k=n \\ n_1,\dots,n_k \in \CA(N) }} 1 .
\]
So opening the square in \eqref{before parseval}, and summing over $\chi\mod q$, we find that
\eq{S_N fund ineq}{
S_{N,r} \le \sum_{\substack{N^k/e^k<m\le N^k \\ (m,q)=1,\, P^-(m)>y}} \frac{d_k(m;N)}{m}
	\sum_{\substack{N^k/e^k<n\le N^k \\ n\equiv m\mod q,\, P^-(n)>y}} \frac{d_k(n;N)}{n} .
}
The right hand side of \eqref{S_N fund ineq} is at most $S_{N,r}^{(1)}+2S_{N,r}^{(2)}$, where
\[
S_{N,r}^{(1)} := \sum_{\substack{N^k/e^k<m\le N^k\\P^-(m)>y}} \frac{d_k(m;N)^2}{m^2}
\quad\text{and}\quad
S_{N,r}^{(2)}:= \sum_{\substack{N^k/e^k<m\le N^k \\ (m,q)=1 \\ P^-(m)>y}} \frac{d_k(m;N)}{m}
	\sum_{\substack{m<n\le N^k \\ n\equiv m\mod q \\ P^-(n)>y}} \frac{d_k(n;N)}{n} .
\]
We shall bound each of these sums in a different way. 

Firstly, note that
\eq{S_N1 bound}{
S_{N,r}^{(1)} \le \frac{e^{k/2}}{N^{k/2}} \sum_{P^-(m)>y} \frac{d_k(m)^2}{m^{3/2}} \ll \frac{e^{O(k)}}{N^{k/2}},
}
by Lemma \ref{d_k^2} with $\epsilon=1$, which is admissible. 

Next, we bound $S_{N,r}^{(2)}$. Note that $n\ge m+q>q$ for $n$ and $m$ in the support of $S_{N,r}^{(2)}$. In particular, for $S_{N,r}^{(2)}$ to have any summands, we need that $N^k>q$. We choose $j\in\{2,\dots,k\}$ such that
\eq{j def}{
N^{j-1} \le q< N^j .
}
Then
\eq{large N e1}{
S_{N,r}^{(2)}
	&\le  \frac{ e^{2k}}{N^{2k}}
		\sum_{\substack{ N^k/e^k<m\le N^k \\ (m,q)=1}} d_k(m;N)
		\sum_{ \substack{ N^k/e^k<n\le N^k \\ n \equiv m \mod q }} d_k(n;N) \\
	&=	\frac{  e^{2k}}{N^{2k}}
		\sum_{\substack{ m\le N^k \\ (m,q)=1}} d_k(m;N)
		\sum_{ \substack{ g\le N^{k-j} \\ (g,q)=1}} d_{k-j}(g;N)
		\sum_{ \substack{ h \le N^j \\ h \equiv  \overline{g} m \mod q }} d_j(h;N) .	
}
Our goal is to bound
\[
D(a) = \sum_{ \substack{ h \le N^j \\h \equiv a \mod q }} d_j(h;N) ,
\]
for every $a\in\{1,\dots,q\}$ that is coprime to $q$. First, assume that $j>1000$. Note that $D(a)$ is supported on integers $h$ which can be written as a product $h=n_1\cdots n_j$ with each of the factors $n_\ell$ lying in the interval $(N/e,N]$ and having all their prime factors $>y$. In particular, $\Omega(n_\ell)\le \log N/\log y$ for all $\ell\in\{1,\dots,j\}$ and consequently 
\[
\Omega(h)\le \frac{j\log N}{\log y} \le \frac{j}{j-1} \cdot \frac{\log q}{\log y} \le \frac{\log q}{0.999\log y} ,
\]
by \eqref{j def}. In particular,
\[
d_j(h;N)\le j^{\Omega(h)} \le k^{\frac{\log q}{0.999\log y}} = q^{\frac{\log k}{0.999\log y}} \le q^{0.334} ,
\]
by our assumptions that $y\ge k^3$ and that $j\le k$. This inequality also holds when $j\le1000$, since in this case $d_j(h;N)\le d_j(h)\ll_\delta h^\delta$ for all $\delta>0$ and $h\le N^j\le N^{1000}\le q^{1000}$ for the numbers $h$ in the range of $D(a)$. So, no matter what $j$ is, we conclude that
\eq{large N e2}{
D(a) \ll q^{0.334} \sum_{ \substack{ h \le N^j \\h \equiv a \mod q }} 1 \ll \frac{N^j}{q^{0.666}} \ll \frac{N^j}{Rq^{0.116}} ,
}
since $R\le N^{21/20}\le q^{11/20}$. Inserting \eqref{large N e2} into \eqref{large N e1}, we arrive at the estimate
\eq{S_N2 bound}{
S_{N,r}^{(2)}
	\ll \frac{  e^{2k} }{N^{2k}}  \cdot N^k \cdot N^{k-j} \cdot  \frac{N^{j}}{R q^{0.116}} 
	= \frac{e^{2k} }{ R q^{0.116} } .
}
Combining relations \eqref{S_N1 bound} and \eqref{S_N2 bound} with \eqref{S_N fund ineq}, we deduce that
\[
S_{N,r} \ll \frac{e^{O(k)}}{N^{k/2}} +  \frac{e^{2k} }{ R q^{0.116} }  .
\]
Together with \eqref{SNr}, the above estimate implies that
\[
\frac{1}{\phi(q)} \sum_{\chi\mod q} S_{\CA(N)}(\chi)^{2k} 
	\ll \frac{N^{21/20} e^{O(k)}}{N^{k/2}} + \frac{ e^{O(k)} }{q^{0.116} }  + \frac{e^{O(k)}}{N^{k/20}} 
	\ll  \frac{ e^{O(k)} }{q^{0.116} }  +  \frac{e^{O(k)}}{N^{k/20}},
\]
since we have assumed that $k\ge3$. Together with \eqref{starting point}, this implies that
\als{
\frac1{\phi(q)} \sum_{\chi\mod q} S_{y,z}(\chi)^{2k}
	&\ll \frac{e^{O(k)}}{(\log y)^{2k+1}} \sum_{ \log y< j \le \log z +1 }
		\left( \frac{  j^{4k} }{ e^{jk/20 } } +
			 \frac{ j^{4k} }{q^{0.116}}   \right) \\
	&\le \frac{ e^{O(k)} (\log y)^{2k-1}}{y^{k/20}}
		+  \frac{  e^{O(k)} (\log z)^{4k+1} }{q^{0.116}(\log y)^{2k+1}} .
}
Since $z\le q^{11/21}$ and $k\le   (\log q)/(400\log\log  q)$, Proposition \ref{tail bound prelim} follows.
\end{proof}


\begin{proof}[Proof of Proposition \ref{tail bound prelim2}]
Without loss of generality, we may assume that $y$ is large enough. We start in a similar way as in the proof of Proposition \ref{tail bound prelim}: we set $R=\fl{z^3}$ and note that
\[
S_{\CA}(\chi)= \max_{1\le r\le R} 
		\abs{  \sum_{n\in\CA} \frac{\chi(n)e(nr/R)}{n} } + O(1/z^2) .
\]
Therefore, Minkowski's inequality implies that
\als{
\left(\frac{1}{\phi(q)}  \sum_{\chi\mod q}  S_{\CA}(\chi)^{2k}\right)^{\frac{1}{2k}}
	&\le \left(\frac{1}{\phi(q)}  \sum_{\chi\mod q}  \max_{1\le r\le R} 
		\abs{  \sum_{n\in\CA} \frac{\chi(n)e(nr/R)}{n} }^{2k}  \right)^{\frac{1}{2k}}
		+ O(1/z^2) \\
	&\le \left( \sum_{r=1}^R \frac{1}{\phi(q)}  \sum_{\chi\mod q}  \abs{  \sum_{n\in\CA} \frac{\chi(n)e(nr/R)}{n} }^{2k}  			\right)^{\frac{1}{2k}} + O(1/z^2) . 
}
We claim that 
\eq{prop2 key claim}{
S_r := \frac{1}{\phi(q)} \sum_{\chi\mod q} \left| \sum_{n\in\CA} \frac{\chi(n)e(nr/R)}{n}\right|^{2k} 
        \le \frac{e^{O(k\log\log y/\log y)}k^k}{(e y\log y)^k} + \frac{e^{O(k)}}{(\log y)^{100k}}
}
for all $r\in\{1,\dots,R\}$. Proposition \ref{tail bound prelim2} follows immediately if we show this relation, since we would then have that
\als{
\frac{1}{\phi(q)}  \sum_{\chi\mod q}  S_{\CA}(\chi)^{2k}
	&\le z^3 e^{O(k\log\log y/\log y)}  \left( \sqrt{\frac{k}{e y\log y}} +O(1/(\log y)^{50}) \right)^{2k}\\
	&\le   \frac{ e^{O(k\log\log y/\log y)} k^k}{ (ey\log y)^k} +\frac{e^{O(k)}}{(\log y)^{50k}} 
}
by our assumption that $z\le k^{\log\log k}$ and \eqref{eps-delta}, which establishes Proposition \ref{tail bound prelim2}. 

 Arguing as in the proof of Proposition \ref{tail bound prelim} and setting
\[
d_k'(n) = \#\{(n_1\cdots n_k)\in\N^k: n=n_1\cdots n_k,\ y<n_j\le z\ (1\le j\le k)\},
\]
we find that 
\eq{S_r fund ineq}{
S_r \le \sum_{\substack{m>y^k \\ (m,q)=1,\, P^-(m)>y}} \frac{d_k'(m)}{m}
	\sum_{\substack{n>y^k \\ n\equiv m\mod q,\, P^-(n)>y}} \frac{d_k'(n)}{n} 
	\le S_r^{(1)}+2\cdot S_r^{(2)},
}
where
\[
S_{r}^{(1)} := \sum_{P^-(m)>y} \frac{d_k'(m)}{m^2}
\quad\text{and}\quad
S_{r}^{(2)}:= \sum_{\substack{(m,q)=1 \\ m\in\CP(y,z)}} \frac{d_k'(m)}{m}
	\sum_{\substack{n>m \\ n\equiv m\mod q \\ P^-(n)>y}} \frac{d_k'(n)}{n} .
\]
We shall bound each of these sums in a different way. 

We start with bounding $S_r^{(2)}$. Set 
\[
\epsilon =   \min\left\{\frac{1}{2},\frac{ \log(y/k)}{\log k} \right\} 
	\ge \frac{\log\log k}{\log k}
\]
for large enough $k$, so that $y \ge k^{1+\epsilon}$. Also, let $\sigma=(2+\epsilon)/(2+2\epsilon)$. Fix $m\in\N$ with $(m,q)=1$, and note that if $n\equiv m\mod{q}$ with $n>m$, then $n\ge m+q>q$. Therefore
\eq{S_r2 e1}{
S_{r}^{(2)}(m) :=  \sum_{\substack{n>m \\ n\equiv m\mod q \\ P^-(n)>y}} \frac{d_k'(n)}{n}  
	&\le \sum_{\substack{ n_1,\dots,n_k\in(y,z] \\ n_1\cdots n_k>q \\ P^-(n_1\cdots n_k)>y 
		\\ n_1\cdots n_k\equiv m\mod{q} }}
		\frac{1}{n_1\cdots n_k}  \\
	&\le \sum_{\substack{1\le r_1,\dots,r_k\le \log(z/y)+1 \\ r_1+\cdots+r_k>\log(q/y^k)}} 	
		 \sum_{\substack{ ye^{r_\ell -1}<n_\ell\le ye^{r_\ell } \\ P^-(n_\ell)>y \, (1\le \ell\le k) 
		 	\\ n_1\cdots n_k\equiv m\mod{q} }}
		\frac{1}{n_1\cdots n_k}  \\
	&\le \sum_{\substack{1\le r_1,\dots,r_k\le \log(z/y)+1 \\ r_1+\cdots+r_k>\log(q/y^k)}} 	
		\frac{e^k}{y^ke^{r_1+\cdots+r_k}}
		 \sum_{\substack{ ye^{r_\ell-1}<n_\ell\le ye^{r_\ell} \\ P^-(n_\ell)>y \, (1\le \ell\le k) 
		 	\\ n_1\cdots n_k\equiv n\mod{q} }} 1  .
}
We fix $r_1,\dots,r_k$ as above, set $y_\ell=ye^{r_\ell}$ for all $\ell\in\{1,\dots,k\}$, and choose $j\in\{2,\dots,k\}$ such that
\eq{j def 2}{
y_1\cdots y_{j-1} \le q < y_1\cdots y_{j} .
}
Then, for any $a\in\N$ that is coprime to $q$, the Cauchy--Schwarz inequality implies that
\als{
\sum_{\substack{ y_\ell/e <n_\ell \le y_\ell \\ P^-(n_\ell)>y \, (1\le \ell\le j) \\ n_1\cdots n_{j}\equiv a\mod{q} }} 1  
	&\le \sum_{\substack{ n\le y_1\cdots y_{j} \\ P^-(n)>y \\ n\equiv a\mod{q} }} d_{j}(n) \\ 
	&\le \left( \sum_{ \substack{ n\le y_1\cdots y_{j} \\ n \equiv a \mod q }} 1 \right)^{1/2}
		\left(  \sum_{ \substack{ n\le y_1\cdots y_{j} \\  P^-(n)>y }}  d_{j}(n)^2 \right)^{1/2} \\
	& \ll \sqrt{\frac{ y_1\cdots y_{j}}{q} } \left( (y_1\cdots y_{j})^{2\sigma}
			 \sum_{  \substack{ n\le y_1\cdots y_j \\ P^-(n)>y }}  \frac{d_{j}(n)^2}{n^{2\sigma}} \right)^{1/2}  .
}
So, applying Lemma \ref{d_k^2}, we deduce that
\als{	
\sum_{\substack{ y_\ell/e<n_\ell \le y_\ell \\ P^-(n_\ell)>y \, (1\le \ell\le j) \\ n_1\cdots n_{j}\equiv a\mod{q} }} 1  
	 	\le e^{O(k)} \sqrt{\frac{(y_1\cdots y_{j})^{1+2\sigma}}{q} }
		\le \frac{e^{O(k)}y_1\cdots y_j}{q^{1-\sigma}} ,
}
since $y_j\le ez\le e^{O(k)}$ and $y_1\cdots y_{j-1}\le q$, by \eqref{j def 2}. Inserting this estimate into \eqref{S_r2 e1}, we deduce that
\eq{S_r2 e2}{
S_{r}^{(2)}(m) \le \frac{e^{O(k)}(\log z)^k}{q^{1-\sigma}} \le \frac{e^{O(k)}(\log z)^k}{q^{\epsilon/3}} .
}
Consequently, we immediately deduce that 
\eq{S_r2 final}{
S_r^{(2)} \le \frac{e^{O(k)}(\log z)^k}{q^{\epsilon/3}} \cdot \left(\frac{\log z}{\log y}\right)^k
		= \frac{e^{O(k)}(\log\log k)^{2k}(\log k)^k}{q^{\epsilon/3}} 
		&\le \frac{e^{O(k)}(\log\log k)^{2k}(\log k)^k}{q^{(\log\log k)/(3\log k)}}  \\
		&\le \frac{e^{O(k)}}{(\log y)^{100k}},
}
which is admissible.

It remains to bound $S_r^{(1)}$. This will be done in a very different way. We observe that if $(X_n)_{n\ge1}$ are the random variables defined in the introduction, then 
\[
S_r^{(1)} \le \E\left[ \abs{ \sum_{\substack{n\in\CP(y,z) \\ n>1 }} \frac{X_n}{n} }^{2k}\right].
\]
We have that
\als{
 \sum_{\substack{n\in\CP(y,z) \\ n>1 }}\frac{X_n}{n}
 	=-1+\prod_{y<p\le z} \left(1-\frac{X_p}{p}\right)^{-1} 
	&=-1 + \exp\left\{\sum_{y<p\le z} \frac{X_p}{p} + O\left(\frac{1}{y\log y}\right) \right\}  \\
	&=-1+ e^T + O(1/y) ,
}
where 
\[
T:=\sum_{y<p\le z} \frac{X_p}{p} .
\]
Therefore
\[
(S_r^{(1)})^{\frac{1}{2k}} \le \E\left[ \abs{ e^T-1}^{2k} \right]^{\frac{1}{2k}}+ O(1/y),
\]
by Minkowski's inequality. Fix $\epsilon\in[1/\log y,1]$. When $|T|\le\epsilon$, then $|e^T-1|\le e^\epsilon|T|$, whereas if $|T|\ge\epsilon$, then we use the trivial bound $|e^T-1|\le 2e^{|T|}\le 2e^{|T|} (|T|/\epsilon)^{\ell}$, for any $\ell\in\N$. Therefore
\[
(S_r^{(1)})^{\frac{1}{2k}} \le e^{\epsilon}\cdot \E\left[ |T|^{2k} \right]^{\frac{1}{2k}} 
	+ 2\cdot \E\left[ e^{2k|T|} (|T|/\epsilon)^{\ell} \right]^{\frac{1}{2k}} + O(1/y) .
\]
We have that
\[
\E\left[ |T|^{2k} \right] = \sum_{\substack{y<p_1,\dots,p_{2k}\le z \\  p_1\cdots p_k =p_{k+1}\cdots p_{2k}}} 
	\frac{1}{p_1\cdots p_{2k}} 
		\le k! \sum_{y<p_1,\dots,p_k\le z} \frac{1}{(p_1\cdots p_k)^2}
		\le \frac{e^{O(k/\log y)} k^k}{(ey\log y)^k} ,
\]
since $k!\ll (k/e)^k\sqrt{k}$ and $\sum_{p>y}1/p^2=1/(y\log y) + O(1/(y\log^2y))$ by the Prime Number Theorem. Moreover,
\als{
\E\left[ e^{2k|T|} (|T|/\epsilon)^{\ell} \right] 
	&= \epsilon^{-\ell}\sum_{m=0}^{\infty}  \frac{(2k)^m}{m!} \E\left[ |T|^{m+\ell} \right]
	\le \epsilon^{-\ell}\sum_{m=0}^{\infty}  \frac{(2k)^m}{m!} \E\left[ |T|^{2m+2\ell} \right]^{1/2} \\
	&\le \epsilon^{-\ell}\sum_{m=0}^{\infty}  \frac{(2k)^m \sqrt{(m+\ell)!}}{m!} \left(\frac{e^{O(1)}}{y\log y}\right)^{m/2+\ell/2}\\
	&\le \left(\frac{e^{O(1)}\ell}{\epsilon^2y\log y}\right)^{\ell/2} \sum_{m=0}^{\infty}  \frac 1{(m/2)!}  \left(\frac{e^{O(1)}k^2}{y\log y}\right)^{m/2} \ll \left(\frac{e^{O(1)}\ell}{\epsilon^2 y\log y}\right)^{\ell/2}  e^{o(k)}
}
as $(m+\ell)!\le m!\ell!2^{m+\ell}$ and $\sqrt{m!}=(m/2)!e^{O(m)}$, with $y\ge k$. Choosing $\ell=\fl{c\epsilon^2 y\log y}$ for an appropriate small constant $c>0$ makes the left hand side $\le e^{O(k)-4c'\epsilon^2 y\log y}$ for some $c'>0$. We take $\epsilon= \log\log y/\log y$ to conclude that
\[
\E\left[ e^{2k|T|} (|T|/\epsilon)^{\ell} \right] \le e^{O(k)-4c'y(\log\log y)^2/\log y} \le e^{-2kc'(\log\log y)^2} ,
\]
where we used our assumption that $k$ is large and $y\ge k\log k$. We thus conclude that
\als{
S_r^{(1)} &\le  e^{O(k\log\log y/\log y)}\left( \sqrt{\frac{k}{e y\log y} } 
	+ 2e^{-c'(\log\log y)^2} \right)^{2k}  \\
	&\le   \frac{ e^{O(k\log\log y/\log y)}  k^k}{(e y\log y)^k } 
	+ e^{O(k)-c'k(\log\log y)^2} ,
}
by \eqref{eps-delta}. Together with \eqref{S_r fund ineq} and \eqref{S_r2 final}, this completes the proof of relation 
\eqref{prop2 key claim} and, thus, of Proposition \ref{tail bound prelim2}.
\end{proof}


\section{The distribution function}\label{prob}

In this section, we prove Theorem \ref{DistributionTheorem}. Throughout this section we fix $\tau>0$ and a large odd prime number $q$ such that $\tau\le (\log\log q)^{5/9}-2\log\log\log q$. Set $y=\exp\{(\log\log  q)^{5/9}\}$ and 
\[
m_y(\chi) = \frac{1}{2e^\gamma} 
	\max_{\alpha\in[0,1]} \abs{\sum_{\substack{n\in\Z,\, n\neq 0 \\ P^+(n)\le y }} \frac{\chi(n)(1-e(n\alpha))}{n} } .
\]
Relation \eqref{polya 2} and Lemma \ref{smooth} imply that
\[
\abs{ m(\chi) - m_y(\chi) }
	\le  2e^{-\gamma} S_{y,q^{11/21}}(\chi) + O\left(q^{-1/\log y}\right).
\]
Therefore, if $q$ is large enough, then Theorem \ref{tail bound cor} and Lemma \ref{smooth} imply that
\[
\frac1{\phi(q)} \#\left\{ \chi \mod q: \abs{m_y(\chi) - m(\chi) } > 2/\log y \right\} 
	 \ll \exp\left\{-  \frac{y}{2(\log y)^3} \right\}
		= o_{q\to\infty}(1) \cdot \Phi_q(\tau)  ,
\]
where the last relation follows by Theorem \ref{main thm} and our assumption that $\tau\le (\log\log q)^{5/9}-2\log\log\log q$. Therefore we conclude that
\[
(1+o_{q\to\infty}(1))\Phi_q(\tau+2/\log y;y) \le \Phi_q(\tau) \le (1+o_{q\to\infty}(1)) \Phi_q(\tau-2/\log y;y) ,
\]
where
\eq{prob e1}{
\Phi_q(t;y):= \frac{1}{\phi(q)} \#\{\chi\mod q: m_y(\chi)  > t \} .
}

We perform the same analysis on $\Phi(\tau)$. With a slight abuse of notation, we set
\[
m_y(X) = \frac{1}{2e^\gamma} \max_{\alpha\in[0,1]} \abs{ \sum_{\substack{n\in\Z,\, n\neq 0 \\ P^+(n)\le y }}  \frac{X_n(1-e(n\alpha))}{n} } . 
\]
The analogue of Theorem \ref{tail bound cor} for the random variables $X_n$ is (more easily) proven by the same method with a few simple changes: we proceed exactly as in the proof of Theorem \ref{tail bound} but wherever we evaluated a sum $(1/\phi(q))\sum_{\chi \pmod q} \chi(h/k)$ there (as $0$, unless $h\equiv k\pmod q$ when it equals $1$), we now evaluate an expectation $\mathbb E[X_h \bar{X_k}]$ which equals $0$ unless $h=k$, when it equals $1$. (In fact this only happens in the proofs of \eqref{S_N fund ineq} and of \eqref{S_r fund ineq}.) We then conclude that 
\eq{prob e2}{
\Phi(\tau+2/\log y ; y) +o_{q\to\infty}( \Phi_q(\tau) ) \le \Phi(\tau) 
	\le  \Phi(\tau-2/\log y ; y) + o_{q\to\infty} (\Phi_q(\tau)) ,
}
where
\[
\Phi(t;y):=\prob( m_y(X) >  t)  .
\]
(We could have written $o_{q\to\infty}(\Phi(\tau))$ in place of $o_{q\to\infty}(\Phi_q(\tau))$ in \eqref{prob e2} but this would have required proving a lower bound for $\Phi(\tau)$. In order to avoid this technical issue, we use for comparison $\Phi_q(\tau)$ whose size we already know by Theorem \ref{main thm}.)
Next, for each $p\le y$ we fix a parameter $\epsilon_p$ of the form $\epsilon_p=2\pi/k_p$ with $k_p\in\N$ to be chosen, and a partition of the unit circle into arcs $\{I_{p,1},\dots,I_{p,k_p}\}$ of length $\epsilon_p$. Moreover, we let $w_{p,j}$ be the point in the middle of the arc $I_{p,j}$ and we define $\CZ$ to be the set of $(\pi(y)+1)$-tuples $\bs z:=(z_{-1},z_2,z_3,z_5,\dots)$ with $z_{-1}\in\{-1,1\}$ and $z_p\in\{w_{p,j}:1\le j\le k_p\}$ for all primes $p\le y$. Given such a choice of $\bs z$ and $n\in\N$, we set $z_n=\prod_{p^e\|n} z_p^e$ and $z_{-n}=z_{-1}z_n$. Moreover, similarly to before, we let
\[
m_y(\bs z) =  \frac{1}{2e^\gamma} \max_{0\le \alpha\le 1}  \abs{   \sum_{\substack{n\in\Z,\, n\neq 0 \\ P^+(n)\le y }} 
	\frac{z_n(1-e(n\alpha))}{n} } .
\]
We will show that there exist choices of $\epsilon_p$ and a constant $C>0$ such that if $X_{-1}=z_{-1}$ and $X_p$ belongs to the arc $I_{p,j}$ centred at $z_p$ for all $p\le y$, then $|m_y(X)  -m_y(\bs z)| \le C/\log y$. This immediately implies that
\eq{prob e3}{
\Phi^*( \tau+C/\log y ;y ) \le  \Phi(y;\tau) \le \Phi^*(\tau-C/\log y ; y) ,
}
where
\[
\Phi^*(t;y):= \frac{1}{|\CZ|} \#\{\bs z\in\CZ: m_y(\bs z) > t\} .
\]

It remains to show that $|m_y(X) -m_y(\bs z)| \le C/\log y$ if $\epsilon_p$ is chosen appropriately. We choose these parameters so that $\max\{(\log p)/2,1\}/(\log y)^3\le \epsilon_p\le \max\{\log p,2\}/(\log y)^3$. Then the condition that $|z_p-X_p|\le \epsilon_p$ for each prime $p\le y$ implies that $|z_n-X_n|\ll (\log |n|)/(\log y)^3$ for all $n\in\Z\setminus\{0\}$ with $P^+(n)\le y$. Hence
\[
|m_y(X) - m_y(\bs z)| 
	\ll  \sum_{\substack{n\ge 1\\ P^+(n)\le y}}  \frac 1n \ \frac{(\log n)}{(\log y)^3}
	= \sum_{\substack{p\le y \\ e\ge 1}}  \frac{(\log p)}{p^e(\log y)^3}
		 \sum_{\substack{m\ge 1 \\ P^+(m)\le y}}  \frac 1m \ll  \frac{1}{\log y},
\] 
which proves our claim. 

We now prove similar inequalities for $\Phi_q$, namely that
\eq{prob e4}{
(1+o_{q\to\infty}(1)) \Phi^*(\tau+C/\log y ; y) \le  \Phi_q(\tau; y) \le (1+o_{q\to\infty}(1))\Phi^*(\tau-C/\log y ;y ) .
}
Theorem 9.3 of \cite{Lamz} states that if $I_p$ is an arc of length $\epsilon_p\ge 1/(\log\log  q)^{5/3}$ for each $p\le y$, then
\[
\frac{1}{\phi(q)} \#\{\chi \mod q : \chi(p)\in I_p \ \text{for each} \ p\le y\} 
	\sim \prod_{p\le y} \epsilon_p \quad(q\to\infty).
\]
The same methods can easily be adapted to also show that
\[
\frac{1}{\phi(q)} \#\{\chi \mod q : \chi(-1)=\sigma,\, \chi(p)\in I_p \ \text{for each} \ p\le y\} \sim 
	\frac{1}{2} \prod_{p\le y} \epsilon_p \quad(q\to\infty)
\]
for $\sigma\in\{-1,1\}$. Since $y=\exp\{(\log\log  q)^{5/9}\}$, the inequality $\epsilon_p\ge 1/(\log\log  q)^{5/3}$ is indeed satisfied for our choice of $\epsilon_p$, thus completing the proof of \eqref{prob e4}. 

Finally, combining relations \eqref{prob e1}, \eqref{prob e2}, \eqref{prob e3} and \eqref{prob e4}, we obtain
\[ 
 \Phi(\tau+(2C+2)/\log y) \le (1+o_{q\to\infty}(1)) \Phi_q(\tau)  
	\le  \Phi(\tau- (2C+2)/\log y). 
\]
If $\Phi$ is continuous in $[a,b]$, then it is also uniformly continuous. It then follows immediately from the above estimate that $\Phi_q\to \Phi$ as $q\to\infty$ over primes, uniformly on $[a,b]$. In order to see that $\Phi_q$ converges also weakly to $\Phi$, consider a continuous function $f:\R\to\R$ of bounded support. Fix $\epsilon>0$. Since $\Phi$ has at most countable many discontinuity points, we deduce that there is an open set $E$ of Lebesgue measure $<\epsilon$ that contains all discontinuities of $\Phi$. If $I$ is a bounded closed interval containing the support of $f$, then the set $I\setminus E$ is compact and thus it can be written as a finite union of closed intervals. Since $(\Phi_q)_{q\ \text{prime}}$ converges uniformly to $\Phi$ on each such interval, it does so on $E\setminus I$ as well. Therefore
\als{
&\limsup_{\substack{q\to\infty \\ q\ \text{prime}}} \abs{ \int_{\R}f(\tau)\Phi_q(\tau) d\tau - \int_{\R}f(\tau)\Phi(\tau) d\tau}\\
	&\quad \le \|f\|_\infty\cdot \left( 2\,\text{meas}(E) 
		+ \text{meas}(I\setminus E) \limsup_{\substack{q\to\infty \\ q\ \text{prime}}}
			\sup_{\tau \in I\setminus E}|\Phi_q(\tau)-\Phi(\tau)|\right) \\ 
	&\quad <2\epsilon  \|f\|_\infty .
}
Since $\epsilon$ was arbitrary, we conclude that 
\[
\lim_{\substack{q\to\infty \\ q\ \text{prime}}} \int_{\R}f(\tau)\Phi_q(\tau) d\tau
	= \int_{\R}f(\tau)\Phi(\tau) d\tau,
\]
thus completing the proof of Theorem \ref{DistributionTheorem}.


\section{Some pretentious results}\label{pretentious}

In this section, we develop some general tools we will use to prove Theorems \ref{structure thm} and \ref{structure thm even}. We begin by stating a result that allows us to concentrate on the case when $\alpha$ is a rational number with a relatively small denominator.


\begin{lemma}\label{centering alpha} Let $y\ge2$, $z\ge(\log y)^5$, $\chi$ be a Dirichlet character, $\alpha\in\R$ and $B\in[(\log y)^5,z]$. Let $a/b$ be a reduced fraction with $b\le B$ and $|\alpha-a/b|\le 1/(bB)$. Then
\[
\sum_{\substack{1\le |n| \le z \\ P^+(n)\le y}} \frac{\chi(n)e(n \alpha)}{n}
	= \sum_{\substack{1\le |n| \le N \\ P^+(n)\le y}} \frac{\chi(n)e(n a/b)}{n}
		+ O(\log B) ,
\]
where $N=\min\{z,|b\alpha - a|^{-1}\}$.
\end{lemma}


\begin{proof}
This follows immediately by the second part of Lemma 4.1 in \cite{G} (see also Lemma 6.2 in \cite{GS07}).
\end{proof}


When $b$ is large, we have the following result.


\begin{lemma}\label{MV} Let $|\alpha-a/b|\le 1/b^2$, where $(a,b)=1$. For all $z, y\ge 3$, we have that
\[
\sum_{\substack{n \le z \\ P^+(n)\le y}} \frac{\chi(n)e(n \alpha)}{n}
	\ll \log b+ \log\log  y +\frac{(\log b)^{5/2}}{\sqrt{b}} \log y .
\]
\end{lemma}

\begin{proof} This is Corollary 2.2 in \cite{G}, which is based on a result due to Montgomery and Vaughan \cite{MV77}.
\end{proof}


For smaller $b$, we shall use the following formula.


\begin{lemma}\label{formula} Let $\chi$ be a Dirichlet character and $(a,b)=1$. For $z,y\ge 1$ we have that
\[
\sum_{ \substack{ 1\le |n|\le z \\ P^+(n)\le y} } \frac{\chi(n)e(an/b)}{n}
	= \frac{2}{b} \sum_{d|b} \frac{\chi(b/d) d}{\phi(d)}
		\sum_{\substack{ \psi \mod d \\ \chi\bar{\psi}\, \text{odd} }} \bar{\psi}(a) \CG(\psi)
		\sum_{ \substack{ n\le zd/b \\ P^+(n)\le y} }  \frac{\chi(n)\bar{\psi}(n)}{n}   .
\]
\end{lemma}


\begin{proof}
If  $(c,d)=1$, then
\[
e(c/d)=\frac{1}{\phi(d)} \sum_{\psi\mod d} \bar{\psi}(c) \CG(\psi).
\]
So, writing $b/d$ for the greatest common divisor of $n$ and $b$, we find that
\als{
\sum_{ \substack{ 1\le |n|\le z \\ P^+(n)\le y }} \frac{\chi(n)e(na/b)}{n}
	&= \sum_{d|b} \frac{\chi(b/d)d}{b}
		\sum_{ \substack{ 1\le |m|\le zd/b \\ P^+(n)\le y,\,(m,d)=1}} \frac{\chi(m)e(am/d)}{m} \\
	& = \sum_{d|b} \frac{\chi(b/d)d}{b\phi(d)}\sum_{\psi\mod d} \bar{\psi}(a) \CG(\psi)
	 \sum_{\substack{ 1\le |m| \le zd/b \\ P^+(n)\le y }} \frac{\chi(m) \bar{\psi}(m) }{m}  .
}
Finally, observe that the innermost sum vanishes if $\chi\bar{\psi}$ is an even character, whereas if $\chi\bar{\psi}$ is odd it equals
\[
2\sum_{ \substack{ m\le z/d \\  P^+(m)\le y }} \frac{\chi(m) \bar{\psi}(m)} {m}.
\]
This concludes the proof of the lemma.
\end{proof}


Following Granville and Soundararajan \cite{GS07}, we are going to show that the all the terms in the right hand side in Lemma \ref{formula} are small unless $\psi$ is induced by some fixed character $\xi$ which depends at most on $\chi$ and $y$. As in \cite{GS07}, in order to accomplish this, we define a certain kind of ``distance'' between two multiplicative functions $f$ and $g$ of modulus $\le1$:
\[
\D(f,g;y)^2 := \sum_{p\le y} \frac{1-\Re(f(p)\bar{g(p)})}{p} .
\]
Then we let $\xi=\xi(\chi,y)$ be a primitive character of conductor $D=D(\chi,y) \le \log y$ such that
\eq{most pretentious}{
\D(\chi,\xi; y) = \min_{\substack{d\le \log y \\ \psi \mod d \\ \psi\ \text{primitive} }} \D(\chi,\psi; y).
}
We need a preliminary result on certain sums of multiplicative functions, which is the second part of Lemma 4.3 in \cite{GS07}.


\begin{lemma}\label{general ub}
Let $f:\N\to\U$ be a multiplicative function. For $z,y\ge 1$ we have that
\[
\sum_{ \substack{ n\le z \\ P^+(n) \le y}} \frac{f(n)}{n} \ll (\log y) \exp\{-\D^2(f,1;y) / 2 \}.
\]
\end{lemma}

Then we have the following ``repulsion'' result.


\begin{lemma}\label{repulsion} Let $\chi$, $y$ and $\xi$ be as above. If $\psi$ is a Dirichlet character modulo $d\le y$, of conductor $\le \log y$, that is not induced by $\xi$, then
\[
\sum_{ \substack{ n\le z \\ P^+(n) \le y}} \frac{\chi(n)\bar{\psi}(n)}{n} \ll (\log y)^{1/2+\sqrt{2}/4+o(1)}    \quad( y\to \infty ) .
\]
\end{lemma}


\begin{proof} Let $\psi_1\mod{d_1}$ be the primitive character inducing $\psi$. Since
\[
\D^2(\chi,\psi;y) \ge \D^2(\chi,\psi_1;y)  - \sum_{p|d} \frac{2}{p} \ge \D^2(\chi,\psi_1;y)  - O(\log\log\log  d) ,
\]
Lemma 3.4 in \cite{GS07} and the definition of $\xi$ imply that
\[
\D^2(\chi,\psi;y)\ge \left( 1 - \frac{\sqrt{2}}{2} +o(1) \right) \log\log  y \quad(y\to\infty) .
\]
The claimed estimate then follows by Lemma \ref{general ub} above.
\end{proof}


When applying Lemma \ref{formula}, we will need to evaluate the Gauss sum that arises. In order to do this, we shall use the following classical result (see, for example, Theorem 9.10 in \cite[p. 289]{MV07}).


 \begin{lemma}\label{gauss sum}
Let $\psi$ be a character modulo $d$ induced by the primitive character $\psi_1$ modulo $d_1$. Then
\[
\CG(\psi) = \mu(d/d_1)\psi_1(d/d_1) \CG(\psi_1) .
\]
\end{lemma}


We also need the following simple estimate, which we state below for easy reference.


\begin{lemma}\label{change gcd}
Let $f:\N \to \U$ be a completely multiplicative function. For all $a\in\N$, we have that
\[
\sum_{ n\le z } \frac{f(n)}{n} 
	=  \prod_{p|a} \left(1-\frac{f(p)}{p}\right)^{-1}
		\sum_{\substack{ n\le z \\ (n,a)=1 }} \frac{f(n)}{n} 
		+ O\left(\frac{a}{\phi(a)}\sum_{p|a} \frac{\log p}{p}  \right) 
\]
and
\[
\sum_{\substack{ n\le z \\ (n,a)=1 }} \frac{f(n)}{n} 
	 = \prod_{p|a} \left(1-\frac{f(p)}{p}\right)
		\sum_{n\le z }\frac{f(n)}{n}  
		+ O\left(\frac{a}{\phi(a)}\sum_{p|a} \frac{\log p}{p} \right)  .
\]
\end{lemma}


\begin{proof} We write $d|a^\infty$ if $p|a$ for all primes $p|d$. Then
\als{
\sum_{n\le z}   \frac{f(n)}{n} 
	= \sum_{d|a^\infty }\frac{f(d)}{d} 
		\sum_{\substack{ m\le z/d \\ (m,a)=1  }} \frac{f(m)}{m}  
	& =  \sum_{d|a^\infty} \frac{f(d)}{d} 
		\sum_{\substack{ m\le z \\ (m,a)=1  }} \frac{f(m)}{m} 
			+ O\left(\sum_{d|a^\infty} \frac{ \log d}{d} \right) \\
	&=  \prod_{p|a} \left(1-\frac{f(p)}{p}\right)^{-1}
		\sum_{\substack{ m\le z \\ (m,a)=1}} \frac{f(m)}{m} 
		+ O\left( \frac{a}{\phi(a)}\sum_{p|a} \frac{\log p}{p}  \right) .
}
The second part is proved similarly, starting from the identity
\[
\sum_{\substack{ n\le z \\ (n,a)=1 }} \frac{f(n)}{n} 
	= \sum_{d|a} \frac{\mu(d)f(d)}{d} \sum_{ m\le z/d } \frac{f(m)}{m}.
        \qedhere
\]
\end{proof}


Combining the above results, we prove the following simplified version of Lemma \ref{formula}.

\begin{lemma}\label{formula2}
Let $\chi$, $y$, $\xi$ and $D$ be as above, and consider a real number $z\ge1$ and a reduced fraction $a/b$ with $1\le b\le (\log y)^{1/100}$ and $(b,q)=1$. If either $D\nmid b$ or $\chi\bar{\xi}$ is even, then
\[
\sum_{ \substack{ 1\le |n|\le z \\ P^+(n)\le y }} \frac{\chi(n) e(na/b) }{n}  \ll  (\log y)^{0.86} ,
\]
whereas if $D|b$ and $\chi\bar{\xi}$ is odd, then
\[
\abs{\sum_{ \substack{ 1\le |n|\le z \\ P^+(n)\le y }} \frac{\chi(n) e(n a/b )}{n}}
		\le (\log y) \min\left\{   \frac{2e^\gamma}{\sqrt{D}} \cdot   (2/3)^{\omega(b/D)}, 
			e^{-\D^2(\chi,\xi;y)/2+O(1)}  \right\}
	+ O((\log y)^{0.86})  .
\]
\end{lemma}


\begin{proof} 
By Lemma \ref{repulsion}, we see that if $\psi$ is a character modulo $d$ that is not induced by $\xi$, then
\eq{psi not xi}{
\sum_{\substack{ m \le x \\ P^+(n)\le y }} \frac{\bar{\psi}(m)\chi(m) }{m}
	 \ll (\log y)^{1/2+\sqrt{2}/4+o(1)} \le (\log y)^{0.854}   \quad(x\ge 1) .
}
So the first result follows by Lemma \ref{formula}. Finally, if $D|b$ and $\chi\xi$ is odd, then Lemmas \ref{formula} and \ref{gauss sum}, and \eqref{psi not xi} imply that
\als{
\sum_{ \substack{ 1\le |n|\le z \\ P^+(n)\le y }} \frac{\chi(n) e(n a/b )}{n}
	&= \frac{2\bar{\xi}(a)\CG(\xi)}{b} \sum_{\substack{d|b \\ d\equiv 0\mod D}} \frac{\chi(b/d) d\, \mu(d/D) \xi(d/D) }{ \phi(d)}
			\sum_{ \substack{ n\le dz/b,\,(n,d)=1 \\ P^+(n)\le y} }  \frac{\chi(n)\bar{\xi}(n)}{n}  \\
	&\quad	+ O((\log y)^{0.86}) .
}
Writing $d=Dc$, noticing that $(c,D)=1$ if $\xi(c)\neq0$, and using the trivial bound \\ $\sum_{dz/b<n\le z}1/n\ll \log(b/d)$ to extend the sum over $n\le dz/b$ to a sum over $n\le z$, we find that 
\begin{multline*}
\sum_{\substack{ 1\le |n|\le z \\ P^+(n)\le y }} \frac{\chi(n) e(n a/b )}{n}
    = \frac{2\bar{\xi}(a)\CG(\xi)}{b} \frac{D}{\phi(D)}
      \sum_{\substack{c|b/D \\ (c,D)=1}} \frac{\chi(b/(Dc)) c\, \mu(c) \xi(c) }{ \phi(c)}
      \sum_{ \substack{  n\le z,\, P^+(n)\le y \\ (n,cD)=1 }}
      \frac{\chi(n)\bar{\xi}(n)}{n}  \\
      + O( (\log y)^{0.86}) .
\end{multline*}
Setting $m=b/D$ and applying Lemma \ref{change gcd} with $f(n)=\chi(n)\bar{\xi}(n){\bf 1}_{(n,cD)=1}{\bf 1}_{P^+(n)\le y}$ and $m$ in place of $a$ we deduce that
\begin{multline*}
\sum_{ \substack{ 1\le |n|\le z \\ P^+(n)\le y }} \frac{\chi(n) e(n a/b )}{n}
	= \frac{2\bar{\xi}(a)\CG(\xi)}{b} \frac{D}{\phi(D)} 
		\sum_{\substack{c|m \\ (c,D)=1 }} \frac{\chi(m/c) c\, \mu(c) \xi(c) }{ \phi(c)}
		\prod_{p|m,\,p\nmid cD}\left(1-\frac{\chi(p)\bar{\xi}(p)}{p}\right)^{-1} \\
	\times		\sum_{ \substack{  n\le z\, P^+(n)\le y \\ (n,mD)=1 }}  \frac{\chi(n)\bar{\xi}(n)}{n}  
					+ O( (\log y)^{0.86}) .
\end{multline*}
Since we have assumed that $(b,q)=1$, we have that
\eq{sum over c}{
\sum_{\substack{c|m \\ (c,D)=1 }} \frac{\chi(m/c) c\, \mu(c) \xi(c) }{\phi(c)}
    &\prod_{p|m,\,p\nmid cD}\left(1-\frac{\chi(p)\bar{\xi}(p)}{p}\right)^{-1} \\
    = \chi(m)&
    \prod_{p|m,\,p\nmid D} \left(1- \frac{\bar{\chi}(p) \xi(p) }{1-1/p}
    \left(1-\frac{\chi(p)\bar{\xi}(p)}{p} \right)\right)\left(1-\frac{\chi(p)\bar{\xi}(p)}{p}\right)^{-1} \\
    = \chi(m)&\frac{m\phi(D)}{\phi(mD)}
    \prod_{p|m,\,p\nmid D} \frac{1-\bar{\chi}(p) \xi(p)}{1-\chi(p)\bar{\xi}(p)/p}.
}
If $z\in\C$ with $|z|=1$, then 
\[
\abs{\frac{1-\bar{z}}{1-z/p}}\le \sqrt{\frac{2}{1+1/p^2}} .
\]
Therefore the absolute value of the sum in \eqref{sum over c} is
\[
\le\frac{m\phi(D)}{\phi(mD)} 
		\prod_{p|m} \sqrt{\frac{2}{1+1/p^2}} .
\]
Since we also have that $mD=b$, we deduce that
\als{
\abs{\sum_{ \substack{ 1\le |n|\le z \\ P^+(n)\le y }} \frac{\chi(n) e(n a/b )}{n} }
	&\le \frac{2}{\sqrt{D}m} \frac{b}{\phi(b)} \abs{\sum_{\substack{ P^+(n)\le  y \\ (n,b)=1}} \frac{\chi(n)\bar{\xi}(n)}{n} }
		\prod_{p|m} \sqrt{\frac{2}{1+1/p^2}}
		+ O((\log y)^{0.86}) \\
	&\le\frac{2}{\sqrt{D}m} \min\left\{ e^\gamma\log y+O(1),
		\left(\frac{b}{\phi(b)}\right)^{3/2} (\log y) e^{-\D^2(\chi,\xi;y)/2+O(1)} \right\}  \\
	&\qquad\times\prod_{p|m} \sqrt{\frac{2}{1+1/p^2}}+  O((\log y)^{0.86}) \\
	&\le (\log y)\min\left\{ \frac{2e^\gamma}{\sqrt{D}} (2/3)^{\omega(m)},  e^{-\D^2(\chi,\xi;y)/2+O(1)} \right\} 
		+ O((\log y)^{0.86}) ,
}
where we used Lemma \ref{general ub} and the inequality $\frac{1}{p}\sqrt{\frac{2}{1+1/p^2}}\le \sqrt{2/5}\le2/3$ for $p\ge2$. This completes the proof of the lemma.
\end{proof}


Finally, if $\chi$ pretends to be 1 in a strong way, then we can get a very precise estimate on the sum of $\chi(n)e(an/b)/n$ over smooth number using estimates for such numbers in arithmetic progressions due to Fouvry and Tenenbaum \cite{FT}.


\begin{lemma}\label{remove exp} Let $z,y\ge2$, $\chi$ be a character mod $q$, and $a/b$ be a reduced fraction with $1\le b\le (\log y)^{100}$. Then
\als{
\sum_{\substack{ n\le z \\ P^+(n)\le y }} \frac{\chi(n) e(na/b)}{n}
	&= \sum_{b=cd} \frac{\mu(c)\chi(d)}{\phi(c)d}
		\sum_{\substack{ n\le z,\, (n,c)=1  \\ P^+(n)\le y }} \frac{\chi(n)}{n}
		 + O(E) \\
	&= \frac{ \mu(q_1)}{\phi(q_1)} \frac{ \chi(b_1)}{\phi(b_1)} \prod_{p|b_1} (1-\overline{\chi}(p)) \left(1-\frac{\chi(p)}{p}\right)^{-1}
		 \sum_{\substack{ n\le z,\,(n,b)=1 \\ P^+(n)\le y  }} \frac{\chi(n)}{n}
		+ O( E) ,
}
where $b_1$ is the largest divisor of $b$ with $(b_1,q)=1$ and $b=b_1q_1$, with 
\[
E = \left(1+\frac{b}{\phi(b)}(e^\Delta-1)\right)\log\log y 
\quad\text{where}\quad 
\Delta = \sum_{p\le y} \frac{|1-\chi(p)|}{p-1} \ll  \D(\chi,1;y) \sqrt{\log\log y}  .
\]
\end{lemma}


\begin{proof} First, note that the inequality $\Delta\ll \D(\chi,1;y) \sqrt{\log\log  y}$ follows by the Cauchy--Schwarz inequality and the fact that $|1-z|^2\le2\Re(1-z)$ for $z\in\U$. We write $\chi=h*1$. Then we have that $|h(p^j)|=|\chi(p^j)-\chi(p^{j-1})|\le|1-\chi(p)|\le2$ and thus
\eq{h bound}{
\sum_{P^+(m)\le y} \frac{|h(m)|}{m} \le e^\Delta \ll (\log y)^2 .
}
Moreover, observe that, using Lemma \ref{smooth}, we may assume that $z\le y^{\log\log y}$. 

We begin by estimating the sum
\[
S(t):=\sum_{\substack{ n\le t \\ P^+(n)\le y }} \chi(n) e(na/b) 
\]
for $t\in[(\log y)^{400},z]$. Set $\ell_0=(\log y)^{200}$ and note that
\eq{remove e1}{
S(t)  &= \sum_{b=cd} \chi(d) \sum_{\substack{ m\le t/d \\ P^+(m)\le y,\,(m,c)=1 }} \chi(m) e(ma/c) \\
	&= \sum_{b=cd} \chi(d)  \sum_{\substack{ k\le t/d \\ P^+(k)\le y,\,(k,c)=1 }} h(k) 	
		  \sum_{\substack{ \ell \le t/(dk) \\ P^+(\ell)\le y,\,(\ell,c)=1 }} e(k\ell a/c) \\
	&= \sum_{b=cd} \chi(d)  \sum_{\substack{ k\le t/(d\ell_0) \\ P^+(k)\le y,\,(k,c)=1 }} h(k) 	
		\sum_{\substack{1\le j\le c \\ (j,c)=1}} e(kja/c)
		 \sum_{\substack{ \ell \le t/(dk) \\ P^+(\ell)\le y,\,\ell\equiv j\mod c }} 1  \\
	 &\quad +   O\left( \frac{tb}{\phi(b)} 
	 	\sum_{\substack{t/(\log y)^{300}<k\le t\\ P^+(k)\le y }} \frac{|h(k)|}{k}\right) 
}
where we bounded trivially the sum over $\ell$ when $k>t/(d\ell_0)$ (note that $d\ell_0\le(\log y)^{300}$ for all $d|b$).
		
Next, we need an estimate for the sum
\[
\sum_{\substack{ \ell \le t/(dk) \\ P^+(\ell)\le y,\,\ell\equiv j\mod c }} 1 
\]
when $dk\le t/\ell_0$. Note that, since $b\le(\log y)^{100}$, Theorems 2 and 5 in \cite{FT} imply that
\[
 \sum_{\substack{ \ell \le t/(dk) \\ P^+(\ell)\le y,\,\ell\equiv j\mod c }} 1
 	 =   \frac{1}{\phi(c)} \sum_{\substack{ \ell \le t/(dk) \\ P^+(\ell)\le y,\,(\ell,c)=1}} 1
		+  O\left( \frac{t}{dk\phi(c) (\log y)^5} \right) 
\]
when $t/(dk)\ge y$; the same result also holds when $t/(dk)\le y$ by elementary techniques since $t/(dkc)\ge \ell_0/c\ge (\log y)^{100}$ for $c|b$ and $dk\le t/\ell_0$. So, using the identity
\[
\sum_{\substack{1\le j\le c \\ (j,c)=1}} e(kja/c)  = \mu(c),
\]
we deduce that
\als{
S(t) &= \sum_{b=cd} \frac{\mu(c)\chi(d) }{\phi(c)}
		\sum_{\substack{ k\le t/(d\ell_0) \\ P^+(k)\le y,\,(k,c)=1 }} h(k)
		\sum_{\substack{ \ell \le x/(dk) \\ P^+(\ell)\le y,\,(\ell,c)=1}} 1  \\
	&\quad+ O\left(\frac{t}{(\log y)^2} +  \frac{tb}{\phi(b)} \sum_{\substack{t/(\log y)^{300}<k\le t\\ P^+(k)\le y }} 
		\frac{|h(k)|}{k}\right) ,
}
where we used \eqref{h bound}. We get the same right side no matter what the value of $a$, as long as $(a,b)=1$. Hence
\[
S(t) = \frac 1{\phi(b)}
	\sum_{\substack{ 1\le r\le b \\ (r,b)=1}} \sum_{\substack{ n\le t \\ P^+(n)\le y }} \chi(n) e(nr/b)
	+ R(t) 
\]
for some function $R(t)$ satisfying the bound
\[
R(t) \ll \frac{t}{(\log y)^2} +  \frac{tb}{\phi(b)} \sum_{\substack{t/(\log y)^{300}<k\le t\\ P^+(k)\le y }} 
		\frac{|h(k)|}{k} .
\]
Letting $d=(n,b)$, and writing $n=md$ and $b=cd$ so that $(m,c)=1$, we deduce that
\[
S(t) =  \sum_{b=cd} \frac{\mu(c)\chi(d)}{\phi(c)}
		\sum_{\substack{ m\le x/d \\ P^+(m)\le y,\, (m,c)=1 }} \chi(m)   + R(t) 
\]
by \cite[eq (7), p. 149]{Dav}, for all $t\in [(\log y)^{400},z]$. Therefore partial summation implies that
\als{
\sum_{\substack{ n\le z \\ P^+(n)\le y }} \frac{\chi(n) e(na/b)}{n}
	&= \sum_{\substack{(\log y)^{400}<n\le z \\ P^+(n)\le y }} \frac{\chi(n) e(na/b)}{n} 
		+ O(\log\log y) \\
	&= \sum_{b=cd} \frac{\mu(c)\chi(d)}{\phi(c)d}
		\sum_{\substack{ (\log y)^{400}/d<m\le z/d,\, (m,c)=1  \\ P^+(m)\le y }} \frac{\chi(m)}{m} \\
	&\quad	 + \int_{(\log y)^{400}}^z \frac{dR(t)}{t} + O(\log\log y) .
}
Integrating by parts and applying relation \eqref{h bound} and our assumption that $z\le y^{\log\log y}$ we conclude that
\als{
 \int_{(\log y)^{400}}^z \frac{dR(t)}{t}  
 	\ll 1 + \frac{b\log\log y}{\phi(b)} \sum_{\substack{k>1 \\ P^+(k)\le y }} \frac{|h(k)|}{k}  
		\le 1+ \frac{b\log\log y}{\phi(b)} (e^\Delta-1)  .
}
Since we also have that
\[
\sum_{b=cd} \frac{\mu(c)\chi(d)}{\phi(c)d}
		\sum_{\substack{ m\in[1,(\log y)^{400}/d]\cup(z/d,z] \\ (m,c)=1,\, P^+(m)\le y }} \frac{\chi(m)}{m}
	\ll\sum_{b=cd} \frac{ \log\log y }{\phi(c)d} \ll\log\log y,
\]
we deduce that
\eq{remove exp-final}{
\sum_{\substack{ n\le z \\ P^+(n)\le y }} \frac{\chi(n) e(na/b)}{n}
	= \sum_{b=cd} \frac{\mu(c)\chi(d)}{\phi(c)d}
		\sum_{\substack{ m\le z,\, (m,c)=1  \\ P^+(m)\le y }} \frac{\chi(m)}{m}
		 + O(E) .
}
Applying Lemma \ref{change gcd} with $f(n) = \chi(n){\bf 1}_{(n,c)=1}{\bf 1}_{P^+(n)\le y}$ and $b$ in place of $a$ implies that
\[
\sum_{\substack{ m\le z,\, (m,c)=1  \\ P^+(m)\le y }} \frac{\chi(m)}{m}
	=   \prod_{p|b,\,p\nmid c }\left(1-\frac{\chi(p)}{p}\right)^{-1}
		\sum_{\substack{ n\le z,\,(n,b)=1 \\ P^+(n)\le y  }} \frac{\chi(n)}{n} 
			+ O\left(\frac{b}{\phi(b)}\sum_{p|b}\frac{\log p}{p}\right) .
\]
Inserting this formula into \eqref{remove exp-final} leads to an error term of size
\[
\ll E+ \sum_{b=cd} \frac{1}{\phi(c)d} \cdot \frac{b}{\phi(b)}\sum_{p|b}\frac{\log p}{p}
	\ll E,
\]
and a main term of
\[
  \prod_{p|b }\left(1-\frac{\chi(p)}{p}\right)^{-1}
		\sum_{\substack{ n\le z,\,(n,b)=1 \\ P^+(n)\le y  }} \frac{\chi(n)}{n}
\]
times
\[
\sum_{b=cd} \frac{\mu(c)\chi(d)}{\phi(c)d} \prod_{p| c}\left(1-\frac{\chi(p)}{p}\right) =
\frac{ \mu(q_1)}{\phi(q_1)} \frac{ \chi(b_1)}{\phi(b_1)} \prod_{p|b_1} (1-\overline{\chi}(p)) ,
\]
thus completing the proof of the lemma.
\end{proof}


\begin{cor}\label{remove exp cor} Let $q$ be an integer that either equals 1 or is prime. Let $z,y\ge 2$, and $a/b$ be a reduced fraction with $1< b\le (\log y)^{100}$. Then
\[
\sum_{\substack{ n\le z \\ P^+(n)\le y  \\ (n,q)=1}} \frac{ e(na/b)}{n}
	= -\frac{{\bf 1}_{b=q}}{\phi(q)}\sum_{\substack{ n\le z \\ P^+(n)\le y  \\ (n,q)=1}} \frac{1}{n} 
		+O\left( \left(1+\frac{{\bf 1}_{q>1}}{q} \frac{b}{\phi(b)}\right)\log\log y \right) .
\]
\end{cor}


\section{Structure of even characters with large $M(\chi)$: proof of Theorem \ref{structure thm even}}\label{even}

The goal of this section is to prove Theorem \ref{structure thm even}. Throughout this section, we set $y=e^{\sqrt{3}\tau+c}$ for some constant $c$. We will show this theorem with
\[
\CC_q^+(\tau) := \left\{\chi\mod q:\chi\neq \chi_0,\ \chi(-1)=1,\ S_{y,q^{11/21}}(\chi)\le 1,\ 
	m(\chi)>\tau \right\},
\]
where the quantity $S_{y,z}(\chi)$ is defined as in Section \ref{outline}. Theorem \ref{tail bound cor} and the lower bound in Theorem \ref{main thm even}, which we already proved in the beginning of Section \ref{outline} (independently of the proof of Theorem \ref{structure thm even}), guarantee that the cardinality of $\CC_q^+(\tau)$ satisfies  \eqref{cc+}, provided that the constant $c$ in the definition of $y$ and the constant $C$ in the statement of Theorem \ref{structure thm even} are large enough. 

We fix a large $\tau\le \log\log q$ and we consider a character $\chi\in \CC_q^+(\tau)$. Let $\alpha=N_\chi/q$. Then
\als{
m(\chi)
	= \frac{1}{2e^\gamma} \abs{ \sum_{1\le |n|\le q^{11/21}} \frac{\chi(n)e(n\alpha )}{n} } 
		+ O(q^{-1/43}) 
	&= \frac{1}{2e^\gamma} \Bigg| \sum_{\substack{1\le |n|\le q^{11/21} \\  P^+(n)\le y}} \frac{\chi(n)e(n\alpha)}{n} \Bigg| 
		+ O(1) \\
	&= \frac{1}{2e^\gamma} \Bigg| \sum_{\substack{n\in\Z\setminus\{0\} \\  P^+(n)\le y}} \frac{\chi(n)e(n\alpha)}{n} \Bigg| + O(1),
}
by  \eqref{polya}, our assumption that $|S_{y,q^{11/21}}(\chi)|\le 1$ for $\chi\in\CC^+_q(\tau)$, and Lemma \ref{smooth}. As in the statement of Theorem \ref{structure thm even}, we approximate $\alpha$ by a reduced fraction $a/b$ with $b\le \tau^{10}$ and $|\alpha-a/b|\le 1/(b\tau^{10})$. We let $N=1/|b\alpha-a|\ge\tau^{10}$ and apply Lemma \ref{centering alpha} with $z=\infty$ to find that
\eq{even e0}{
m(\chi)
	= \frac{1}{2e^\gamma}  \Bigg|\sum_{\substack{1\le |n|\le N \\  P^+(n)\le y}} \frac{\chi(n)e(an/b)}{n} \Bigg|
		+ O(\log \tau) .
}
Since $\chi\in\CC_q^+(\tau)$, we must have that $m(\chi)> \tau$, which implies that
\eq{even e1}{
\Bigg| \sum_{\substack{1\le |n|\le N \\  P^+(n)\le y}} \frac{\chi(n)e(an/b)}{n} \Bigg| 
	\ge 2e^\gamma \tau - O(\log \tau).
}
We now proceed to show that $\chi$ satisfies properties (1) and (2) of Theorem \ref{structure thm even}.


\begin{proof}[Proof of Theorem \ref{structure thm even} - Property (1)]
We choose $\xi\mod D$ with $D\le \log y$ to satisfy  \eqref{most pretentious}. We claim that 
\eq{xi=leg mod 3}{
b=D=3\quad \text{and}\quad\ \xi=\leg{\cdot}{3} ,
}
the first claim being equivalent to $a/b\in\{1/3,2/3\}$. 

Firstly, note that Lemma \ref{MV} in conjunction with  \eqref{even e1} implies that $b\ll1$. Equation \eqref{even e1} also tells us that we must be in the second case of Lemma \ref{formula2} (provided that $\tau$ is large enough), that is to say $D|b$ and $\chi\bar{\xi}$ is odd. Since $\chi$ is even, we conclude that $\xi$ is odd. Thus $3\le D\le b\ll1 $. Moreover, the last inequality in Lemma \ref{formula2} implies that
\[
\Bigg| sum_{\substack{1\le|n|\le N \\ P^+(n)\le y}} \frac{\chi(n)e(an/b)}{n} \Bigg|
	\le \frac{2e^\gamma \sqrt{3}\tau}{\sqrt{D}} \cdot (2/3)^{\omega(b/D)} +O(\tau^{0.86}) .
\]
Comparing this inequality with \eqref{even e1}, we deduce that $b=D=3$ and thus $\xi=(\cdot/3)$, the quadratic character modulo 3, which completes the proof of our claim and hence of the fact that Property (1) holds.   
\end{proof}


\begin{proof}[Proof of Theorem \ref{structure thm even} - Property (2)]
We start with the proof of  \eqref{M(chi) even}. Note that
\als{
\sum_{ \substack{1\le |n|\le N \\ P^+(n)\le y}}\frac{\chi(n)e(an/3)}{n}
	= \sum_{\substack{n\le N \\ P^+(n)\le y}}\frac{\chi(n)(e(an/3)-e(-an/3))}{n} ,
}
where $a\in\{1,2\}$. Then using  \eqref{leg 3}, we find that
\[
\sum_{ \substack{|n|\le N \\ P^+(n)\le y}} \frac{\chi(n)e(an/3)}{n}
	= i\sqrt{3}\leg{a}{3} \sum_{ \substack{ n\le N \\ P^+(n)\le y}} \frac{\chi(n) \leg{n}{3}}{n} ,
\]
so that, by \eqref{even e0},
\[
m(\chi) = \frac{\sqrt{3}}{2e^\gamma} \Bigg| \sum_{ \substack{ n\le N \\ P^+(n)\le y}} \frac{\chi(n) \leg{n}{3}}{n} \Bigg| 
	+O(\log \tau) .
\]
Hence, we conclude that
\als{
2e^\gamma\tau-O(\log\tau) 
	&\le \sqrt{3} \, \Bigg| \sum_{ \substack{ n\le N \\ P^+(n)\le y}} \frac{\chi(n) \leg{n}{3}}{n} \Bigg|
		\le \sqrt{3}\sum_{ \substack{ n\le N,\,(n,3)=1\\ P^+(n)\le y}} \frac{1}{n}  \\
	&= \sqrt{3} \Bigg( \frac{2}{3} e^\gamma\log y + O(1)
		-  \sum_{ \substack{ n> N,\,(n,3)=1\\ P^+(n)\le y}} \frac{1}{n} \Bigg) \\
	&= 2e^\gamma \tau + O(1) - \sqrt{3} \sum_{ \substack{ n > N,\,(n,3)=1\\ P^+(n)\le y}} \frac{1}{n} .
}
Therefore
\[
\sum_{ \substack{ n>N,\,(n,3)=1\\ P^+(n)\le y}} \frac{1}{n} \ll \log \tau,
\]
which, in turn, implies that
\[
\sum_{ \substack{ n\le N \\ P^+(n)\le y}} \frac{\chi(n) \leg{n}{3}}{n}
	= \sum_{P^+(n)\le y } \frac{\chi(n) \leg{n}{3}}{n}  + O(\log \tau) ,
\]
so that
\eq{M(chi) even, y-smooth}{
m(\chi) = \frac{\sqrt{3}}{2e^\gamma} \abs{ \sum_{P^+(n)\le y} \frac{\chi(n) \leg{n}{3}}{n} } 
	+O(\log \tau) .
}
Finally, we have that
\als{
L\left(1,\chi \leg{\cdot}{3}\right) 
	&= \sum_{n\le q^{11/21}} \frac{\chi(n)\leg{n}{3}}{n} + O(q^{-1/43})  \\
	&= \sum_{\substack{n\le q^{11/21} \\ P^+(n)\le y}} \frac{\chi(n)\leg{n}{3}}{n} 
		+ \frac{1}{i\sqrt{3}}\sum_{\substack{n\le q^{11/21} \\ P^+(n)>y}} \frac{\chi(n) (e(n/3) - e(-n/3) )}{n} 
		+ O(q^{-1/43})  \\
	&= \sum_{P^+(n)\le y} \frac{\chi(n)\leg{n}{3}}{n} +O(\log\tau),
}
by the P\'olya--Vinogradov inequality, our assumption that $\chi\in\CC_q^+(\tau)$ and Lemma \ref{smooth}. Inserting the above estimates into \eqref{M(chi) even, y-smooth} completes the proof of  \eqref{M(chi) even}.  

\medskip

Finally, we prove  \eqref{mod3-pretentious}. For convenience, we set $\psi(n)=\chi(n)\leg{n}{3}$. Then  \eqref{M(chi) even, y-smooth} and our assumption that $M(\chi)>\frac{e^\gamma}{\pi} \tau\sqrt{q}$ for $\chi\in\CC_q^+(\tau)$ imply that
\[
\abs{\sum_{P^+(n)\le y} \frac{\psi(n)}{n} }\ge\frac{2}{3} e^\gamma\log y + O(\log\tau)
	= \sum_{\substack{P^+(n)\le y\\(n,3)=1}} \frac{1}{n} + O(\log\tau) .
\]
Since this lower bound is also  an upper bound, we deduce that
\[
\abs{\prod_{\substack{p\le y \\ p\neq 3}}\left(1-\frac{\psi(p)}{p}\right)^{-1}\left(1-\frac{1}{p}\right) } 
	= 1+ O\left(\frac{\log \tau}{\tau}\right) .
\]
Then, the argument leading to  \eqref{chi close to 1} implies that
\eq{strong prop3}{
\sum_{\substack{p\le y \\ p\neq 3}} \sum_{j=1}^\infty\frac{|1-\psi^j(p)|}{p^j} \ll \frac{\log \tau}{\sqrt{\tau}} ,
}
thus completing the proof of slightly more than Property (2).
\end{proof}


\begin{proof}[Proof of Theorem \ref{structure thm even} - Property (3)] Define $w$ via the relation $|\beta-k/\ell|=1/(\ell y^w)$. Note that $w=u(1+O(1/\tau))$ as $y=e^{\sqrt{3}\tau+c}$. So we may show the theorem with $w$ in place of $u$. Arguing as at the beginning of Section \ref{even}, and applying Lemma \ref{centering alpha} with $z=\infty$, we find that
\als{
\frac{\pi }{\CG(\chi)}\sum_{n\le\beta q} \chi(n)
	&= \frac{-1}{2i} \sum_{\substack{n\in\Z,\, n\neq 0 \\ P^+(n)\le y}}
	\frac{\bar{\chi}(n)e(-\beta n)}{n} + O(1) \\
	&= \frac{-1}{2i}   \sum_{\substack{1\le |n|\le y^w \\ P^+(n)\le y}}
			\frac{\bar{\chi}(n)e(-k n/\ell)}{n} + O(\log \tau) \\
	&= \sum_{\substack{n\le y^w \\ P^+(n)\le y}} \frac{\bar{\chi}(n)\sin(2\pi kn/\ell)}{n} + O(\log \tau) .
}
We note that inequality \eqref{strong prop3} and the argument leading to \eqref{d(chi,1) main thm} imply that
\eq{d(psi,1)}{
\sum_{\substack{P^+(n)\le y \\ (n,3)=1}} \frac{|1-\psi(n)|}{n}  \ll \sqrt{\tau}  \log\tau ,
}
where we have set $\psi(n)=\chi(n)\leg{n}{3}$. We write $n=3^jm$ with $(m,3)=1$, so that
\als{
\frac{\pi }{\CG(\chi)} \sum_{n\le\beta q} \chi(n)
	&=  \sum_{j=0}^\infty \frac{\bar{\chi}(3^j)}{3^j}
	\sum_{\substack{P^+(m)\le y \\ m\le y^w/3^j \\ (m,3)=1}} \frac{\bar{\chi}(m)\sin(2\pi 3^jkm/\ell)}{m} 
	 		+ O(\log \tau) \\
	&= \sum_{j=0}^\infty \frac{\bar{\chi}(3^j)}{3^j}
	\sum_{\substack{P^+(m)\le y \\ m\le y^w}} \frac{\leg{m}{3}\sin(2\pi 3^jkm/\ell)}{m} 
	 + O(\sqrt{\tau}\log \tau) ,
}
by \eqref{d(psi,1)} and as $\sum_{y^w/3^j<m\le y^w}1/m\ll j$. Using the formula $2\sin(2\pi m/3) = \sqrt{3} \leg{m}{3}$, we deduce that
\als{
\frac{\pi}{\CG(\chi)}\sum_{n\le\beta q} \chi(n) 
	&= \frac{1}{\sqrt{3}} \sum_{j=0}^\infty \frac{\bar{\chi}(3^j)}{3^j}
	\sum_{\substack{P^+(m)\le y \\ m\le y^w}} \frac{2\sin(2\pi m/3)\sin(2\pi 3^jkm/\ell)}{m} 
	 + O(\sqrt{\tau}\log \tau) \\
	&=  \frac{1}{\sqrt{3}} \sum_{j=0}^\infty \frac{\bar{\chi}(3^j)}{3^j}
	\sum_{\substack{P^+(m)\le y \\ m\le y^w}} \frac{ \cos(2\pi m\frac{3^{j+1}k-\ell}{3\ell}) - \cos(2\pi m\frac{3^{j+1}k+\ell}{3\ell}) }{m} 
	 + O(\sqrt{\tau}\log \tau) .
}
If $\ell$ is not a power of 3, then $\frac{3^{j+1}k\pm\ell}{3\ell}\notin\Z$ for each $j$, so Corollary \ref{remove exp cor}  implies that 
\[
\frac{\pi}{\CG(\chi)} \sum_{n\le\beta q} \chi(n) \ll \sqrt{\tau}\log \tau
\]
as claimed. Finally, if $\ell=3^v$ and $\epsilon\in\{-1,1\}$, then $\frac{3^{j+1}k-\epsilon\ell}{3\ell}\notin\Z$, unless $j=v-1$ and $k\equiv \epsilon\mod3$, so that $\epsilon=\leg{k}{3}$. Therefore Corollary \ref{remove exp cor} and Lemma \ref{smooth2} imply that 
\als{
\frac{\pi}{\CG(\chi)} \sum_{n\le\beta q} \chi(n) 
	&=  \frac{\leg{k}{3}}{\sqrt{3}} \frac{\bar{\chi}(3^{v-1})}{3^{v-1}}
		\sum_{\substack{P^+(m)\le y \\ m\le y^w}} \frac{1}{m} + O(\sqrt{\tau}\log \tau)  \\
	&= \frac{\leg{k}{3}\bar{\chi}(3^{v-1})}{3^{v-1}}  e^\gamma\tau P(w) 
		+ O(\sqrt{\tau}\log \tau).
}
This completes the proof of Theorem \ref{structure thm even}.
\end{proof}


\section{The structure of characters with large $M(\chi)$: proof of Theorems \ref{structure thm} and \ref{partial sums odd}}\label{odd}

The goal of this section is to prove Theorems \ref{structure thm} and \ref{partial sums odd}. Throughout this section, we set $y=e^{\tau+c}$ for some constant $c>0$ (note that this a different value of $y$ than in the previous section). We will show this theorem with
\[
\CC_q(\tau) := \left\{\chi\mod q:\chi\neq \chi_0,\ S_{y,q^{11/21}}(\chi)\le 1,\ 
	m(\chi) > \tau  \right\},
\]
where the quantity $S_{y,z}(\chi)$ is defined as in Section \ref{outline}. Theorem \ref{tail bound cor} and the lower bound in Theorem \ref{main thm}, which we already proved in the beginning of Section \ref{outline}, guarantee that the cardinality of $\CC_q(\tau)$ satisfies  \eqref{cc}, provided that the constant $c$ in the definition of $y$ and the constant $C$ in the statement of Theorem \ref{structure thm} are large enough.

We fix a large $\tau\le \log\log q$ and we consider a character $\chi\in \CC_q(\tau)$. Let $\alpha\in[0,1)$ be such that
\[
M(\chi) = \abs{ \sum_{n\le \alpha q} \chi(n) }.
\]
As in the statement of Theorem \ref{structure thm}, we pick a reduced fraction $a/b$ such that $1\le b\le \tau^{10}$ and $|\alpha-a/b|\le1/(b\tau^{10})$, and we define $b_0$ to equal $b$ if $b$ is prime, and 1 otherwise. Following the argument leading to \eqref{even e0}, we deduce that
\eq{odd e0}{
m(\chi)
	= \frac{1}{2e^\gamma}  \Bigg|
		 \sum_{\substack{ n\in\Z\setminus\{0\} \\  P^+(n)\le y}} \frac{\chi(n)}{n}  
		 	- \sum_{\substack{1\le |n|\le N \\  P^+(n)\le y}} \frac{\chi(n)e(an/b)}{n} \Bigg|
		+ O(\log \tau) ,
}
where $N=1/|b\alpha -a|\ge \tau^{10}$.


\begin{proof}[Proof of Theorem \ref{structure thm} - Property (1)]
We choose $\xi\mod D$ with $D\le \log y$ to satisfy  \eqref{most pretentious}. Since $\chi\in\CC_q(\tau)$, we must have that $m(\chi)> \tau$, which implies that
\eq{odd e1}{
\Bigg| \sum_{\substack{ n\in\Z\setminus\{0\} \\  P^+(n)\le y}} \frac{\chi(n)}{n}  
		 	- \sum_{\substack{1\le |n|\le N \\  P^+(n)\le y}} \frac{\chi(n)e(an/b)}{n} \Bigg|
	\ge 2e^\gamma \tau - O(\log \tau).
}
We claim that
\eq{xi=1}{
\chi\ \text{is odd},\quad D=1,\quad \xi=1 ,\quad \text{and}\quad   \D(\chi,1;y) \ll 1 ,
}
the first relation being Property (1). We separate two cases.

First, assume that $b\ge \tau^{1/100}$. Then we apply Lemma \ref{MV} to find that
\[
\sum_{\substack{1\le |n|\le N \\  P^+(n)\le y}} \frac{\chi(n)e(an/b)}{n} 
	\ll \tau^{1-1/300} ,
\]
which, together with \eqref{odd e1}, implies that
\eq{odd e100}{
|1-\chi(-1)|\abs{ \sum_{P^+(n)\le y} \frac{\chi(n)}{n}  } \ge 2e^\gamma\tau  - O(\tau^{299/300}) .
}
Then we must have that $\chi(-1)=-1$. Furthermore, the first part of Lemma \ref{formula2} implies that $D=1$ and thus $\xi=1$. Finally, Lemma \ref{general ub} and  \eqref{odd e100} imply that $\D(\chi,1;y)\ll1$, which completes the proof of \eqref{xi=1} in this case.

Finally, assume that $b\le \tau^{1/100}$. Suppose that either $\chi$ is even or $\xi\neq1$. Then Lemma \ref{formula2} implies that
\[
\sum_{\substack{ n\in\Z\setminus\{0\} \\  P^+(n)\le y}} \frac{\chi(n)}{n}  \ll \tau^{0.86} .
\]
So \eqref{odd e1} becomes
\eq{odd e2}{
\abs{ \sum_{\substack{1\le |n|\le N \\  P^+(n)\le y}} \frac{\chi(n)e(an/b)}{n} } 
	\ge 2e^\gamma \tau - O(\tau^{0.86}) .
}
Thus we must be in the second case of Lemma \ref{formula2} as far as the above sum is concerned, that is to say that $D|b$ and $\chi\bar{\xi}$ is odd. Then the second part of Lemma \ref{formula2} implies that
\[
\abs{\sum_{\substack{1\le|n|\le N \\ P^+(n)\le y}} \frac{\chi(n)e(an/b)}{n}}
	\le  \frac{2e^\gamma \tau}{\sqrt{D}} \cdot (2/3)^{\omega(b/D)} +O(\tau^{0.86}) .
\]
Comparing this inequality with \eqref{odd e2}, we deduce that $b=D=1$ and thus $\xi=1$, provided that $\tau$ is large enough. But then $\chi$ has to be odd, which contradicts our initial assumption. So we conclude that our initial assumption must be wrong, that is to say $\chi$ must be odd and $\xi=1$, so that $D=1$. 

It remains to show that $\D(\chi,1;y)\ll1$ in the case when $b\le \tau^{1/100}$. We apply Lemma \ref{general ub} and the second part of Lemma \ref{formula2}    to deduce that
\als{
\sum_{\substack{ 1\le |n| \le N \\ P^+(n)\le y}} \frac{\chi(n)}{n}
	-  \sum_{\substack{ 1\le |n| \le z \\ P^+(n)\le y}} \frac{\chi(n)e(an/b)}{n}
		\ll  \tau \exp\{ - \D^2(\chi,1;y)/2\}  + \tau^{0.86} .
}
Combining the above inequality with  \eqref{odd e1} yields the estimate $\D(\chi,1;y)\ll1$, thus completing the proof of our claim \eqref{xi=1} and, consequently, of Property (1).
\end{proof}


In order to prove Property (2) in Theorem \ref{structure thm}, we need an intermediate result: we set
\[
L_d(\chi)= \sum_{\substack{ n\ge 1 ,\, (n,d)=1 \\ P^+(n)\le y}} \frac{\chi(n)}{n} 
\]
for $d\in\N$, and
\[
\Delta=\sum_{p\le y} \frac{|1-\chi(p)|}{p-1}   \ll \D(\chi,1;y) \sqrt{\log\tau} \ll\sqrt{\log\tau}
\]
as in Lemma \ref{remove exp}, where we used \eqref{xi=1}. Additionally, we set
\[
E= \left(1+\frac{b}{\phi(b)}(e^\Delta-1) \right) \log\tau = \tau^{o(1)} \quad(\tau\to\infty),
\]
and we write $L_d(\chi)=L_d^{(1)}(\chi)+L_d^{(2)}(\chi)$, where
\[
L_d^{(1)}(\chi)=\sum_{\substack{ n\le N ,\, (n,d)=1 \\ P^+(n)\le y}} \frac{\chi(n)}{n}
\quad\text{and}\quad
L_d^{(2)}(\chi)= \sum_{\substack{ n>N,\, (n,d)=1 \\ P^+(n)\le y}} \frac{\chi(n)}{n}  .
\]
The intermediate result we need to show is that
\eq{final claim}{
m(\chi)
	= e^{-\gamma}\cdot \begin{cases}
		\ds\abs{L_1(\chi)} + O(E)  						&\mbox{if $b$ is not a prime power},\\
		\ds\abs{L_1(\chi)} + O\left(\sqrt{\tau E} \right)  					&\mbox{if $b=p^e,\ e\ge 2$}, \\
		\ds \frac{b}{\phi(b)} \abs{L_b(\chi)} + O\left( \sqrt{\tau E} \right)  		&\mbox{if $b$ is prime}.
	\end{cases}
}
Before proving this, we show how to use it to complete the proof of Theorem \ref{structure thm}.


\begin{proof}[Proof of Theorem \ref{structure thm} - Property (2)]
We argue as in the proof of Property (2) in Theorem \ref{structure thm even}. For  \eqref{M(chi) odd}, note that our assumption that $\chi\in\CC_q(\tau)$ implies, as $S_{y,q^{11/21}}(\chi)\le 1$, that, for all $d\le \tau^{10}$,
\als{
\sum_{\substack{P^+(n)> y \\ (n,d)=1 \\  n\le q^{11/21} }} \frac{\chi(n)}{n} 
	&=\sum_{g|d} \frac{\mu(g)\chi(g)}{g} \sum_{\substack{P^+(mg)> y \\  m\le q^{11/21}/g}} \frac{\chi(m)}{m} \\
	&=\sum_{g|d} \frac{\mu(g)\chi(g)}{g} \sum_{\substack{P^+(m)> y \\  m\le q^{11/21}}} \frac{\chi(m)}{m} 
		+ O\left(\sum_{g|d} \frac{\mu^2(g)\log g}{g} \right) \\
	&\ll \sum_{g|d} \frac{\mu^2(g)\log(2g)}{g} \ll \frac{d}{\phi(d)}\left(1+\sum_{p|d}\frac{\log p}{p}\right)
		\ll (\log\log\tau)^2 .
}
Therefore,
\als{
\sum_{(n,d)=1} \frac{\chi(n)}{n} 
	= \sum_{\substack{(n,d)=1 \\  n\le q^{11/21} }} \frac{\chi(n)}{n} +O(q^{-1/43}) 
	&= \sum_{\substack{P^+(n)\le y \\ (n,d)=1 \\  n\le q^{11/21} }} \frac{\chi(n)}{n} +O( (\log\log\tau)^2)  \\
	&= L_d(\chi) +O( (\log\log\tau)^2) 
}
for all $d\le \tau^{10}$, by the P\'olya--Vinogradov inequality, and Lemma \ref{smooth}. So  \eqref{M(chi) odd} follows from  \eqref{final claim} but with the weaker error term $O(E_1)$ in place of $O(\sqrt{\tau}\log\tau)$, where $E_1=\sqrt{E\tau}$ if $b=p^e$, and $E_1=E$ otherwise, so that $E_1\ll  \tau^{1/2+o(1)}$. We argue much like we did getting to to  \eqref{chi close to 1}: \ We have
\[
m(\chi) = 
			 \frac{b_0}{e^\gamma\phi(b_0)} 
			 \Bigg| \prod_{\substack{p\neq b_0\\ p\le y}} \left(1-\frac{\chi(p)}{p}\right)^{-1} 
				+  O\left(E_1\right)   \Bigg| 
				\ge e^{-\gamma} \prod_{p\le y} \left( 1 -\frac 1p \right)^{-1},
\]
so that 
\[
\Bigg| \prod_{\substack{p\neq b_0\\ p\le y}} \left(1-\frac{\chi(p)}{p}\right)^{-1} \left( 1 -\frac 1p \right)\Bigg|
	=1+ O(E_1/\tau) 
\]
and therefore
\[
\sum_{\substack{p\le y \\ p\neq b_0}} \frac{1-\Re(\chi(p))}{p} \ll \frac{E_1}\tau.
\]
This implies that
\[
\sum_{\substack{p\le y \\ p\neq b_0}} \frac{|1-\chi(p)|}{p-1} \ll \left( \frac{E_1\log \tau}\tau \right)^{1/2}
\]
In particular, $\Delta = \sum_{p\le y}|1-\chi(p)|/(p-1)\ll 1$, so that we always have that $E\ll \frac{b}{\phi(b)}\log \tau$ and $E_1\ll \sqrt{\tau\log \tau}$. This completes the proof of Property (2) in Theorem \ref{structure thm}.
\end{proof}


\begin{proof}[Proof of  \eqref{final claim}]
We separate two main cases.

\medskip

{\bf Case 1.}  Assume that $b=1$. Then  \eqref{odd e0} and the fact that $\chi$ is odd imply that
\eq{formula b=1}{
m(\chi)
	&= e^{-\gamma}\, \Bigg| \sum_{ P^+(n)\le y}  \frac{\chi(n)}{n}
		- \sum_{\substack{ n\le N \\ P^+(n)\le y}} \frac{\chi(n) }{n}  \Bigg| 
			+ O(\log \tau)\\
	&= e^{-\gamma} |L_1^{(2)}(\chi)| + O(\log \tau) .
}
Since $m(\chi)>\tau$, we find that
\[
e^\gamma\tau + O(\log \tau)
	\le  |L_1^{(2)}(\chi)|
		\le \sum_{\substack{n>N\\P^+(n)\le y}} \frac{1}{n}
		\le e^\gamma\tau   -  \sum_{\substack{ n\le N \\ P^+(n)\le y}} \frac{1}{n}  +O(\log\tau) .
\]
Consequently, $\sum_{n\le N,\,P^+(n)\le y} 1/n\ll \log \tau$, which in turn gives us that $L_1^{(2)}(\chi)=L_1(\chi)+O(\log \tau)$. Inserting this estimate into \eqref{formula b=1}, we deduce that
\als{
m(\chi)  = e^{-\gamma} |L_1(\chi)| + O(\log \tau),
}
that is to say, \eqref{final claim} holds (with a stronger error term).


\medskip

{\bf Case 2.} Assume that $1<b\le \tau^{10}$. Then \eqref{odd e0} and Lemma \ref{remove exp} implies that
\als{
m(\chi)
	&=  e^{-\gamma} \,\Bigg|\sum_{\substack{ n \le z \\ P^+(n)\le y}} \frac{\chi(n)}{n}
		-   \sum_{\substack{ n \le N \\ P^+(n)\le y}} \frac{\chi(n)e(na/b)}{n}  \Bigg| + O(\log \tau)\\
	&=e^{-\gamma}\,  \abs{L_1(\chi)
		-   L_b^{(1)}(\chi)  \frac{ \chi(b)}{\phi(b)} \prod_{p|b} (1-\overline{\chi}(p)) \left(1-\frac{\chi(p)}{p}\right)^{-1}}
		  + O( E) .
}
Now Lemma \ref{change gcd}, applied with $f(n)=\chi(n){\bf 1}_{P^+(n)\le y}$ and $b$ in place of $a$, implies that
\eq{change gcd 2}{
L_1^{(1)}(\chi)
	&= L_b^{(1)}(\chi)   \prod_{p|b }\left(1- \frac{\chi(p)}{p} \right)^{-1}
		 	+ O\bigg( \frac{b}{\phi(b)} \sum_{p|b} \frac{\log p}{p} \bigg) .
}
So, if we set
\eq{Cb def}{
C_b=
	\bigg( 1 - \frac{ \chi(b)}{\phi(b)} \prod_{p|b} (1-\overline{\chi}(p)) \bigg) 
		 \prod_{p|b }\left(1- \frac{\chi(p)}{p} \right)^{-1},
}
then we have that
\eq{fund formula}{
m(\chi)
	& = e^{-\gamma} \abs{L_1^{(2)}(\chi) + C_b L_b^{(1)}(\chi) } +  O( E)    \\
	& =e^{-\gamma}\, \bigg| L_1(\chi)
		-   L_1^{(1)}(\chi)  \frac{ \chi(b)}{\phi(b)} \prod_{p|b} (1-\overline{\chi}(p)) \bigg|
		  + O(E) .
}

Our goal is to show, using formula \eqref{fund formula}, that $\chi(c)\approx 1$ for all $c|b$. Indeed, if this were true, then $C_b\approx b/\phi(b) \approx \prod_{p|b} (1-\chi(p)/p)^{-1}$. We start by observing that 
\eq{Cb formula}{
C_b = \sum_{b=cd} \frac{\chi(d)}{d} \left( 1 - \frac{\mu(c)}{\phi(c)}\right) \prod_{p|b, \ p\nmid c}\left(1- \frac{\chi(p)}{p} \right)^{-1} .
}
Indeed, reversing the last steps of the proof of Lemma \ref{remove exp}, we find that
\[
\frac{\chi(b)}{\phi(b)} \prod_{p|b }\left(1- \frac{\chi(p)}{p} \right)^{-1}   (1-\overline{\chi}(p)) 
	= \sum_{b=cd} \frac{\chi(d)}{d}  \frac{\mu(c)}{\phi(c)} \prod_{p|b, \ p\nmid c}\left(1- \frac{\chi(p)}{p} \right)^{-1} .
\]
Moreover,
\als{
\prod_{p|b }\left(1- \frac{\chi(p)}{p} \right)^{-1}
 =  \sum_{p|n\ \Rightarrow\ p|b} \frac{\chi(n)}{n} 
 = \sum_{d|b} \sum_{\substack{p|n\ \Rightarrow\ p|b \\ (n,b)=d}} \frac{\chi(n)}{n}  
 &=\sum_{b=cd} \frac{\chi(d)}{d} \sum_{\substack{p|m\ \Rightarrow\ p|b \\ (m,c)=1}} \frac{\chi(m)}{m} \\
 &=\sum_{b=cd} \frac{\chi(d)}{d} \prod_{p|b,\,p\nmid c}\left(1-\frac{\chi(p)}{p}\right)^{-1} .
}
Combining the two above relations, we obtain \eqref{Cb formula}. 

Using \eqref{Cb formula} and the inequality $|1- \chi(p)/p |^{-1}=p/|p-\chi(p)| \le \frac p{p-1}$, we deduce that
\eq{C_b-ub}{
|C_b| \le \sum_{b=cd} \frac{1}{d} \left( 1 - \frac{\mu(c)}{\phi(c)}\right)
\frac{b/\phi(b)}{c/\phi(c)} = \frac{b}{\phi(b)}
}
for $b>1$. Inserting this into \eqref{fund formula}, and since 
\[
m(\chi)  >\tau =e^{-\gamma} \sum_{P^+(n)\le y} \frac{1}{n} +O(1),
\]
we obtain that
\eq{L1L2 ineq}{
\sum_{P^+(n)\le y}\frac{1}{n}
	&\le |L_1^{(2)}(\chi)| + |C_b| \cdot |L_b^{(1)}(\chi)|  +  O( E ) \\
	&\le |L_1^{(2)}(\chi)| + \frac{b}{\phi(b)} |L_b^{(1)}(\chi)|  +  O( E ) . 
}
We also have
\[
|L_b^{(1)}(\chi)|
	\le \sum_{ \substack{ n\le N,\,(n,b)=1 \\ P^+(n)\le y}} \frac{1}{n} 	
	 = \frac{\phi(b)}{b} \sum_{ \substack{ m\le N\\ P^+(m)\le y}} \frac{1}{m} 
	 	+ O\left(\frac{b}{\phi(b)}\sum_{p|b}\frac{\log p}{p}\right) 
\]
by Lemma \ref{change gcd}. Together with \eqref{L1L2 ineq}, this implies that
\[
|L_1^{(2)}(\chi)| \ge \sum_{ \substack{ N<n\le z,  \\ P^+(n)\le y}} \frac{1}{n} +  O( E)  
	=: S^{(2)} + O( E)  .
\]
Since this holds   as an upper bound too, we deduce that
\eq{L2}{
|L_1^{(2)}(\chi)| = S^{(2)} + O( E ) .
}
Substituting \eqref{L2} into \eqref{L1L2 ineq} we obtain
\eq{L1}{
\frac{b}{\phi(b)}|L_b^{(1)}(\chi)|
	= \sum_{ \substack{ n\le N \\ P^+(n)\le y}} \frac{1}{n} +  O( E ) 
	&=: S^{(1)} + O(E) 
}
and then comparing \eqref{L1L2 ineq} with the displayed line above,
\eq{Cb}{
\frac{\phi(b)}{b} |C_b|= 1 + O(\epsilon_1) ,
}
where we have set $\epsilon_j := E/S^{(j)}$ for $j\in\{1,2\}$. 

\medskip

{\bf Case 2a.} Assume that $b$ has at least two distinct prime factors. If $b\neq 6$, then we can find two distinct primes $p$ and $q$ such that $b=p^eq^fb'$ with $p,q\nmid b'$, $p^eq^f\neq6$, and $\phi(b')\ge2^{\omega(b')}$ (which only fails for $b'=2$ or $6$). Therefore, taking absolute values in \eqref{Cb def}, we find that
\[
\frac{\phi(b)}b |C_b|
	\le  \left( 1 + \frac{|1-\chi(p)|}{\phi(p^e)}\frac{|1-\chi(q)|}{\phi(q^f)} \right)
		\cdot  \left(1- \frac{1}{p} \right) \left|1- \frac{\chi(p)}{p} \right|^{-1} 
		\cdot  \left(1- \frac{1}{q} \right) \left|1- \frac{\chi(q)}{q} \right|^{-1} .
\]
Note that
\eq{euler factor ineq}{
\left(1- \frac{1}{\ell} \right) \left|1- \frac{\chi(\ell)}{\ell} \right|^{-1}
	= \exp\left( -\sum_k \frac{ \text{Re}(1-\chi(\ell^k))} {k \ell^k} \right)
	\le  \exp\left( - \frac{ \text{Re}(1-\chi(\ell))} {\ell} \right) .
}
Therefore
\[
\epsilon_1 \gg \frac{ \text{Re}(1-\chi(p))} {p} + \frac{ \text{Re}(1-\chi(q))} {q} 
	- \log \left( 1  +     \frac{|1- \chi(p)|}{\phi(p^e)} \cdot   \frac{|1- \chi(q)|}{\phi(q^f)} \right)  .
\]
Now $\phi(p^eq^f)\ge (8/7)\sqrt{pq}$ when $p^eq^f\ne 6$ so that, for $\alpha=\text{Re}(1-\chi(p))$ and $\beta=\text{Re}(1-\chi(q))$,
\[
\log \left( 1 +     \frac{|1- \chi(p)|}{\phi(p^e)} \cdot   \frac{|1- \chi(q)|}{\phi(q^f)} \right)
	\le \frac{|1- \chi(p)||1- \chi(q)|}{\phi(p^eq^f)}
	\le \frac{2\sqrt{\alpha\beta}}{(8/7)\sqrt{pq}}
	\le  \frac{7}{8} \left( \frac \alpha p + \frac \beta q   \right)
\]
since $|1- \chi(p)|^2=2\alpha$ and $|1- \chi(q)|^2=2\beta$. We therefore deduce that
\[
\frac{ \text{Re}(1-\chi(p))} {p} , \frac{ \text{Re}(1-\chi(q))} {q} \ll \epsilon_1 .
\]
Now $|1-\chi(p)|^2=2\text{Re}(1-\chi(p))$, and so 
\[
|1-\chi(p)| \cdot |1-\chi(q)| \ll \sqrt{pq}\ \epsilon_1  \le \sqrt{b}\ \epsilon_1 ,
\] 
and therefore, by \eqref{Cb def},
\[
C_b =  (1+O(\epsilon_1 ))    \prod_{p|b }\left(1- \frac{\chi(p)}{p} \right)^{-1} .
\]
Substituting this into \eqref{fund formula}, and using \eqref{change gcd 2}, yields \eqref{final claim} in this case, except if $b=6$.

When $b=6$, we use relations \eqref{Cb formula} and \eqref{euler factor ineq} to deduce that 
\[
\frac{\phi(b)}{b}|C_b| \le
	\frac{1}{3}\left( \frac{1}{2}+\frac{3}{2} \exp\left\{-\frac{\Re(1-\chi(2))}{2}\right\} 	
		+ \exp\left\{ - \frac{\Re(1-\chi(3))}{3}\right\}\right) .
\]
Since the left hand side is $\ge1+O(\epsilon_1)$ and $\Re(1-\chi(2)),\Re(1-\chi(3))\ge0$, we deduce that $\Re(1-\chi(2)),\Re(1-\chi(3))\ll\epsilon_1$, as before. Proceeding now as in the case $b\neq 6$ completes the proof \eqref{final claim} when $b=6$ too.



\medskip

{\bf Case 2b.} Now suppose that $b=p^e$ is a prime power with $e\ge 2$. Using \eqref{Cb formula}, we see that
\[
\frac{\phi(b)}{b} C_b = \left(1- \frac{1}{p} \right)  \sum_{j=0}^{e-2}  \frac{ \chi(p^j)}{p^j} 
	+ \frac{ \chi(p^{e-1})}{p^{e-1}} .
\]
Define $\lambda$ so that $|\lambda|=1$ and $\lambda C_b=|C_b|$. Then
\[
\left(1- \frac{1}{p} \right)  \sum_{j=0}^{e-2}  \frac{1- \lambda\chi(p^j)}{p^j} 
+ \frac{1- \lambda\chi(p^{e-1})}{p^{e-1}} =  1- \frac{\phi(b)}{b} |C_b|  
\]
which is $\ge 0$ and $O(\epsilon_1)$ by \eqref{C_b-ub} and \eqref{Cb}. Taking real parts, and noting that each
Re$(1- \lambda\chi(p^j))\ge 0$, we deduce that Re$(1- \lambda),$ Re$(1- \lambda\chi(p))/p=O(\epsilon_1)$
by considering the $j=0$ and $1$ terms. Hence 
\[
|1-\chi(p)| = |\lambda(1-\chi(p))|\le |1- \lambda|+|1- \lambda\chi(p)|\ll \sqrt{\epsilon_1p} 
\]
and therefore 
\[
C_b = (1+O(\epsilon_1^{1/2} ) )   \left(1- \frac{\chi(p)}{p} \right)^{-1}   .
\]
Substituting this into \eqref{fund formula}, and using \eqref{change gcd 2}, yields the result in this case.


\medskip

{\bf Case 2c.} Now suppose that $b=p$ is a prime, so that $C_p=p/(p-1)$. Hence we cannot use \eqref{Cb} to gain information on $\chi(p)$. Now, using  Lemma \ref{change gcd}, we have
\[
|L_1^{(2)}(\chi)| = \abs{ \frac{p}{p-\chi(p)} L_p^{(2)}(\chi) } + O\left(\frac{\log p}{p}\right)
	\le \frac{p}{ |p-\chi(p)| } \frac{p-1}{p} S^{(2)} + O\left(\frac{\log p}{p}\right) .
\]
Combining this with \eqref{L2}   yields that $(p-1)/|p-\chi(p)|\ge1+O(\epsilon_2)$; that is
\[
\abs{1+\frac{1-\chi(p)}{p-1} } \le 1+O(\epsilon_2),
\]
Taking real parts, we deduce that
\[
1\le 1+ \frac{\Re(1-\chi(p))}{p-1}  \le 1+O(\epsilon_2) ,
\]
with the lower bound being trivial. This implies that
\[
|1-\chi(p)| \ll \sqrt{\epsilon_2p} .
\]
Using Lemma \ref{change gcd}, we then conclude that
\als{
L_1^{(2)}(\chi) + C_b L_b^{(1)}(\chi)
	&= L_1^{(2)}(\chi) + \frac{p}{p-1} L_p^{(1)}(\chi) \\
	&= \frac{p}{p-\chi(p)} L_p^{(2)}(\chi) +  \frac{p}{p-1} L_p^{(1)}(\chi) + O\left(\frac{\log p}{p}\right)  \\
	&= \frac{p}{p-1} L_p^{(2)}(\chi) + \frac{p}{p-1} L_p^{(1)}(\chi) +  O\left(\sqrt{\tau E}\right) ,
}
which, together with \eqref{fund formula}, completes the proof of \eqref{final claim} in this last case too. 
\end{proof}

\begin{rk}
Note that $\epsilon_2$ can be quite big if $N$ is small and if $\alpha$ is not very close to $a/b$. So  we cannot say more than the last formula  without more information on the location of $\alpha$.
\end{rk}


We conclude the paper with the proof of Theorem \ref{partial sums odd}.


\begin{proof}[Proof of Theorem \ref{partial sums odd}]
Let $y=e^{\tau+c}$ for a large enough constant $c$, as above, and define $w$ via the relation $|\beta-k/\ell|=1/(\ell y^w)$. Note that $w=u(1+O(1/\tau))$. So we may show the theorem with $w$ in place of $u$. Arguing as at the beginning of this section, and applying Lemma \ref{centering alpha} with $z=\infty$, we find that
\als{
\frac{\pi i}{\CG(\chi)} \sum_{n\le\beta q} \chi(n) 
	&= \frac{1}{2} \sum_{\substack{n\in\Z,\, n\neq 0 \\ P^+(n)\le y}} 
	\frac{\bar{\chi}(n)}{n} - \frac{1}{2}
	 \sum_{\substack{P^+(n)\le y \\ 1\le|n|\le y^w}}\frac{\bar{\chi}(n)e(-kn/\ell)}{n}
		+ O(\log \tau) \\
	&= \sum_{P^+(n)\le y}\frac{\bar{\chi}(n)}{n} 
		-  \sum_{\substack{P^+(n)\le y \\ n\le y^w}}\frac{\bar{\chi}(n)\cos(2\pi kn/\ell)}{n}
		+ O(\log\tau) .
}
Note that \eqref{1-pretentious} and the argument leading to \eqref{d(chi,1) main thm} imply that
\eq{d(chi,1)}{
\sum_{\substack{P^+(n)\le y \\ (n,b_0)=1}} \frac{|1-\chi(n)|}{n} 
		\ll (\tau\log\tau)^{3/4}
}
for all $\epsilon>0$. Hence $\chi$ is $1$-pretentious.

Now suppose that  $b_0=1$. Substituting \eqref{d(chi,1)} in our formula for $\sum_{n\le\beta q}\chi(n)$, we obtain 
\[
\frac{\pi i}{\CG(\chi)} \sum_{n\le\beta q} \chi(n) = \sum_{P^+(n)\le y}\frac{1}{n} 
		-  \sum_{\substack{P^+(n)\le y \\ n\le y^w}}\frac{\cos(2\pi kn/\ell)}{n}
		+ O( (\tau\log\tau)^{3/4}) 
\]
If $\ell>1$, then we bound the second sum using Corollary \ref{remove exp cor}. If $\ell=1$, then the result follows from Lemma \ref{smooth2}. This concludes the proof of part (a).

Finally, assume that $b$ is a prime number so that $b_0=b$. Writing $n=b^jm$ with $(m,b)=1$, we find that
\als{
\frac{\pi i}{\CG(\chi)}\sum_{n\le\beta q} \chi(n) 
	&= \sum_{j=0}^\infty \frac{\bar{\chi}(b^j)}{b^j}
		\left(\sum_{\substack{P^+(m)\le y \\ (m,b)=1}}\frac{\bar{\chi}(m)}{m} 
			- \sum_{\substack{P^+(m)\le y \\ m\le y^w,\ (m,b)=1}}\frac{\bar{\chi}(m)\cos(2\pi b^jk m / \ell ) }{m} \right)
				+O(\log\tau) \\ 
	&=  \sum_{j=0}^\infty  \frac{\bar{\chi}(b^j)}{b^j}
		\left(\sum_{\substack{P^+(m)\le y \\ (m,b)=1}}\frac{1}{m} 
			- \sum_{\substack{P^+(m)\le y \\ m\le y^w,\ (m,b)=1}}\frac{\cos(2\pi b^jk m / \ell ) }{m} \right)
			+ O( (\tau\log\tau)^{3/4}) 
}
by \eqref{d(chi,1)} and the trivial estimate $\sum_{y^w/b^j<m\le y^w}1/m\ll j\log b\ll j\log \tau$. If $\ell=1$, then we have that 
\als{
\frac{\pi i}{\CG(\chi)}\sum_{n\le\beta q} \chi(n) 
	&= \sum_{j=0}^\infty  \frac{\bar{\chi}(b^j)}{b^j} \sum_{\substack{P^+(m)\le y \\ m>y^w,\ (m,b)=1}}\frac{1}{m} 
			 + O(\tau^{3/4}\log\tau ) \\
	&=e^\gamma \tau (1-P(w)) \cdot \frac{1-1/b}{1-\bar{\chi}(b)/b}
		+ O( (\tau\log\tau)^{3/4}) 
}
by Lemmas \ref{change gcd} and \ref{smooth2}. Next, if $\ell\neq b^v$ for all $v\ge0$, then $b^jk/\ell\notin\Z$ for all $j\ge0$. So Corollary \ref{remove exp cor}  implies that
\[
\frac{\pi i}{\CG(\chi)}\sum_{n\le\beta q} \chi(n) 
	= \sum_{j=0}^\infty  \frac{\bar{\chi}(b^j)}{b^j}
		\sum_{\substack{P^+(m)\le y \\ (m,b)=1}}  \frac{1}{m} 
			+ O(\tau^{3/4}\log\tau )
	= \frac{1-1/b}{1-\bar{\chi}(b)/b} e^\gamma\tau
		+ O( (\tau\log\tau)^{3/4}) 
\]
as claimed. Finally, assume that $\ell=b^v$ for some $v\ge1$. The terms with $j\ge v$ contribute
\[
\sum_{j=v}^\infty  \frac{\bar{\chi}(b^j)}{b^j} \sum_{\substack{P^+(m)\le y \\ m>y^w,\ (m,b)=1}}\frac{1}{m}  
	 = e^\gamma \tau (1-P(w)) \cdot \frac{1-1/b}{1-\bar{\chi}(b)/b} \cdot \frac{\bar{\chi}(b^v)}{b^v}
			 + O(1)
\]
by Lemmas \ref{change gcd} and \ref{smooth2}. When $j\le v-1$, we apply Corollary \ref{remove exp cor}. The total contribution of those terms is
\als{
	&\sum_{j=0}^{v-1}  \frac{\bar{\chi}(b^j)}{b^j} \sum_{\substack{P^+(m)\le y \\ (m,b)=1}} \frac{1}{m}
		+ \frac{\bar{\chi}(b^{v-1})}{b^{v-1}} \frac{1}{\phi(b)}
		 \sum_{\substack{P^+(m)\le y \\ m\le y^w,\, (m,b)=1}} \frac{1}{m}  
			 + O(\log\tau) \\
	&\quad= e^\gamma\tau \left(1-\frac{1}{b}\right)  \left( \frac{1-\bar{\chi}(b^{v})/b^{v}}{1-\bar{\chi}(b)/b}
		+  \frac{\bar{\chi}(b^{v-1})}{b^{v-1}} \cdot  \frac{P(w)}{b-1} \right)
		 + O(\log\tau)
}
by Lemmas \ref{change gcd} and \ref{smooth2}. Putting the above estimates together, yields the estimate
\[
\frac{e^{-\gamma}\pi i}{\CG(\chi)}\sum_{n\le\beta q} \chi(n) 
	= \tau \cdot \frac{1-\bar{\chi}(b)/b}{1-1/b} \cdot\left( 1 + P(u) \left(\frac{\bar{\chi}(b)}{b}\right)^{v-1} \frac{1-\bar{\chi}(b)}{b-1} + O(\epsilon)\right)
\]
with $\epsilon = (\log \tau)^{3/4}/\tau^{1/4}$. To finish the proof of the theorem, we specialize the above formula when $\beta=\alpha$, in which case $k/\ell=a/b$, so that $v=1$. Then the modulus of the left hand side equals $m(\chi)$, which is $>\tau$ by assumption. On the other hand, if we set $z=(1-\bar{\chi}(b))/(b-1)$, then
\[
\frac{1-\bar{\chi}(b)/b}{1-1/b} \cdot \left( 1 + P(u_0) \frac{1-\bar{\chi}(b)}{b-1}\right) 
	= \frac{1+P(u_0)z}{1+z} .
\]
Consequently,
\[
\abs{\frac{1+P(u_0)z}{1+z}} \ge 1+O(\epsilon) .
\]
Since $0\le P(u_0)\le1$ and $2\Re(z)=|z|^2\ge0$, we have that
\[
\abs{\frac{1+P(u_0)z}{1+z}}^2 = 1 - \frac{|z|^2(1-P(u_0)^2)+2(1-P(u_0))\Re(z)}{|1+z|^2}
	\le 1-\frac{(1-P(u_0))|z|^2}{|1+z|^2} .
\]
Putting together the above inequalities proves the last claim of part (b). Hence the proof of Theorem \ref{partial sums odd} is now complete.
\end{proof}

\newpage

\section{Additional tables}\label{moredata}



\begin{table}[!htbp]
\begin{tabular}{ccccccccccccc}
 & \multicolumn{4}{c}{even} &&\multicolumn{4}{c}{odd} &&\multicolumn{2}{c}{all} \\
\cmidrule{2-5} \cmidrule{7-10} \cmidrule{12-13}
$q$ & min & mean & $.9999$ & max&& min & mean & $.9999$ & max&& mean & $.9999$\\
10000019 & 0.728 & 0.994 & 1.74 & 2.11& & 0.788 & 1.51 & 3.35 & 3.74& & 1.25 & 3.25\\
10000079 & 0.725 & 0.994 & 1.75 & 2.05& & 0.795 & 1.51 & 3.35 & 3.81& & 1.25 & 3.26\\
10000103 & 0.725 & 0.994 & 1.75 & 2& & 0.793 & 1.51 & 3.34 & 3.83& & 1.25 & 3.25\\
10000121 & 0.724 & 0.994 & 1.75 & 2.02& & 0.793 & 1.51 & 3.34 & 3.78& & 1.25 & 3.26\\
10000139 & 0.726 & 0.994 & 1.75 & 2.01& & 0.797 & 1.51 & 3.35 & 3.74& & 1.25 & 3.25\\
10000141 & 0.719 & 0.994 & 1.74 & 2.02& & 0.79 & 1.51 & 3.33 & 3.82& & 1.25 & 3.25\\
10000169 & 0.721 & 0.994 & 1.75 & 2.06& & 0.788 & 1.51 & 3.34 & 3.73& & 1.25 & 3.25\\
10000189 & 0.709 & 0.994 & 1.75 & 2.01& & 0.793 & 1.51 & 3.35 & 3.71& & 1.25 & 3.25\\
10000223 & 0.73 & 0.994 & 1.75 & 2& & 0.783 & 1.51 & 3.34 & 3.81& & 1.25 & 3.25\\
10000229 & 0.723 & 0.994 & 1.74 & 2.04& & 0.784 & 1.51 & 3.33 & 3.79& & 1.25 & 3.25\\
10000247 & 0.716 & 0.994 & 1.74 & 2.05& & 0.794 & 1.51 & 3.34 & 3.71& & 1.25 & 3.25\\
10000253 & 0.724 & 0.994 & 1.75 & 2.05& & 0.783 & 1.51 & 3.34 & 3.75& & 1.25 & 3.25\\
11000027 & 0.724 & 0.994 & 1.75 & 2.03& & 0.797 & 1.51 & 3.34 & 3.72& & 1.25 & 3.25\\
11000053 & 0.733 & 0.994 & 1.75 & 2.07& & 0.781 & 1.51 & 3.34 & 3.81& & 1.25 & 3.25\\
11000057 & 0.707 & 0.994 & 1.75 & 2.03& & 0.79 & 1.51 & 3.34 & 3.74& & 1.25 & 3.25\\
11000081 & 0.724 & 0.994 & 1.75 & 2.01& & 0.789 & 1.51 & 3.33 & 3.77& & 1.25 & 3.25\\
11000083 & 0.724 & 0.994 & 1.75 & 2.05& & 0.799 & 1.51 & 3.34 & 3.84& & 1.25 & 3.25\\
11000089 & 0.728 & 0.994 & 1.75 & 2.05& & 0.794 & 1.51 & 3.33 & 3.77& & 1.25 & 3.26\\
11000111 & 0.724 & 0.994 & 1.75 & 2.03& & 0.796 & 1.51 & 3.33 & 3.72& & 1.25 & 3.26\\
11000113 & 0.719 & 0.994 & 1.75 & 2.01& & 0.781 & 1.51 & 3.34 & 3.73& & 1.25 & 3.25\\
11000149 & 0.731 & 0.994 & 1.75 & 2.03& & 0.805 & 1.51 & 3.34 & 3.72& & 1.25 & 3.25\\
11000159 & 0.722 & 0.994 & 1.75 & 2& & 0.797 & 1.51 & 3.33 & 3.83& & 1.25 & 3.25\\
11000179 & 0.724 & 0.994 & 1.74 & 2.03& & 0.796 & 1.51 & 3.35 & 3.86& & 1.25 & 3.26\\
11000189 & 0.728 & 0.994 & 1.75 & 2.1& & 0.794 & 1.51 & 3.34 & 3.68& & 1.25 & 3.26\\
12000017 & 0.723 & 0.994 & 1.74 & 2.08& & 0.8 & 1.51 & 3.34 & 3.8& & 1.25 & 3.26\\
12000029 & 0.72 & 0.994 & 1.74 & 2.06& & 0.791 & 1.51 & 3.33 & 3.84& & 1.25 & 3.26\\
12000073 & 0.735 & 0.994 & 1.75 & 2.05& & 0.794 & 1.51 & 3.34 & 3.9& & 1.25 & 3.26\\
12000091 & 0.728 & 0.994 & 1.75 & 2.02& & 0.794 & 1.51 & 3.35 & 3.73& & 1.25 & 3.26\\
12000097 & 0.719 & 0.994 & 1.75 & 2.08& & 0.788 & 1.51 & 3.34 & 3.75& & 1.25 & 3.25\\
12000127 & 0.724 & 0.994 & 1.75 & 2.09& & 0.794 & 1.51 & 3.34 & 3.71& & 1.25 & 3.25\\
12000133 & 0.727 & 0.994 & 1.75 & 2.11& & 0.785 & 1.51 & 3.34 & 3.73& & 1.25 & 3.25\\
12000239 & 0.715 & 0.994 & 1.75 & 2.02& & 0.797 & 1.51 & 3.34 & 3.8& & 1.25 & 3.25\\
12000253 & 0.713 & 0.994 & 1.75 & 2.02& & 0.786 & 1.51 & 3.34 & 3.76& & 1.25 & 3.26\\
    \end{tabular}
\caption{The minimum, maximum, mean, and $.9999$-quantile for $m(\chi)$ over even,
    odd, and all nontrivial $\chi$ mod $q$ for some selected values of $q$.}
\label{minmax}
    \end{table}


\bibliographystyle{alpha}

\end{document}